\documentclass{amsart}
\usepackage[utf8]{inputenc}
\usepackage{amsmath}
\usepackage{amsfonts}
\usepackage{amssymb}
\usepackage{amsthm}
\usepackage{amscd}
\usepackage{float}
\usepackage{tikz}
\usepackage{graphicx}
\usepackage[colorlinks=true]{hyperref}
\hypersetup{urlcolor=blue, citecolor=red}
\usepackage{hyperref}
\usepackage{soul}

\newtheorem{theorem}{Theorem}[section]
\newtheorem{corollary}[theorem]{Corollary}

\newtheorem{lemma}[theorem]{Lemma}
\newtheorem{proposition}[theorem]{Proposition}

\newtheorem{assumption}[theorem]{Assumption}

\theoremstyle{definition}

\newtheorem{remark}[theorem]{Remark}

\newcommand{\R}{\mathbb{R}}
\newcommand{\N}{\mathbb{N}}
\newcommand{\C}{\mathbb{C}}
\newcommand{\T}{\mathbb{T}}
\newcommand{\Z}{\mathbb{Z}}
\newcommand{\D}{\mathbb{D}}
\newcommand{\K}{\mathbb{K}}

\begin{document}

\author{Shinya Kinoshita}
\author{Akansha Sanwal}
\author[Robert Schippa]{Robert Schippa*}

\address[S.~Kinoshita]{Department of Mathematics, Tokyo Institute of Technology, Meguro-Ku, Tokyo 152-8551, Japan}
\email{kinoshita@math.titech.ac.jp}

\address[A.~Sanwal]{Universit\"at Innsbruck, Institut f\"ur Mathematik, Technikerstra{\ss}e 13, 6020 Innsbruck,
Austria}
\email{akansha.sanwal@uibk.ac.at}

\address[R.~Schippa]{UC Berkeley, Department of Mathematics, 847 Evans Hall
Berkeley, CA 94720-3840}
\email{rschippa@berkeley.edu}

\title[Improved well-posedness for KP-I equations]{Improved well-posedness for quasilinear and sharp local well-posedness for semilinear KP-I equations}

\begin{abstract}
We show new well-posedness results in anisotropic Sobolev spaces for dispersion-generalized KP-I equations with increased dispersion compared to the KP-I equation. We obtain the sharp dispersion rate, below which generalized KP-I equations on $\R^2$ and on $\R \times \T$ exhibit quasilinear behavior. In the quasilinear regime, we show improved well-posedness results relying on short-time Fourier restriction. In the semilinear regime, we show sharp well-posedness with analytic data-to-solution mapping.  On $\R^2$ we cover the full subcritical range, whereas on $\R \times \T$ the sharp well-posedness is strictly subcritical. 
Nonlinear Loomis-Whitney inequalities are one ingredient. These are presently proved for Borel measures with growth condition reflecting the different geometries of the plane $\R^2$, the cylinder $\R \times \T$, and the torus $\T^2$.
Finally, we point out that on tori $\T^2_\gamma$, KP-I equations are never semilinear.

\end{abstract}

\date{\today}
\thanks{*Corresponding author.}

\maketitle

\section{Introduction}

\subsection{Dispersion-generalized KP-I equations}

In this article, we show new low regularity well-posedness for the dispersion-generalized Kadomtsev Petviashvili-I (KP-I) equations:
\begin{equation}
\label{eq:FKPI}
\left\{ \begin{array}{cl}
\partial_t u - \partial_x D_x^\alpha u - \partial_{x}^{-1} \partial_{y}^2 u &= u \partial_x u, \quad (t,x,y) \in \R \times \D, \\
u(0) &= u_0,
\end{array} \right.
\end{equation}
where $\mathbb{D} = \R \times \mathbb{K}$, $\mathbb{K} \in \{ \R , \T = \R /(2 \pi \Z) \}$, and $\alpha \geqslant 2$. The dual variables of $(t,x,y)$ are denoted by $(\tau,\xi,\eta)$. We define $\partial_x^{-1}$ and $D_x^\alpha$ on $\D$ as Fourier multiplier
\begin{equation*}
\widehat{(\partial_x^{-1} f)} (\xi,\eta) = (i \xi)^{-1} \hat{f}(\xi,\eta), \quad \widehat{(D_x^{\alpha} f)} (\xi,\eta)= |\xi|^{\alpha} \hat{f}(\xi,\eta).
\end{equation*}
By local well-posedness we refer to existence, uniqueness, and continuity of the data-to-solution mapping assigning initial data from suitable Sobolev spaces to continuous curves on time intervals
\begin{equation*}
S_T : H^{s,0}(\D) \to X_T \hookrightarrow C([0,T], H^{s,0}(\D)), \quad u_0 \mapsto u
\end{equation*}
with $T=T(\| u_0 \|_{H^{s,0}})$ depending lower semicontinuously on the norm of the initial data and $T(u) \gtrsim 1$ for $u \downarrow 0$. We refer to the evolution as \emph{semilinear} if the solution mapping is real-analytic. This will be the case when the solution mapping is constructed via Picard iteration as a consequence of real analyticity of the nonlinearity. The evolution is referred to as \emph{quasilinear} if the data-to-solution mapping fails to be real analytic.

\medskip

For $\alpha = 2$, \eqref{eq:FKPI} becomes the original KP-I equation. The KP equations were introduced by Kadomtsev and Petviashvili to model two-dimensional water waves (see \cite{Kadomtsev1970} and \cite{AblowitzSegur1979}) with an emphasis on the transverse stability of solitons of the Korteweg-de Vries (KdV) equation:
\begin{equation*}
\partial_t u + \partial_{x}^3 u = u \partial_x u.
\end{equation*}
The KdV equation is a one-dimensional model of traveling waves in shallow water. Adding the weakly transverse approximation of the wave dispersion relation
\begin{equation*}
\omega_{\Box}(\xi,\eta) = \pm (\xi^2 + \eta^2)^{\frac{1}{2}} = \pm \xi \Big(1 + \frac{\eta^2}{ 2 \xi^2}\Big) + \mathcal{O}\Big(\frac{\eta^4}{\xi^3}\Big), \quad \xi > 0
\end{equation*}
to the KdV equation results in the KP evolution after changing the system of rest and mild dilation:
\begin{equation}
\label{eq:KPEquations}
\partial_t u + \partial_x^3 u + \kappa \partial_{x}^{-1} \partial_{y}^2 u = u \partial_x u, \quad \kappa \in \{1; - 1 \}.
\end{equation}
For $\kappa = -1$ the equation is known as KP-I equation, and for $\kappa=1$ the equation is referred to as KP-II.
Similarly, starting from dispersion-generalized versions of the KdV equation
\begin{equation*}
\partial_t u + \partial_x D_x^\alpha u = u \partial_x u, \quad \alpha \geqslant 2,
\end{equation*}
adding the weakly transverse perturbation leads to the dispersion-generalized KP equations:
\begin{equation*}
\partial_t u + \partial_x D_x^\alpha u + \kappa \partial_x^{-1} \partial_y^2 u = u \partial_x u.
\end{equation*}
In this paper we are only concerned with the analysis of the dispersion-generalized KP-I equations. The dispersion relation in this case is given by
\begin{equation*}
\omega_\alpha(\xi,\eta) = \xi |\xi|^\alpha + \frac{\eta^2}{\xi}.
\end{equation*}
The behavior of solutions to KP-II equations deviates significantly due to a nonlinear defocusing effect absent in the KP-I case, see below.

\medskip

On $\R^2$ we observe that for any solution $u$ to \eqref{eq:FKPI} with initial data $\phi$ the rescalings
\begin{equation*}
u_\lambda(t,x,y) = \lambda^{-\alpha} u(\lambda^{-(\alpha+1)}t, \lambda^{-1} x, \lambda^{-\frac{\alpha+2}{2}} y)
\end{equation*}
solve \eqref{eq:FKPI} with initial data
\begin{equation*}
\phi_\lambda(x,y)= \lambda^{-\alpha} \phi(\lambda^{-1} x, \lambda^{-\frac{\alpha+2}{2}} y).
\end{equation*}
Let $\dot{H}^{s_1,s_2}(\R^2)$ denote the anisotropic homogeneous Sobolev space with norm
\begin{equation*}
\| f \|^2_{\dot{H}^{s_1,s_2}(\R^2)} = \int_{\R^2} |\xi|^{2s_1} |\eta|^{2s_2} |\hat{f}(\xi,\eta)|^2 d\xi d\eta.
\end{equation*}
We compute
\begin{equation*}
\| \phi_\lambda \|_{\dot{H}^{s_1,s_2}(\R^2)} = \lambda^{-\frac{3 \alpha}{4} + 1 - s_1 - \big( \frac{\alpha}{2}+ 1 \big) s_2 } \| \phi \|_{\dot{H}^{s_1,s_2}(\R^2)}.
\end{equation*}
In the scale of the anisotropic Sobolev spaces with $s_2 = 0$, this distinguishes $\dot{H}^{s_c,0}$ with $s_c = 1 - \frac{3 \alpha}{4}$ as scaling critical space. Moreover, $s_2 = 0$ is distinguished as this is the lowest regularity in the $y$-variable, which still respects the Galilean invariance:
\begin{equation*}
\eta \to \eta + A \xi.
\end{equation*}
For $\alpha \in \R$, we have the following conserved quantities for real-valued solutions:
\begin{align}
\label{eq:MassKPI}
M(u)(t) &= \int_{\D} u^2(x,y) dx dy, \\
\label{eq:EnergyKPI}
E_\alpha(u)(t) &= \int_{\D} \big( \frac{1}{2} |D_x^{\frac{\alpha}{2}} u|^2 + \frac{1}{2} |\partial_{x}^{-1} \partial_y u|^2 + \frac{1}{6} u^3 \big) dx dy.
\end{align}
The energy space is given by
\begin{equation*}
E^{\frac{\alpha}{2}}(\D) = \Big\{ \phi \in L^2(\D) : \Big\| \big(1+|\xi|^{\frac{\alpha}{2}} + \frac{|\eta|}{|\xi|} \big) \hat{\phi}(\xi,\eta)\Big \|_{L^2} < \infty \Big\}. 
\end{equation*}
Since the energy spaces are smaller spaces than the anisotropic Sobolev spaces, we opt to work in the anisotropic Sobolev spaces $H^{s,0}$, which are defined by
\begin{equation*}
H^{s,0}(\D) = \{ f \in L^2(\D) : \| \langle \xi \rangle^{s}\hat{f}(\xi,\eta) \|_{L^2_{\xi,\eta}} < \infty \}.
\end{equation*}
Both KP equations \eqref{eq:KPEquations} are known to be completely integrable, as they admit a Lax pair (\cite{Dryuma1974}). Recently, Killip-Vi\c{s}an \cite{KillipVisan2019} introduced the ``method of commuting flows" taking advantage of the complete integrability to obtain sharp global well-posedness in standard Sobolev spaces for one-dimensional models (see also preceding works \cite{KillipVisanZhang2018} and Koch--Tataru \cite{KochTataru2018}). However, to the best of the authors' knowledge, these arguments could not be extended so far to models in higher dimensions, like the KP equations. Moreover, the dispersion-generalized models are not known to be completely integrable. 

\subsection{Nonlinear evolution: Resonance and transversality considerations}

The resonance relation is given by
\begin{equation*}
\begin{split}
\Omega_\alpha(\xi_1,\eta_1,\xi_2,\eta_2) &= \omega_\alpha(\xi_1+\xi_2,\eta_1+\eta_2) - \omega(\xi_1,\eta_1) - \omega(\xi_2,\eta_2) \\
&= \underbrace{(\xi_1+\xi_2) |\xi_1+\xi_2|^\alpha - \xi_1 |\xi_1|^\alpha - \xi_2 |\xi_2|^\alpha}_{\Omega_{\alpha,1}(\xi_1,\xi_2)} - \frac{(\eta_1 \xi_2 - \eta_2 \xi_1)^2}{\xi_1 \xi_2 (\xi_1+\xi_2)}.
\end{split}
\end{equation*}
For KP-I equations we need to consider the possibility
\begin{equation}
\label{eq:ResonantIntroduction}
|\Omega_\alpha(\xi_1,\eta_1,\xi_2,\eta_2)| \ll |\Omega_{\alpha,1}(\xi_1,\xi_2)|,
\end{equation}
and consequently, we cannot recover the derivative loss in a High$\times$Low-interaction (with regard to the $x$-frequencies) through the resonance. Indeed, Molinet--Saut--Tzvetkov \cite{MolinetSautTzvetkov2002Illposedness} showed that the data-to-solution mapping of the KP-I equation on $\R^2$ is not in $C^2$ in any anisotropic Sobolev space. They proved global well-posedness in the second energy space in \cite{MolinetSautTzvetkov2002}. 

\smallskip

Ionescu--Kenig--Tataru \cite{IonescuKenigTataru2008} showed global well-posedness in the first energy space in a seminal contribution, combining Fourier restriction analysis originating from Bourgain's work (see \cite{Bourgain1993KPII} and references therein) and the energy method (see Koch--Tzvetkov \cite{KochTzvetkov2003} for a preceding work combining frequency-de\-pen\-dent time localization and Strichartz estimates). They analyze the nonlinear interaction on frequency-dependent time intervals. Our analysis in the quasilinear case rests likewise on short-time Fourier restriction. We briefly explain our implementation of short-time Fourier restriction on $\R \times \T$ (the choice of frequency-dependent time localization depends on the domain) in Subsection \ref{subsection:ShorttimeFourierRestriction}.

\smallskip

We shall see that in case \eqref{eq:ResonantIntroduction} holds, we have
\begin{equation*}
\Big| \frac{(\eta_1 \xi_2 - \eta_2 \xi_1)^2}{\xi_1 \xi_2 (\xi_1+\xi_2)} \Big| \sim \big| \Omega_{\alpha,1}(\xi_1,\xi_2) \big|,
\end{equation*}
and we obtain favorable bounds for the full transversality. This allows us to obtain trilinear smoothing estimates as consequence of nonlinear Loomis--Whitney inequalities, on which we elaborate in Section \ref{subsection:NonlinearLW}. The trilinear smoothing estimate for the KP-I equation was the second key ingredient in \cite{IonescuKenigTataru2008}.

\smallskip

Further tools to control the nonlinear evolution are linear and bilinear Strichartz estimates, which improve on frequency-dependent time intervals. The basic bilinear Strichartz estimates follow from transversality arguments (see Proposition \ref{prop:BilinearStrichartzGeneral}, Lemma \ref{lem:AlternativeBilinearStrichartzEstimate}),  which are standard by now. We shall see that in case of small transversality we can use a variant of the C\'ordoba--Fefferman square function estimate, which leads to refined bilinear Strichartz estimates (see Lemma \ref{lem:CFBilinearStrichartz}).

\smallskip

We show linear Strichartz estimates using $\ell^2$-decoupling for elliptic hypersurfaces due to Bourgain--Demeter \cite{BourgainDemeter2015}. Increasing the dispersion parameter $\alpha$ on $\R \times \T$, $\ell^2$-decoupling captures a smoothing effect, which is not observed in the fully periodic case.
Increasing $\alpha$ the resonance relation and the nonlinear Loomis-Whitney inequality become likewise more favorable and on $\R^2$ and $\R \times \T$, the evolution becomes semilinear for $\alpha$ large enough. This is described in detail in Subsection \ref{subsection:LWPResults}.

\smallskip

We remark that the preceding description to control the nonlinear evolution does not depend on the underlying geometry $\R^2$ or $\R \times \T$. It is one purpose of the present work to emphasize robust perturbative methods to control nonlinear interactions and to compare the influence of the different geometries on the (multi)linear Strichartz estimates, which are presently described in a unified way.

\medskip

Lastly, we remark on KP-II equations. The resonance relation for KP-II equations is given by
\begin{equation*}
\Omega^{(II)}_\alpha(\xi_1,\eta_1,\xi_2,\eta_2) = \Omega_{\alpha,1}(\xi_1,\xi_2) + \frac{(\eta_1 \xi_2 - \eta_2 \xi_1)^2}{\xi_1 \xi_2 (\xi_1 + \xi_2)}
\end{equation*}
and observing that both terms are of the same sign, it follows that
\begin{equation*}
|\Omega^{(II)}_\alpha(\xi_1,\eta_1,\xi_2,\eta_2)| \gtrsim |\Omega_{\alpha,1}(\xi_1,\xi_2)|.
\end{equation*}
This reflects the aforementioned nonlinear defocusing effect. This was pointed out by Bourgain \cite{Bourgain1993KPII} who showed global well-posedness of the KP-II equation for real-valued solutions in $L^2(\T^2)$ and $L^2(\R^2)$. Global well-posedness and scattering for small initial data in the scaling-critical space $\dot{H}^{-\frac{1}{2},0}(\R^2)$ was proved by Hadac--Herr--Koch \cite{HadacHerrKoch2009}. In the recent work \cite{HerrSchippaTzvetkov2024} Bourgain's $L^2$-well-posedness on the torus is extended to some Sobolev space of negative order. The arguments similarly rely on Strichartz estimates from decoupling and frequency-dependent time localization.

\subsection{Nonlinear Loomis--Whitney inequalities with general measure}
\label{subsection:NonlinearLW}
The nonlinear Loomis-Whitney inequalities are convolution inequalities for functions supported on the  Pontryagin dual
\begin{equation*}
\D^* = (\R \times \K_1 \times \K_2)^* \text{ with } \R^* = \R \text{ and } \T^* = \Z.
\end{equation*}

The literature on Loomis--Whitney inequalities and the related Brascamp--Lieb inequalities is vast, and the following list of references is by no means exhaustive. We refer the interested reader to the references therein (\cite{BennettCarberyWright2005,BejenaruHerrTataru2010,
BennettBez2010,KinoshitaSchippa2021}) for further reading and to the recent overview by Bennett--Bez \cite{BennettBez2021}.
However, we remark that most statements in the literature are local and not immediately suitable for application to PDE since the ``small" support assumptions on the involved functions are not quantified.
 We shall be brief here and refer to Subsection \ref{subsection:NonlinearLWGeneral} for precise notions.

\medskip

Let $(S_i)_{i=1}^3 \subseteq \R^3$ denote $C^{1,\beta}$-hypersurfaces, $\beta >0$, which allow for a global graph parametrization. For $x_i \in S_i$ we denote with $\mathfrak{n}_i(x_i)$ the outer unit normal. We suppose that there is $A \geqslant 1$ such that for any $x_i \in S_i$:
\begin{equation}
\label{eq:Transversality}
A^{-1} \leqslant |\mathfrak{n}_1(x_1) \wedge \mathfrak{n}_2(x_2) \wedge \mathfrak{n}_3(x_3) | \leqslant 1 .
\end{equation}

For the (classical) nonlinear Loomis--Whitney inequality we endow $S_i$ with the surface measure. Let $f_i : S_i \to \R$, $i=1,2,3$ with $S_i \subseteq \R^3$ like above. $S_i$ carries the surface measure $\sigma_i$, and the global convolution estimate proved in \cite{KinoshitaSchippa2021} reads as follows:
\begin{equation*}
\int_{S_3} (f_1 * f_2) f_3 d\sigma_3 \lesssim A^{\frac{1}{2}} \prod_{i=1}^3 \| f_i \|_{L^2(S_i)}.
\end{equation*}

For our purposes when applying the convolution estimates, we shall thicken the hypersurfaces and consider functions $f_i \in L^2(S_i(\varepsilon),d\nu_\gamma)$, which are supported on the $\varepsilon$-neighbourhood $S_i(\varepsilon)$ of $S_i$, where $\nu_\gamma$ is a Borel measure on $\R^3$, which satisfies for any $r>0$ and $x \in \R^3$ the estimate:
\begin{equation}
\label{eq:gammaBorelMeasure}
\nu_\gamma(B(x,r)) \leqslant C_{\nu} r^\gamma \text{ for } 0 < r \leqslant \varepsilon.
\end{equation}

We formulate following version of global nonlinear Loomis-Whitney inequalities, which is related to convolution estimates for functions on product spaces.

\begin{theorem}[Nonlinear~Loomis-Whitney~inequalities~with~general~measure]
\label{thm:GenearlNLW}
Let $\varepsilon>0$, $(S_i)_{i=1,2,3}$ be a collection of hypersurfaces, which satisfy Assumption \ref{assumption:Surfaces}, $\nu_{\gamma}$ be a Borel measure, which satisfies \eqref{eq:gammaBorelMeasure}, and $f_i \in L^2(S_i(\varepsilon),d\nu_\gamma)$, $i=1,2$. Then the following estimate holds:
\begin{equation}
\label{eq:LoomisWhitneyMeasure}
\| f_1 * f_2 \|_{L^2(S_3(\varepsilon),d\nu_\gamma)} \lesssim A^{\frac{1}{2}} \varepsilon^{\frac{\gamma}{2}} \prod_{i=1}^2 \| f_i \|_{L^2(S_i(\varepsilon),d\nu_\gamma)}.
\end{equation}
\end{theorem}
Its proof uses almost orthogonal decompositions like in \cite{KinoshitaSchippa2021}. Interestingly, the constant in the estimate \eqref{eq:LoomisWhitneyMeasure} does not depend on the regularity parameters $b$, $\beta$ of $(S_i)_{i=1,2,3}$ from Assumption \ref{assumption:Surfaces}, but it depends only on the constant $C_\nu$ from \eqref{eq:gammaBorelMeasure}.

\subsection{Improved local well-posedness results}
\label{subsection:LWPResults}
Here we state the new results concerning the well-posedness of KP-I equations. We emphasize that the arguments described above \emph{per se} do not depend on the geometry, but the constants in the linear and multilinear estimates do.

\subsubsection{Euclidean geometry}
On $\R^2$ Z. Guo \emph{et al.} \cite{GuoPengWang2010} applied the argument of Ionescu--Kenig--Tataru \cite{IonescuKenigTataru2008} to show local well-posedness of the KP-I equation in $H^{1,0}(\R^2)$. This is a significantly larger space than the energy space $E^{1}(\R^2)$ considered in \cite{IonescuKenigTataru2008}. The second and third author considered generalized KP equations \eqref{eq:FKPI} for $\alpha \geqslant 2$ in \cite{SanwalSchippa2023}. The analysis combined the nonlinear Loomis-Whitney inequality, bilinear Strichartz estimates and frequency dependent time localization in the quasilinear regime. In the present work we bring improved bilinear Strichartz estimates into the mix recorded in Lemma \ref{lem:CFBilinearStrichartz}, which take advantage of an orthogonality in case of low transversality.
The argument is reminiscent of the \emph{Córdoba--Fefferman square function estimate} (cf. \cite{Fefferman1973,Cordoba1979}). Estimating differences of solutions at negative Sobolev regularity with a suitable low frequency weight further improves the local well-posedness results.

\smallskip

We summarize the presently obtained results refining the arguments from \cite{SanwalSchippa2023}:
\begin{itemize}
\item We obtain the sharp dispersion parameter $\alpha_c = \frac{5}{2}$ such that \eqref{eq:FKPI} evolves in a semilinear way on $\R^2$. We show that the data-to-solution mapping fails to be $C^2$ for $\alpha < \alpha_c$ in Theorem \ref{thm:C2IllposednessR2}.
\item In the semilinear regime $\alpha \geqslant \alpha_c$ we show local well-posedness in the full subcritical range in $H^{s,0}(\R^2)$ for $s>1-\frac{3 \alpha}{4}$ in Theorem \ref{thm:SemilinearWPR2}. This improves \cite[Theorem~1.3]{SanwalSchippa2023}.
\item In the quasilinear case we improve the results in \cite{SanwalSchippa2023} in Theorem \ref{thm:ImprovedQuasilinearLWPFKPIR2}.
\item The present argument shows local well-posedness for the KP-I equation on $\R^2$ in $H^{s,0}(\R^2)$ for $s>\frac{1}{2}$; see Theorem \ref{thm:ImprovedQuasilinearLWPKPIR2}. Shortly before announcing our work, this was independently reported by Guo \cite{Guo2024}.
\end{itemize}

\begin{remark}
We remark that in 04/2024 Guo--Molinet \cite{GuoMolinet2024} reported unconditional\footnote{This means that solutions are constructed in $C_T H^{s,0}$ without intersecting with a smaller function space.} global well-posedness in the energy space and  local well-posedness in $H^{s,0}(\R^2)$ for $s>\frac{3}{4}$. Moreover, shortly before our results were announced in 08/2024, Guo \cite{Guo2024} independently reported local well-posedness of the KP-I equation on Euclidean space for $s>\frac{1}{2}$. Guo uses the same frequency-dependent time localization, but does not employ weighted spaces for the differences of solutions. To compare with our analysis in the dispersion-generalized case, we opted to include the details as well for the KP-I equation in Section \ref{section:LWPFKPIR2}. 
\end{remark}
%

\subsubsection{Partially periodic geometry} 

On $\R \times \T$, T. Robert \cite{Robert2018} showed global well-\-posed\-ness of the KP-I equation in the energy space and as a consequence stability of the line soliton. Here we show improved local well-posedness in the larger anisotropic Sobolev spaces for the dispersion-generalized equations. We remark that the present analysis yields global well-posedness in the energy space. Indeed, with the (improved) Strichartz estimates and Loomis-Whitney inequality for higher dispersion at hand, this is a direct consequence of the analysis in \cite{Robert2018}. We take the following global result for granted:
\begin{theorem}[\cite{Robert2018}]
\label{thm:GWPEnergySpace}
Let $\alpha \geqslant 2$. Then \eqref{eq:FKPI} posed on $\R \times \T$ is globally well-posed in $E^{\frac{\alpha}{2}}(\R \times \T)$.
\end{theorem}
This facilitates our analysis in anisotropic Sobolev spaces because we can work with global solutions from the energy space. The data-to-solution mapping will then be extended to the larger space $H^{s,0}(\D)$.

Firstly, we obtain the sharp dispersion parameter $\alpha_c = 5$ for semilinear well-posedness. For $\alpha < \alpha_c$ we show that \eqref{eq:FKPI} fails to be $C^2$-wellposed in anisotropic Sobolev spaces $H^{s_1,s_2}(\R \times \T)$, and so it is not possible to solve it via the contraction mapping principle. This is recorded in Theorem \ref{thm:C2IllposedCylinder}.

\smallskip

To state our improved local well-posedness results, define
\begin{equation}
\label{eq:RegSolution}
s_1(\alpha,\varepsilon,\R \times \T) = \begin{cases}
 \frac{11}{6} - \frac{2 \alpha}{3} + \varepsilon, &\quad 2 < \alpha \leqslant \frac{11}{4} \text{ and } \varepsilon > 0,\\
  0, &\quad \frac{11}{4} < \alpha \leqslant 5.
  \end{cases}
\end{equation}

\begin{theorem}
\label{thm:WellposednessKPI}
Let $\alpha \in (2,5)$. The Cauchy problem \eqref{eq:FKPI} posed on $\D = \R \times \T$ is locally well-posed in $H^{s,0}(\D)$ for $s \geqslant s_1(\alpha,\varepsilon,\D) $ for $\varepsilon > 0$.
\end{theorem}


\medskip

Arguing that the evolution is semilinear for $\alpha \geqslant \alpha_c$ is more delicate than the example in Theorem \ref{thm:C2IllposedCylinder} and based on combining the nonlinear Loomis-Whitney inequality with bilinear Strichartz estimates. The following arguments should only be understood morally. In the resonant interaction for $N_2 \ll N \sim N_1$ we find for $T \sim 1$:
\begin{equation*}
\begin{split}
&\quad \| P_N \partial_x (S_5(t) P_{N_1} u_0 S_5(t) P_{N_2} u_0) \|_{L_t^1([0,T],L^2_x(\R \times \T))} \\
 &\lesssim N_1 \big( \frac{N_1}{N_2} \big)^{\frac{1}{2}} N_1^{-\frac{5}{2}} \| P_{N_1} u_0 \|_{L^2} \| P_{N_2} u_0 \|_{L^2} \\
&\lesssim N_1^{-1} N_2^{-\frac{1}{2}} \| P_{N_1} u_0 \|_{L^2} \| P_{N_2} u_0 \|_{L^2}.
\end{split}
\end{equation*}
This shows that the estimate is favorable for $N_2 \gtrsim N_1^{-2}$. On the other hand for $N_2 \ll N_1^{-2}$, we find in the resonant case by a bilinear Strichartz estimate
\begin{equation*}
\begin{split}
\| P_N \partial_x (S_5(t) P_{N_1} u_0 S_5(t) P_{N_2} u_0) \|_{L^1([0,T],L^2_x)} &\lesssim N_2^{\frac{1}{2}} N_1 \| P_{N_1} u_0 \|_{L^2} \| P_{N_2} u_0 \|_{L^2} \\
&\lesssim \| P_{N_1} u_0 \|_{L^2} \| P_{N_2} u_0 \|_{L^2}.
\end{split}
\end{equation*}

The above line of argument indicates semilinear global well-posedness of \eqref{eq:FKPI} for $\alpha \geqslant 5$. We obtain local well-posedness sharp up to endpoints. Notably, the sharp regularity for local well-posedness is strictly above the scaling critical regularity:
\begin{theorem}
\label{thm:GWPFKPICylinder}
Let $\alpha \geqslant 5$, $\D = \R \times \T$, and $s^*(\alpha) = \frac{1-\alpha}{4}$. \eqref{eq:FKPI} is locally well-posed for complex valued initial data in $H^{s,0}(\D)$ for $s > s^*(\alpha)$. For real-valued initial data \eqref{eq:FKPI} is globally well-posed in $L^2(\D)$. Moreover, \eqref{eq:FKPI} is ill-posed in $H^{s,0}(\R \times \T)$ for $s<s^*(\alpha)$.
\end{theorem}

\subsection{Short-time Fourier restriction}
\label{subsection:ShorttimeFourierRestriction}
Finally, we sketch the proof of local well-posedness in the quasilinear case via short-time Fourier restriction.  We shall be brief in the following since the argument from \cite{IonescuKenigTataru2008} is standard by now (see also \cite{rsc2019} for a detailed introduction). For definiteness we consider only the partially periodic domain $\R \times \T$. We make use of a function space $F^s(T)$ which measures the Fourier restriction norm after frequency-dependent time localization. This is accompanied by a ``dual'' space $\mathcal{N}^s(T)$, which captures the nonlinearity. A consequence of the frequency-dependent time localization is the necessity to introduce an energy space $E^s(T)$, which takes into account the $C_T H^{s,0}$ norm of dyadic frequency ranges separately.

\medskip

We show the following set of estimates for solutions to \eqref{eq:FKPI} on $\R \times \T$:
\begin{equation}
\label{eq:SolutionEstimatesIntro}
\left\{ \begin{array}{cl}
\| u \|_{F^s(T)} &\lesssim \| u \|_{E^s(T)} + \| \partial_x (u^2) \|_{\mathcal{N}^s(T)}, \\
\| \partial_x( u^2) \|_{\mathcal{N}^s(T)} &\lesssim T^{\delta(\alpha)} \| u \|^2_{F^s(T)}, \\
\| u \|^2_{E^s(T)} &\lesssim \| u_0 \|_{H^{s,0}}^2 + T^\delta \| u \|^3_{F^{s}(T)}
\end{array} \right.
\end{equation}
provided that $s \geqslant s_1(\alpha,\varepsilon,\R \times \T)$ and some $\delta(\alpha) > 0$, $\delta > 0$ for $\alpha \in (2,5)$.
 From \eqref{eq:SolutionEstimatesIntro} and a standard bootstrap argument follow a priori estimates in the same regularity range.
 
 \medskip

Secondly, analyzing the difference of solutions we show Lipschitz continuous dependence of the solutions 
in a space $\bar{H}^{-\frac{1}{2}}$ (with a low frequency weight) depending on solutions of higher regularity $s \geqslant s_1(\alpha,\varepsilon,\R \times \T)$. The corresponding short-time spaces are denoted by $\bar{F}^{-\frac{1}{2}}(T)$, $\bar{\mathcal{N}}^{-\frac{1}{2}}(T)$, and $\bar{B}^{-\frac{1}{2}}(T)$:
\begin{equation*}
\left\{ \begin{array}{cl}
\| v \|_{\bar{F}^{-\frac{1}{2}}(T)} &\lesssim \| v \|_{\bar{B}^{-\frac{1}{2}}(T)} + \| \partial_x (v(u_1+u_2)) \|_{\bar{\mathcal{N}}^{-\frac{1}{2}}(T)}, \\
\| \partial_x (v(u_1+u_2)) \|_{\bar{\mathcal{N}}^{-\frac{1}{2}}(T)} &\lesssim \| v \|_{\bar{F}^{-\frac{1}{2}}(T)} ( \| u_1 \|_{F^s(T)} + \| u_2 \|_{F^s(T)} ), \\
\| v \|^2_{\bar{B}^{-\frac{1}{2}}(T)} &\lesssim \| v(0) \|_{\bar{H}^{-\frac{1}{2}}}^2 + \| v \|_{\bar{F}^{-\frac{1}{2}}(T)}^2 ( \| u_1 \|_{F^s(T)} + \| u_2 \|_{F^s(T)} ).
\end{array} \right.
\end{equation*}

Lipschitz dependence of solutions in $H^{s,0}$ in the $H^{-\frac{1}{2},0}$-topology follows from the above. The proof of continuous dependence in $H^{s,0}$ is then concluded using the Bona--Smith argument.

\medskip

Our choice of frequency dependent time localization is given by interpolation between $T=T(N)=N^{-1}$ for $\alpha =2$ and $T=T(N)=1$ for $\alpha =5$. In the following we denote the linear propagation by
\begin{equation*}
\widehat{S_\alpha(t) u_0} (\xi,\eta) = e^{it (|\xi|^\alpha \xi - \frac{\eta^2}{\xi})} \hat{u}_0(\xi,\eta).
\end{equation*}

 For $\alpha=2$ on the time-scale $T=T(N)=N^{-1}$ the bilinear Strichartz estimate ameliorates the derivative loss in the resonant interaction: Suppose that $N \sim N_1 \gg N_2$. H\"older's inequality gives
\begin{equation*}
\begin{split}
&\quad \| P_N \partial_x (S_2(t) P_{N_1} u_0 S_2(t) P_{N_2} u_0) \|_{L_t^1([0,T],L^2_x(\R \times \T))} \\
&\lesssim T^{\frac{1}{2}} N_1 \| S(t) P_{N_1} u_0 S(t) P_{N_2} u_0 \|_{L^2_{t}([0,T],L^2_x(\R \times \T))}.
\end{split}
\end{equation*}
We will show the short-time bilinear Strichartz estimate in the resonant case for $T \gtrsim N^{-1}$:
\begin{equation*}
\| S_2(t) P_{N_1} u_0 S_2(t) P_{N_2} u_0 \|_{L^2_t([0,T],L^2_x(\R \times \T))} \lesssim (T N_1)^{\frac{1}{2}} N_1^{-\frac{1}{2}} N_2^{\frac{1}{2}} \| P_{N_1} u_0 \|_{L^2} \| P_{N_2} u_0 \|_{L^2},
\end{equation*}
which is based on showing the estimate in the above display for $T=T(N)=N^{-1}$. 

Taking the above estimates together we find
\begin{equation*}
\| P_N \partial_x (S_2(t) P_{N_1} u_0 S_2(t) P_{N_2} u_0) \|_{L_t^1([0,T],L^2_x(\R \times \T))} \lesssim T N_1 N_2^{\frac{1}{2}} \| P_{N_1} u_0 \|_{L^2} \| P_{N_2} u_0 \|_{L^2}.
\end{equation*}

This shows that indeed for $T=T(N) = N^{-1}$ the nonlinear derivative interaction can be controlled for the KP-I equation. We have already indicated above how $\alpha = 5$ can appear as threshold for semilinear local well-posedness, explaining the choice $T=T(N)=N^0$.

\medskip

\emph{Further remarks and thrust of the analysis.}

It can be expected that the presently proved multilinear estimates yield new local well-posedness results for the original KP-I equation on $\R \times \T$ since it appears that due to certain logarithmic divergences, the function spaces require modifications like in Robert's work \cite{Robert2018}. For this reason the analysis is presently not detailed.

\smallskip

Secondly, it is now straight-forward to prove global well-posedness for the disper\-sion-generalized KP-I equations posed on $\T^2_\gamma$ in the first energy space $E^{\frac{\alpha}{2}}(\T^2_\gamma)$:
\begin{equation*}
\left\{ \begin{array}{cl}
\partial_t u + \partial_x D_x^\alpha u + \partial_{x}^{-1} \partial_{y}^2 u &= u \partial_x u, \quad (t,x,y) \in \R \times \T^2_\gamma, \\
u(0) &= u_0 \in E^{\frac{\alpha}{2}}(\T^2_\gamma).
\end{array} \right.
\end{equation*}
The reason is that with the frequency-dependent time localization $T=T(N)=N^{-1}$ we recover the estimates proved by Zhang \cite{Zhang2016}, and for $\alpha > 2$ the energy space becomes smaller than for $\alpha = 2$. This will allow to cover the energy space with quasilinear local well-posedness. Global well-posedness is a consequence of energy conservation.

%
%

\medskip

\emph{Outline of the paper.} 
 In Section \ref{section:Notations} we introduce notations and function spaces to solve the nonlinear equation \eqref{eq:FKPI}.
In Section \ref{section:LinearStrichartz} we show linear Strichartz estimates using $\ell^2$-decoupling, and in Section \ref{section:Transversality} we analyze the interplay between resonance and transversality. In the resonant case we show bilinear Strichartz estimates and show a trilinear estimate based on the nonlinear Loomis--Whitney inequality. Here we moreover compare the nonlinear Loomis--Whitney inequalities on product spaces obtained for dispersion-generalized KP-I equations. 
In Section \ref{section:C2IllPosedness} we show that the data-to-solution mapping for \eqref{eq:FKPI} on $\R^2$ cannot be $C^2$ for $\alpha < \frac{5}{2}$. On $\R \times \T$ the mapping cannot be $C^2$ for $\alpha < 5$ and on tori, it cannot be expected to be $C^2$ for any $\alpha$. 
In Section \ref{section:QuasilinearLWPKPIR2} we show the new local well-posedness result for the KP-I equation on $\R^2$ based on frequency-dependent time localization, (multi)linear Strichartz estimates and the nonlinear Loomis--Whitney inequality. In Section \ref{section:LWPFKPIR2} we show the new well-posedness results for dispersion-generalized KP-I equations on $\R^2$ and in Section \ref{section:LWPFKPICylinder} the corresponding results on $\R \times \T$. Sharp results for the semilinear KP-I equations on $\R^2$ and $\R \times \T$ are proved in Section \ref{section:SemilinearWellposedness}.

\medskip

\textbf{Basic notations:}
\begin{itemize}
\item Time and space variables are denoted by $(t,x,y) \in \R \times \D$, $\D \in \{ \R^2, \R \times \T, \T^2 \}$. The dual variables are denoted by $(\tau,\xi,\eta) \in \R \times \D^*$. 
\item Capital letters $M,N,L,\ldots$ denote dyadic numbers in\\ $2^{\Z} := \{ \ldots ,1/4,1/2,1,2,4,8,\ldots\}$. 
\item We write $N_+ = \max(1,N)$. We also denote $a \vee b:=\max(a,b)$ and $a \wedge b:=\min(a,b)$ for $a,b\in \R$.
\end{itemize}

\section{Notations and Function spaces}
\label{section:Notations}

\subsection{Fourier transform preliminaries}

The (spatial) Fourier transform of $f: \D \to \C$ is denoted by
\begin{equation*}
\hat{f}(\xi,\eta) = (\mathcal{F}_{x,y} f)(\xi,\eta) = \frac{1}{(2\pi)^2} \int_{\D} e^{-i (\xi,\eta)\cdot (x,y)} f(x,y) dx dy.
\end{equation*}
The inverse Fourier transform of $g: \D^* \to \C$ is given by
\begin{equation*}
f(x,y) = (\mathcal{F}^{-1}_{x,y} g)(x,y) = \frac{1}{(2\pi)^2} \int_{\D^*} e^{i (\xi,\eta)\cdot (x,y)} g(\xi,\eta) d\xi d\eta.
\end{equation*}
The space-time Fourier transform of $u:\R \times \D \to \C$ is obtained from extending the definition to $\R \times \D$. We abuse notation and denote it as well as $\hat{u}$. In case $\D^* = \R \times \Z$ we indicate the counting measure by writing $(d \eta)_1$.

\subsection{Function spaces}

Next, we introduce notations and function spaces for the nonlinear analysis to solve \eqref{eq:FKPI}. Recall that the dispersion relation is given by $\omega_{\alpha}(\xi,\eta) = \xi |\xi|^\alpha + \frac{\eta^2}{\xi}$. Let $\zeta: \R \to \R$ denote a smooth cutoff, which is radially decreasing and satisfies $\zeta(\tau) = 1$ for $0\leqslant |\tau| \leqslant 1$, and $\zeta(\tau) = 0$ for $|\tau| \geqslant 2$. We denote $\zeta_{1} = \zeta$ and for $L \in 2^{\N}$ let
$\zeta_L(\tau) = \zeta(\tau / L) - \zeta(\tau/(L/2))$. For $N \in 2^{\Z}$ let $A_N = \{ \xi \in \R : N/4 \leqslant |\xi| \leqslant 4N \}$ denote the $N$-annulus on the real line. We let $A_{\leqslant 1}= \{ \xi \in \R : |\xi| \leqslant 2 \}$.  We shall write $\tilde{A}_N = \R \times A_N \times \K^* \subseteq \R^3$ and similarly for $\tilde{A}_{\leqslant 1}$.  We define
\begin{equation*}
    D_{\alpha, N, L}:= \{ (\tau,\xi,\eta) \in \R \times \D^* : \xi \in A_N, |\tau-\omega_{\alpha}(\xi,\eta)| \leqslant L\}.
\end{equation*}
We shall often omit the subscript $\alpha$ to lighten the notation.

\medskip

For functions $f:\R \times \D^* \to \C$ with $\text{supp}(f) \subseteq \tilde{A}_N$ and $N \in 2^{\Z}$ we define:
\begin{equation*}
\| f \|_{X_N} = \sum_{L \geqslant 1} L^{\frac{1}{2}} \| \eta_L(\tau - \omega_\alpha(\xi)) f \|_{L^2_{\tau,\xi,\eta}}.
\end{equation*}

We choose the frequency-dependent time localization depending on the domains:
\begin{itemize}
\item On $\R^2$ we choose
\begin{equation*}
T=T(N)=N^{-(5-2\alpha)-\varepsilon(\alpha)}.
\end{equation*}
This interpolates between $T=T(N)=N^{-1}$ for the KP-I equation and semilinear local well-posedness for $\alpha = \frac{5}{2}$, which corresponds to $T=T(N)=N^0$. 
\item On $\R \times \T$ we choose similarly
\begin{equation*}
T=T(N)=N^{-\frac{5-\alpha}{3}-\varepsilon(\alpha)}.
\end{equation*}
Since in this case the evolution becomes semilinear for $\alpha \geqslant 5$, the time localization is chosen slightly stronger accordingly.
\end{itemize}
Currently, we cannot abandon the additional frequency dependent time localization for $\alpha > 2$ to cover some endpoint cases in the estimates. To ease notation, we write in the following $\varepsilon$ for $\varepsilon(\alpha)$ for brevity.

\medskip

Let $u: \R \times \D \to \C$ with $\text{supp}(\hat{u}) \subseteq \tilde{A}_N$. Let $N_+ := \max(N,1)=: N \vee 1$. We define for $N \in 2^{\Z}$:
\begin{equation*}
\begin{split}
\| u \|_{F_N} &= \sup_{t_N \in \R} \| \mathcal{F}_{t,x,y}(\chi(N_+^{\frac{5-\alpha}{3}+\varepsilon} (t-t_N)) u) \|_{X_N}, \\
\| u \|_{\mathcal{N}_N} &= \sup_{t_N \in \R} \| (\tau - \omega_\alpha(\xi,\eta) + i N_+^{\frac{5-\alpha}{3}+\varepsilon})^{-1} \mathcal{F}_{t,x,y}(\chi(N_+^{\frac{5-\alpha}{3}+\varepsilon} (t-t_N)) u) \|_{X_N}.
\end{split}
\end{equation*}
These function spaces can be localized for $T \in (0,1]$:
\begin{equation*}
  E_N = \{f : \D \rightarrow \R: \hat{f} \text{ supported in } A_N, \; \|f\|_{E_N} = \| f \|_{L^2} \} 
\end{equation*}
Let $u \in C([0,T],E_N)$. Then we define
\begin{equation*}
\| u \|_{F_N(T)} = \inf_{\substack{ \tilde{u}: \R \times \D \to \C, \\ \tilde{u} \big\vert_{[0,T]} = u }}
\| \tilde{u} \|_{F_N}, \quad \| u \|_{\mathcal{N}_N(T)} = \inf_{\substack{ \tilde{u}: \R \times \D \to \C, \\ \tilde{u} \big\vert_{[0,T]} = u }}
\| \tilde{u} \|_{\mathcal{N}_N}.
\end{equation*}

We assemble the function spaces by Littlewood-Paley decomposition. Let $u \in C([0,T],L^2(\D))$. Then
\begin{equation*}
\| u \|^2_{F^s(T)} = \sum_{N \in 2^{\Z}} N_+^{2s} \|P_N u \|_{F_N(T)}^2, \quad \| u \|^2_{\mathcal{N}^s(T)} = \sum_{N \in 2^{\Z}} N_+^{2s} \| P_N u \|^2_{\mathcal{N}_N(T)}.
\end{equation*}

Moreover, we define
\begin{equation*}
\| u \|_{B^s(T)}^2 = \sum_{N \geqslant 1} N^{2s} \sup_{t \in [0,T]} \|P_N u(t) \|^2_{L^2} + \| P_{\leqslant 1} u(0) \|_{L^2}^2.
\end{equation*}

\medskip

The following function space properties are standard (see e.g. \cite{GuoOh2018}) and will be used in the following. For $K,L \in 2^{\N_0}$, and $f_N \in X_N$, the following estimate holds:
\begin{equation*}
\begin{split}
&\sum_{J \geqslant L} J^{\frac{1}{2}} \Big\| \zeta_J(\tau - \omega_\alpha(\xi,\eta)) \int_{\R} \big| f_N(\tau,\xi,\eta) \big| L^{-1} (1+L^{-1} |\tau - \tau'|^{-4} ) d \tau' \Big\|_{L^2_{\tau,\xi,\eta}} \\
&\quad + L^{\frac{1}{2}} \Big\| \zeta_{\leqslant L}(\tau-\omega_\alpha(\xi,\eta)) \int_{\R} |f_N(\tau',\xi,\eta) | L^{-1} (1 + L^{-1} |\tau - \tau'|)^{-4} d\tau' \Big\|_{L^2_{\tau,\xi,\eta}} \\
&\lesssim \| f_N \|_{X_N}
\end{split}
\end{equation*}
with implicit constant independent of $J$ and $L$.

\smallskip

For $\gamma \in \mathcal{S}(\R)$ and $K,L \in 2^{\N_0}$, and $t_0 \in \R$, it follows
\begin{equation*}
\| \mathcal{F}_{t,x,y}[\gamma(L(t-t_0)) \mathcal{F}^{-1}_{t,x,y}(f_k)] \|_{X_K}  \lesssim \| f_K \|_{X_K}.
\end{equation*}
The implicit constant is independent of $K$, $L$, and $t_0 \in \R$.

\medskip

Next, we introduce admissible time multiplication. Let $N \in 2^{\Z}$, and define the set $S_N$ of $N$-acceptable time multiplication factors as
	\begin{equation*}
		S_N  = \{ m_N: \R \to \R \, : \, \| m_N \|_{S_N} = \sum_{0 \leqslant j \leqslant 30} N_+^{-(\frac{5-\alpha}{3}+\varepsilon)j} \| \partial^j m_N \|_{L^\infty} < \infty \}.
	\end{equation*}
 The following estimates hold for any $s \in \R_{\geqslant 0}$ and $T \in (0,1]$ (cf. \cite[p.~273]{IonescuKenigTataru2008}):
	\begin{equation}\label{eq:TimeMult}
		\left\{ \begin{array}{cl} 
			\| \sum_{N \in 2^{\N_0}} (m_N(t) P_N u ) \|_{F^{s}(T)} &\lesssim \big( \sup\limits_{N \in 2^{\N_0}} \| m_N \|_{S_N} \big) \| u \|_{F^{s}(T)}, \\
			\| \sum_{N \in 2^{\N_0}} (m_N(t) P_N u ) \|_{\mathcal{N}^{s}(T)} &\lesssim \big( \sup\limits_{N \in 2^{\N_0}} \| m_N \|_{S_N} \big) \| u \|_{\mathcal{N}^{s}(T)}, \\
			\| \sum_{N \in 2^{\N_0}} (m_N(t) P_N u ) \|_{E^{s}(T)} &\lesssim \big( \sup\limits_{N \in 2^{\N_0}} \| m_N \|_{S_N} \big) \| u \|_{E^{s}(T)}.
		\end{array} \right.
	\end{equation}
	
	We recall the embedding $F^{s}(T) \hookrightarrow C([-T,T];H^{s,0})$ for short-time $X^{s,b}$-spaces.
	
	\begin{lemma}\label{lemma:Embed}
		Let $s \in \R_{\geqslant 0}$. For all $T \in (0,1]$ and $u\in F^{s}(T)$ we have
		\begin{equation}
			\sup_{t\in [-T,T]}\|u(t)\|_{H^{s,0} }\lesssim \|u\|_{F^{s}(T)}.
		\end{equation}
	\end{lemma}
	For variants of the lemma on different domains we cite \cite{IonescuKenigTataru2008,HerrSanwalSchippa2024}.

\medskip
	
	The following plays the role of an energy estimate in short-time Fourier restriction spaces (cf. \cite[Proposition~2.5.2]{rsc2019}):
	\begin{lemma}\label{lemma:LinShortTime}
		Let $s \in \R_{\geqslant 0}$. For all $ T \in (0,1]$ and (mild) solutions \\ $ u \in C([-T,T],E^\infty(\D))$ of
		\begin{equation*}
			\partial_t u - \partial_{x} D_x^\alpha u - \partial_{x}^{-1} \Delta_y u = \partial_x(u^2) \text{ in } \D_\lambda \times (-T,T),
		\end{equation*}
		the following estimate holds:
		\begin{equation*}
			\|u\|_{F^{s}(T)} \lesssim \|u\|_{E^{s}(T)} + \|\partial_x (u^2) \|_{\mathcal{N}^{s}(T)}.
		\end{equation*}
		
	\end{lemma}

For the large data result we recall how to trade regularity in the modulation variable for powers of the time localization. For $b \in \R$ define
\begin{equation*}
    \| f_N \|_{X^b_N} = \sum_{L \geq 1} \| \eta_L(\tau - \omega_\alpha(\xi,\eta)) f_N(\tau,\xi,\eta) \|_{L^2_{\tau,\xi,\eta}}
\end{equation*}
and the short-time function spaces $F_N^b$, $\mathcal{N}_N^b$, which are defined following along the above, but based on $X^b_N$.

We have the following:
\begin{lemma}{\cite[Lemma~3.4]{GuoOh2018}}
\label{lem:TradingModulationRegularity}
Let $T \in (0,1]$, and $b<\frac{1}{2}$. Then it holds:
\begin{equation*}
    \| P_N u \|_{F_N^b} \lesssim T^{\frac{1}{2}-b-} \| P_N u \|_{F_N}
\end{equation*}
for any function $u: [-T,T] \times \D \to \C$.
\end{lemma}


\section{Linear Strichartz estimates}
\label{section:LinearStrichartz}

\subsection{Linear Strichartz estimates on Euclidean space}

Hadac \cite{Hadac2008} proved linear Strichartz estimates for generalized KP-II equations in \cite[Theorem~3.1]{Hadac2008}. The proof extends to KP-I equations, which leads to the following:
\begin{proposition}
\label{prop:L4StrichartzEstimates}
Let $\alpha \geqslant 2$, $2<q \leqslant \infty$, $\frac{1}{r}+\frac{1}{q}= \frac{1}{2}$, and $s = (1-\frac{2}{r})(\frac{1}{2}-\frac{\alpha}{4})$. Then we have
\begin{equation*}
\|  S_\alpha(t) u_0 \|_{L_t^q(\R;L_{xy}^r(\R^2))} \lesssim \| u_0 \|_{\dot{H}^{s,0}(\R^2)}.
\end{equation*}
For $q=r=4$ we obtain the estimate
\begin{equation*}
\|  S_\alpha(t) u_0 \|_{L_t^4(\R;L_{xy}^4(\R^2))} \lesssim \| u_0 \|_{\dot{H}^{\gamma,0}(\R^2)}
\end{equation*}
with $\gamma = \frac{2-\alpha}{8}$.
\end{proposition}

\subsection{Linear Strichartz estimates on cylinders}
In the following we use $\ell^2$-decoupling due to Bourgain--Demeter \cite{BourgainDemeter2015} for elliptic hypersurfaces to show linear Strichartz estimates on cylinders.
We state the special case of decoupling in $2+1$-dimensions for convenience:
\begin{theorem}[{$\ell^2$-decoupling~for~elliptic~hypersurfaces}]
\label{thm:l2Decoupling}
Let $\Lambda > 0$, $A \subseteq \R^2$ be compact, and $\varphi: A \to \R$ be a $C^2$-function such that
\begin{equation*}
S = \{(\xi,\varphi(\xi)) : \, \xi \in A \}
\end{equation*}
has principal curvatures in $[\Lambda^{-1},\Lambda]$.
Define the Fourier extension operator adapted to $S$ by
\begin{equation*}
\mathcal{E}_S f(x) = \int_A e^{i (x',x_3) \cdot (\xi, \varphi(\xi))} f(\xi) d\xi, \quad x = (x',x_3) \in \R^2 \times \R.
\end{equation*}
Let $R \geqslant 1$. The following estimate holds:
\begin{equation}
\label{eq:EllipticDecoupling}
\| \mathcal{E}_S f \|_{L^4_x(B(0,R))} \lesssim_\varepsilon R^\varepsilon \big( \sum_{\theta: R^{-\frac{1}{2}}-\text{ball}} \| \mathcal{E}_S f_\theta \|_{L^4(w_{B(0,R)})}^2 \big)^{\frac{1}{2}},
\end{equation}
where the sum ranges over a finitely overlapping collection of balls $\theta$ of radius $R^{-\frac{1}{2}}$ covering $A$. The implicit constant in \eqref{eq:EllipticDecoupling} depends on $S$, $\Lambda$, but not on $R$. 
\end{theorem}

We now formulate its consequence for generalized KP-I equations:
\begin{proposition}
\label{prop:StrichartzCylinder}
Let $N \in 2^{\N_0}$, $A \in \R$, and $u_0: \R \times \T \to \C$ with $\text{supp}( \hat{u}_0) \subseteq \{(\xi,\eta) \in \R^2 : |\xi| \sim N, \; \big| \frac{\eta}{\xi} - A \big| \lesssim N^{\frac{\alpha}{2}} \}.$ Then the following estimate holds:
\begin{equation*}
\| P_N S_\alpha(t) u_0 \|_{L_t^4([0,1],L^4_{xy}(\R \times \T))} \lesssim_\varepsilon N^{\frac{2-\alpha}{8}+\varepsilon} \| u_0 \|_{L^2(\R \times \T)}.
\end{equation*}
\end{proposition}
\begin{proof}
First, we consider the case $|A| \lesssim N^{\frac{\alpha}{2}}$, in which case we find $|\eta| \lesssim N^{\frac{\alpha}{2}+1}$.
We use the anisotropic scaling
\begin{equation}
\label{eq:AnisotropicScalingDecoupling}
x \to N x, \quad y \to N^{\frac{\alpha}{2}+1} y, \quad t \to N^{\alpha+1} t, \qquad \xi \to \frac{\xi}{N}, \quad \eta \to \frac{\eta}{ N^{\frac{\alpha}{2}+1}}
\end{equation}
to find\small
\begin{equation*}
\begin{split}
&\quad \| P_N S_\alpha(t) u_0 \|^4_{L^4_{t,x,y}([0,1], \R \times \T))} = N^{-(\alpha+1)} N^{-4(1-\frac{1}{4})} N^{-(\frac{\alpha}{2}+1)} \\
&\quad \quad \Big\| \int_* e^{i(x' \xi' + y' \eta' + t' \omega_{\alpha}(\xi',\eta'))} \hat{u}_0(N \xi', N^{\frac{\alpha}{2}+1} \eta') d\xi' d\eta' \Big\|_{L^4_{t'} ([0,N^{\alpha+1}],L^4_{x'y'}(\R \times \R / (2 \pi N^{\frac{\alpha}{2}+1} \Z))}^4.
\end{split}
\end{equation*}
\normalsize
Note that in the rescaled support of $\hat{u}_0$ it holds $|\xi'| \sim 1$ and $|\eta'| \lesssim 1$. 

\smallskip

Now we enlarge the $y$-domain of integration $\R / (2 \pi N^{\frac{\alpha}{2}+1} \Z)$ perceived as an interval in $\R$ to an interval of size $N^{\alpha+1}$. By spatial periodicity this amounts to a factor of $(N^{\alpha+1} / N^{\frac{\alpha}{2}+1} )^{-1} = N^{- \alpha/2}$, and we continue the above as \small
\begin{equation}
\label{eq:RescaledDecoupling}
\begin{split}
&\; N^{-(\alpha+1)} N^{-3} N^{-\big( \frac{\alpha}{2}+1 \big)} N^{- \frac{\alpha}{2}} \\
&\times \Big\| \int_* e^{i(x' \xi'+ y' \eta' + t'(|\xi'|^\alpha \xi' + (\eta')^2/\xi'))} \hat{u}_0(N \xi', N^{\frac{\alpha}{2}+1} \eta') d\xi' d\eta' \Big\|^4_{L^4_{t'}([0,N^{\alpha+1}],L^4_{x' y'}(\R \times [0,N^{\alpha+1}]))}.
\end{split}
\end{equation}
\normalsize
We invoke $\ell^2$-decoupling for elliptic surfaces in two dimensions. Here we note that the Hessian is for $\xi' \neq 0$ given by
\begin{equation*}
\partial^2 \omega_\alpha(\xi',\eta') =
\begin{pmatrix}
\alpha (\alpha + 1) \text{sgn}(\xi') |\xi'|^{\alpha-1} + 2 \frac{(\eta')^2}{(\xi')^3} & - \frac{2 \eta'}{(\xi')^2} \\
- \frac{2 \eta'}{(\xi')^2} & \frac{2}{\xi'}
\end{pmatrix}
.
\end{equation*}
It is easy to see that $\partial^2 \omega_\alpha$ is uniformly elliptic for $|\xi'| \sim 1$ and $|\eta'| \lesssim 1$. Indeed, we compute
\begin{equation*}
\det \partial^2 \omega_\alpha(\xi',\eta') = \alpha (\alpha-1) |\xi'|^{\alpha-2}, \quad |\text{tr} (\partial^2 \omega_\alpha)| \sim 1.
\end{equation*}
We infer for $\xi' \sim 1$ two positive eigenvalues of size comparable to $1$, and for $\xi' \sim -1$ two negative eigenvalues of size comparable to $-1$. The latter case can be reduced to the elliptic case by time reversal.

Few remarks about decoupling are in order: To apply decoupling as formulated in Theorem \ref{thm:l2Decoupling}, we break the domain of integration $[0,N^{\alpha+1}] \times \R \times [0,N^{\alpha+1}]$ into finitely overlapping balls of size $N^{\alpha+1}$. The shift in $x$ is admissible by translation invariance.
Secondly, decoupling requires a continuous approximation, i.e., approximating the exponential sum with a Fourier extension operator by mollifying the Dirac comb. 

We give the details: Firstly, we break the $x$-integration into balls of size $N^{\alpha+1}$.
\small
\begin{equation*}
\begin{split}
&\quad \big\| \int e^{i(x' \xi' + y' \eta' + t' \omega_\alpha(\xi',\eta'))} \hat{u}_0(N \xi', N^{\frac{\alpha}{2}+1} \eta') d\xi (d \eta')_{N^{-\frac{\alpha}{2}-1}} \big\|_{L^4_{t',x',y'}([0,N^{\alpha+1}] \times \R \times [0,N^{\alpha+1}])}^4 \\
&\lesssim \sum_{B_{N^{\frac{\alpha}{2}+1}}} \big\| \int e^{i(x' \xi' + y' \eta' + t' \omega_\alpha(\xi',\eta'))} \hat{u}_{0N} d\xi (d \eta')_{N^{-\frac{\alpha}{2}-1}} \big\|_{L^4_{t',x',y'}([0,N^{\alpha+1}] \times B_{N^{\alpha+1}} \times [0,N^{\alpha+1}])}^4.
\end{split}
\end{equation*}
\normalsize
Next, by a linear change of variables $x' \to x' + c(B_{N^{\frac{\alpha}{2}+1}})$ with $c(B)$ denoting the center of the ball, we can suppose that the ball is centered at the origin. The phase factor is absorbed into $\hat{u}_{0N}$.

For the continuous approximation, we consider a sequence of smooth functions $\hat{f}_\lambda$ such that for $|x'|,|t'|,|y'| \lesssim N^{\alpha+1}$ we have
\small
\begin{equation*}
\int e^{i(x' \xi' + y' \eta' + t' \omega_\alpha(\xi',\eta'))} \hat{f}_\lambda(\xi',\eta') d\xi' d\eta' \to \int e^{i(x' \xi' + y' \eta' + t' \omega_\alpha(\xi',\eta'))} \hat{u}_{0N} d\xi' (d\eta')_{N^{-(\frac{\alpha}{2}+1)}}.
\end{equation*}
\normalsize
More specifically, we can choose $\text{supp}(\hat{f}_\lambda) \subseteq \mathcal{N}_{\lambda^{-1}}(\text{supp}(\hat{u}_{0N}))$ in a $\lambda^{-1}$-neigh\-bour\-hood. Then, choosing $\lambda$ large enough, we can estimate by the theorem of dominated convergence:
\begin{equation*}
\begin{split}
&\quad \big\| \int e^{i(x' \xi' + y' \eta' + t' \omega_\alpha(\xi',\eta'))} \hat{u}_{0N} d\xi (d \eta')_{N^{-\frac{\alpha}{2}-1}} \big\|_{L^4_{t',x',y'}([0,N^{\alpha+1}] \times B_{N^{\frac{\alpha}{2}+1}} \times \T_{N^{\frac{\alpha}{2}+1}})} \\
&\lesssim \big\| \int e^{i(x' \xi' + y' \eta' + t' \omega_\alpha(\xi',\eta'))} \hat{f}_\lambda(\xi',\eta') d\xi' d\eta'  \big\|_{L^4_{t',x',y'}(B_{N^{\alpha+1}})}.
\end{split}
\end{equation*}
The final expression is amenable to decoupling as recorded and after reversing the continuous approximation by letting $\lambda \to \infty$, we obtain an $\ell^2$-sum into balls on the frequency side of size $N^{- \frac{\alpha+1}{2}}$. On this scale the phase function can be trivialized by Taylor expansion and we obtain from integration in time after reversing the continuous approximation:
\small
\begin{equation*}
\begin{split}
&\quad \big\| \int_* e^{i(x' \xi' + y' \eta' + t'(|\xi'|^\alpha \xi' + (\eta')^2/\xi'))} \hat{u}_{0 \theta}(N \xi', N^{\frac{\alpha}{2}+1} \eta') d\xi' d\eta' \big\|^2_{L_{t'}^4([0,N^{\alpha+1}],L^4_{x' y'}(B_{N^{\alpha+1}})} \\
&\lesssim N^{\frac{\alpha+1}{2}} \big\| \int_* e^{i(x' \xi' + y' \eta')} \hat{u}_{0 \theta}(N \xi', N^{\frac{\alpha}{2}+1} \eta') d\xi' d\eta' \big\|^2_{L^4_{x' y'}(B_{{N^{\alpha+1}}})}.
\end{split}
\end{equation*}
\normalsize
Now we reverse the scaling in $x$ and $y$, which compensates the scaling factors $N^{-3}$ and $N^{-(\frac{\alpha}{2}+1)}$ in \eqref{eq:RescaledDecoupling}.

It remains to estimate
\begin{equation*}
\big\| \int_* e^{i(x \xi + y \eta)} \hat{u}_{0 \theta}(\xi,\eta) d\xi d\eta \big\|^2_{L^4_{xy}}
\end{equation*}
with $\theta$ denoting a rectangle of size $N^{\frac{1-\alpha}{2}} \times N^{\frac{1}{2}}$. We can use Bernstein's inequality to find
\begin{equation}
\label{eq:BernsteinsInequalityStrichartz}
\big\| \int_* e^{i(x \xi + y \eta)} \hat{u}_{0 \theta}(\xi,\eta) d\xi d\eta \big\|_{L^4_{xy}} \lesssim N^{\frac{1-\alpha}{8}} N^{\frac{1}{8}} \| u_{0 \theta} \|_{L^2_{xy}}.
\end{equation}

We summarize the estimates as follows: Scaling to unit frequencies and enlarging the spatial domain to $N^{\alpha+1}$ incurs a factor of
\begin{equation*}
N^{-(\alpha+1)} N^{-3} N^{-(\frac{\alpha}{2} + 1)} N^{- \frac{\alpha}{2}}.
\end{equation*}
 
Invoking $\ell^2$-decoupling incurs a factor of $N^\varepsilon$ and decouples the integral into balls of size $N^{-\frac{\alpha+1}{2}}$. Carrying out the time integral yields a factor $N^{\alpha+1}$ and reversing the scaling in $x$ and $\xi$, and $y$ and $\eta$ gives factors $N^3$ and $N^{\frac{\alpha}{2}+1}$, respectively. Moreover, we use periodicity to decrease the spatial domain to $\T$ again after inverting the scaling. This incurs a factor $N^{\frac{\alpha}{2}}$.

We have proved so far
\begin{equation}
\label{eq:DecouplingI}
\| P_N S_\alpha(t) u_0 \|^4_{L_{t,x,y}^4([0,1], \R \times \T)} \lesssim_\varepsilon N^{\varepsilon} \big( \sum_{\theta} \| u_{0 \theta} \|^2_{L^4_{xy}} \big)^{2}
\end{equation}
with the sum over $\theta$ being over essentially disjoint rectangles of size $N^{\frac{1-\alpha}{2}} \times N^{\frac{1}{2}}$.
Plugging \eqref{eq:BernsteinsInequalityStrichartz} into \eqref{eq:DecouplingI} yields
\begin{equation*}
\| P_N S_\alpha(t) u_0 \|^4_{L^4_{t,x,y}} \lesssim_\varepsilon N^{\varepsilon + \frac{2-\alpha}{2}} \| u_0 \|_{L^2_{xy}}^4,
\end{equation*}
which completes the proof in case $|A| \lesssim N^{\frac{\alpha}{2}}$.

\medskip

Now we turn to the case $|A| \gg N^{\frac{\alpha}{2}}$: In the first step we apply the scaling \eqref{eq:AnisotropicScalingDecoupling}, which maps the Fourier support to the set
\begin{equation*}
|\xi'| \sim 1, \quad \big| \frac{\eta'}{\xi'} - A' \big| \lesssim 1
\end{equation*}
for $A' = A/N^{\frac{\alpha}{2}}$. Like above, by periodic extension in $y'$, we reduce to estimate the expression \small
\begin{equation*}
\big\| \int e^{i(x'\xi' + y' \eta' + t' \omega_{\alpha}(\xi',\eta'))} \hat{u}_0(N \xi', N^{\frac{\alpha}{2}+1} \eta') d\xi' (d\eta')_{N^{-\frac{\alpha}{2}-1}} \big\|_{L^4_{t'}([0,N^{\alpha+1}], L^4_{x' y'}(\R \times B_{N^{\alpha+1}}))}.
\end{equation*}

\normalsize
We can apply the Galilean invariance 
\begin{equation*}
(\xi'',\eta'') = F'(\xi',\eta') = (\xi',\eta' - A' \xi')
\end{equation*}
to obtain
\begin{equation*}
|\xi''| \sim 1, \quad \big| \frac{\eta''}{\xi''} \big| \lesssim 1.
\end{equation*}
Note that this does not change the domain of integration $\R \times B_{N^{\alpha+1}}$, but it is important to note that in general $(\xi'',\eta'') \notin \R \times \Z$. Now we again break the domain of integration $[0,N^{\alpha+1}] \times \R \times B_{N^{\alpha+1}}$ into balls of size $B_{N^{\alpha+1}}$ and use translation invariance.

\smallskip

We have by continuous approximation
\small
\begin{equation*}
\begin{split}
&\quad \big\| \int d \xi'' \sum_{\eta'' \in \Z - A' \xi''} e^{i(x' \xi''+y' \eta'' + t' \omega_\alpha(\xi'',\eta''))} \hat{u}_{0N}(\xi'',\eta''+ A' \xi') \big\|_{L^4_{t'}([0,N^{\alpha+1}], L^4_{x' y'}(B_{N^{\alpha+1}}))} \\
&\lesssim \big\| \int_{|\xi''| \sim 1, \, |\eta''| \lesssim 1} e^{i(x' \xi'' + y' \eta'' + t' \omega_\alpha(\xi'',\eta''))} \hat{f}_\lambda(\xi'',\eta'') d\xi'' d\eta'' \big\|_{L^4_{t',x',y'}(B_{N^{\alpha+1}})}.
\end{split}
\end{equation*}
\normalsize
Now we can apply $\ell^2$-decoupling and reverse the continuous approximation, which gives the estimate
\small
\begin{equation*}
\begin{split}
&\quad \big\| \int d \xi'' \sum_{\eta'' \in \Z - A' \xi''} e^{i(x' \xi'' + y' \eta'' + t' \omega_\alpha(\xi'',\eta'')} \hat{u}_{0,N}(\xi'',\eta'' + A' \xi'') \big\|_{L^4_{t',x',y'}(B_{N^{\alpha+1}})} \\
&\lesssim_\varepsilon N^{\varepsilon} \big( \sum_{\theta \in \mathcal{B}_{N^{-\frac{\alpha+1}{2}}}} \big\| \int d\xi'' \sum_{\eta'' \in \Z - A' \xi''} e^{i(x' \xi'' + y' \eta'' + t' \omega_\alpha(\xi'',\eta''))}\\
&\quad \quad \times \chi_{\theta}(\xi'',\eta'') \hat{u}_{0N}(\xi'',\eta''+A'\xi'') \big\|_{L^4_{t',x',y'}(w_{B_{N^{\alpha+1}}})}^2 \big)^{\frac{1}{2}}.
\end{split}
\end{equation*}
\normalsize
The sum over $\theta$ is over essentially disjoint $N^{-\frac{\alpha+1}{2}}$-balls, which cover the set $\{|\xi''| \sim 1, |\eta''| \lesssim 1 \}$.

Summing over the $B_{N^{\alpha+1}}$-balls in the $x$-coordinate and reversing all linear transformations and the scalings, we arrive at the expression
\small
\begin{equation*}
\begin{split}
&\quad \big\| \int d\xi \sum_{\eta \in \Z} e^{i(x\xi + y \eta + t \omega_\alpha(\xi,\eta))} \hat{u}_0(\xi,\eta) \big\|_{L_t^4([0,1], L^4_{xy}(\R \times \T))} \\
&\lesssim_\varepsilon N^\varepsilon \big( \sum_{\theta \in \mathcal{B}_{N^{-\frac{\alpha+1}{2}}}} \big\| \int d\xi \sum_{\eta \in \Z} e^{i(x \xi + y \eta + t \omega_\alpha(\xi,\eta))} \\
&\quad \quad \quad \times \chi_{\theta}(N^{-1} \xi, N^{-(\frac{\alpha}{2}+1)}(\eta - A \xi)) \hat{u}_0(\xi,\eta) \big\|^2_{L_t^4(w_1,L^4_{xy}(\R \times \T))}.
\end{split}
\end{equation*}
\normalsize

Since there are no oscillations of $\omega_\alpha(\xi,\eta)$ for $(N^{-1} \xi, N^{-(\frac{\alpha}{2}+1)}(\eta-A\xi))$ contained in $\theta$, we can carry out the integration in time without loss. 

\smallskip

Let $\bar{\theta} = I_{N^{\frac{1-\alpha}{2}}} \times I_{N^{\frac{1}{2}}}$ be the anisotropic dilation of $\theta$. Let
\begin{equation*}
(\bar{\xi},\bar{\eta}) = F(\xi,\eta) = (\xi,\eta - A \xi)
\end{equation*}
denote the Galilean transform without anisotropic dilation.

To conclude the argument, we apply Bernstein's inequality for which we need to understand $F^{-1}(\bar{\theta})$. Since $\xi \in I_{N^{\frac{1-\alpha}{2}}}$ the condition $\eta - A \xi \in I_{N^{\frac{1}{2}}}$ yields the bound 
\begin{equation*}
\# \{ \eta \in \Z \; | \; \exists \xi \in I_{N^{\frac{1}{2}}}: \eta - A \xi \in I_{N^{\frac{1}{2}}} \} \lesssim N^{\frac{1}{2}} + A N^{\frac{1-\alpha}{2}} \sim A N^{\frac{1-\alpha}{2}}.
\end{equation*}
The final estimate is due to the size assumption on $A$. For fixed $\eta$, we estimate the measure of $\xi$ such that $F(\xi,\eta) \in \bar{\theta}$ as
\begin{equation*}
\text{meas}( \{ \xi \in \R : F(\xi,\eta) \in \bar{\theta} \}) \lesssim N^{\frac{1}{2}} / A.
\end{equation*}
Then it is a consequence of Bernstein's inequality that
\begin{equation*}
\begin{split}
&\quad \big\| \int_{F(\xi,\eta) \in \bar{\theta}} e^{i(x \xi + y \eta)} \hat{f}(\xi,\eta) d\xi d\eta \big\|_{L^4_{xy}(\R \times \T)} \\
&\lesssim (A N^{\frac{1-\alpha}{2}})^{\frac{1}{4}} (N^{\frac{1}{2}} / A)^{\frac{1}{4}} \| \int_{(\xi,\eta) \in F^{-1}(\bar{\theta})} e^{i(x\xi + y \eta)} \hat{u}_0(\xi,\eta) \|_{L^2_{xy}} \\
&\lesssim N^{\frac{2-\alpha}{8}} \| \int_{(\xi,\eta) \in F^{-1}(\bar{\theta})} e^{i(x\xi + y \eta)} \hat{u}_0(\xi,\eta) \|_{L^2_{xy}}.
\end{split}
\end{equation*}
The proof is then concluded by the essential disjointness of $F^{-1}(\bar{\theta})$.
\end{proof}

\begin{remark}
We remark that the argument does not yield an improved estimate when considering frequency-dependent time intervals $T=T(N)=N^{-\kappa}$. Moreover, a comparison of the estimate on cylinders with the Strichartz estimates on Euclidean space indicates sharpness of the derivative loss up to endpoints.
\end{remark}

\subsection{Linear Strichartz estimates on tori}

The above arguments yield the following result on tori:
\begin{proposition}
Let $N \in 2^{\N_0}$, $\alpha \geqslant 2$, and $\text{supp}(\hat{u}_0) \subseteq \{(\xi,\eta) \in \R^2 : |\xi| \sim N, \quad |\eta| \lesssim N^{\frac{\alpha}{2}+1} \} = A_{N,N^{\frac{\alpha}{2}+1}}$. Then the following estimate holds:
\begin{equation}
\label{eq:LinearStrichartzTorus}
\| S_\alpha(t) u_0 \|_{L_t^4([0,1],L^4_{xy}(\T^2_{\gamma}))} \lesssim_\varepsilon N^{\frac{1}{8}+\varepsilon} \| u_0 \|_{L^2_{xy}(\T^2_\gamma)}.
\end{equation}
\end{proposition}
\begin{proof}
We can follow along the arguments of the proof of Proposition \ref{prop:L4StrichartzEstimates}, i.e., employ the anisotropic scaling \eqref{eq:AnisotropicScalingDecoupling}, use a continuous approximation to invoke $\ell^2$-decoupling, and finally reverse the continuous approximation and scaling. After carrying out the integration in time, we arrive at the expression:
\begin{equation*}
\| S_\alpha(t) u_0 \|_{L_t^4([0,1],L^4_{xy}(\T^2_\gamma))} \lesssim_\varepsilon N^\varepsilon \big( \sum_{\theta \in \mathcal{R}_{N^{\frac{1-\alpha}{2}} \times N^{\frac{1}{2}}}} \| u_{0 \theta} \|^2_{L^4_{xy}(\T^2_\gamma)} \big)^{\frac{1}{2}},
\end{equation*}
where the sum is carried out over essentially disjoint rectangles $\theta$ of size $N^{\frac{1-\alpha}{2}} \times N^{\frac{1}{2}}$, which cover $A_{N,N^{\frac{\alpha}{2}+1}}$. An application of Bernstein's inequality yields
\begin{equation*}
\| u_{0 \theta} \|_{L^4_{xy}(\T^2_\gamma)} \lesssim N^{\frac{1}{8}} \| u_{0 \theta} \|_{L^2 (\T^2_\gamma)},
\end{equation*}
since presently the estimate is carried out with counting measure compared to the previous section. The proof is concluded by essential disjointness of $\theta$.
\end{proof}

\begin{remark}[Sharpness~of~the~Strichartz~estimate]
The Strichartz estimate is sharp, which can be seen from considering the initial data
\begin{equation*}
\hat{u}_0(\xi,\eta) = \delta_{\xi,N} 1_{[1,N^{\frac{1}{2}}]}(\eta), \quad \delta_{N}(\xi) 1_{[1,N^{\frac{1}{2}}]}(\eta).
\end{equation*}
In this case, there are no oscillations, for which reason
\begin{equation*}
\| S_\alpha(t) u_0 \|_{L_t^4([0,1],L^4_{xy}(\T^2_{\gamma})} \sim \| u_0 \|_{L^4_{xy}(\T^2_\gamma)}.
\end{equation*}
Next, we observe by $u_0(0) = N^{\frac{1}{2}}$ and the uncertainty principle that
\begin{equation*}
\| u_0 \|_{L^4_{xy}(\T^2_\gamma)} \gtrsim N^{\frac{1}{2}} (N^{-\frac{1}{2}})^{\frac{1}{4}} \sim N^{\frac{3}{8}}.
\end{equation*}
Since $\| u_0 \|_{L^2_{xy}} \sim N^{\frac{1}{4}}$, this example exhausts \eqref{eq:LinearStrichartzTorus} up to the endpoint.

This points out that for the KP-I dispersion relation the Strichartz estimates on tori deviate significantly from the Strichartz estimates on Euclidean space, which is not the case for the Schr\"odinger equation. For the Schr\"odinger evolution, $\ell^2$-decoupling recovers Euclidean Strichartz estimates on finite times up to arbitrarily small derivative loss (see \cite{BourgainDemeter2015}).

%

\end{remark}

\section{Resonance, Transversality, and multilinear estimates}
\label{section:Transversality}

In this section we show bilinear Strichartz estimates and trilinear convolution estimates for approximate solutions to dispersion-generalized KP-I equations.

\subsection{Resonance and bilinear Strichartz estimates}
We recall resonance and transversality identities from \cite{SanwalSchippa2023}. The resonance function is given by
\begin{equation*}
\begin{split}
\Omega_{\alpha}(\xi_1,\xi_2,\eta_1,\eta_2) &= \omega_\alpha(\xi_1+\xi_2,\eta_1+\eta_2) - \omega_\alpha(\xi_1,\eta_1) - \omega_\alpha(\xi_2,\eta_2) \\
&= \underbrace{|\xi_1+\xi_2|^\alpha (\xi_1+\xi_2) - |\xi_1|^\alpha \xi_1 - |\xi_2|^\alpha \xi_2}_{\Omega_{\alpha,1}} - \frac{(\eta_1 \xi_2 - \eta_2 \xi_1)^2}{\xi_1 \xi_2 (\xi_1+\xi_2)}.
\end{split}
\end{equation*}
We say that we are in the \emph{resonant case}, if
\begin{equation}
\label{eq:ResonanceCondition}
|\Omega_\alpha| \ll \big| \Omega_{\alpha,1} \big|. 
\end{equation}
Suppose that $|\xi_1+\xi_2| \sim |\xi_1| \sim N_{\max} \gtrsim N_{\min} \sim |\xi_2|$: Applying the mean-value theorem gives
\begin{equation*}
|\Omega_{\alpha,1}| \sim N_{\max}^\alpha N_{\min}.
\end{equation*}
We see that in the non-resonant case the derivative loss is readily recovered, e.g., using standard Fourier restriction spaces (cf. \cite{Bourgain1993B,Bourgain1993KPII}) because $|\Omega_{\alpha,1}|^{\frac{1}{2}} \gtrsim N^{\frac{\alpha}{2}}_{\max}  N_{\min}^{\frac{1}{2}}$.

\medskip

We focus on the resonant case. Here we obtain a bound for the transversality quantified by the difference of the group velocity:
\begin{equation*}
\nabla \omega_{\alpha}(\xi,\eta) = ( (\alpha+1) |\xi|^\alpha - \frac{\eta^2}{\xi^2}, \frac{2 \eta}{\xi} ).
\end{equation*}
It follows
\begin{equation}
\label{eq:BoundGroupVelocityResonantCase}
|\partial_\eta \omega_{\alpha}(\xi_1,\eta_1) - \partial_\eta \omega_\alpha(\xi_2,\eta_2) | \gtrsim \big| \frac{\eta_1}{\xi_1} - \frac{\eta_2}{\xi_2} \big| \sim N_{\max}^{\frac{\alpha}{2}}.
\end{equation}
This is a key ingredient for the bilinear Strichartz estimates in the resonant case. 
Define
\begin{equation}
\label{eq:ConstantBilinearStrichartz}
C_{1,\D}(L,N) = \begin{cases}
(L / N^{\frac{\alpha}{2}} )^{\frac{1}{2}}, \; &\D = \R^2, \\
\langle L / N^{\frac{\alpha}{2}} \rangle^{\frac{1}{2}}, \; &\D = \R \times \T.
\end{cases}
\end{equation}

For later use we show the following more general bilinear Strichartz estimate:
\begin{proposition}
\label{prop:BilinearStrichartzGeneral}
Let $\D \in \{ \R^2, \R \times \T \}$. Let $f_i: \R \times \D^* \to \C$, $i=1,2$ satisfy the support conditions $\pi_{\xi} (\text{supp}(f_j)) \subseteq I_j$ with $|I_j| \lesssim N_{\min}$ and
\begin{equation}
\label{eq:GroupVelocityGeneral}
|\partial_{\eta} \omega_\alpha(\xi_1,\eta_1) - \partial_\eta \omega_\alpha(\xi - \xi_1, \eta - \eta_1)| \gtrsim N_{\max}^{\frac{\alpha}{2}}
\end{equation}
for $(\xi_1,\eta_1) \in \pi_{\xi,\eta}(\text{supp}(f_1))$, $(\xi-\xi_1,\eta-\eta_1) \in \pi_{\xi,\eta}(\text{supp}(f_2))$, and $|\tau_i - \omega_\alpha(\xi_i,\eta_i) | \lesssim L_i$ for $(\tau_i,\xi_i,\eta_i) \in \text{supp}(f_i)$. Then the following estimate holds:
\begin{equation}
\label{eq:BilinearStrichartzGeneral}
\| f_{1,N_1,L_1} * f_{2,N_2,L_2} \|_{L^2_{\tau,\xi,\eta}} \lesssim N_{\min}^{\frac{1}{2}} L_{\min}^{\frac{1}{2}} C_{1,\D}(L_{\max},N_{\max}) \prod_{i=1}^2 \| f_{i,N_i,L_i} \|_{L^2_{\tau,\xi,\eta}}.
\end{equation}
\end{proposition}
\begin{proof}
Suppose that $L_1 = \min(L_1,L_2)$ and $L_2 = \max(L_1,L_2)$ by symmetry. We obtain by the Cauchy-Schwarz inequality:
\begin{equation*}
\begin{split}
&\quad \| (f_{1,N_1,L_1} * f_{2,N_2,L_2}) \|_{L^2_{\tau,\xi,\eta}} \\
&\leqslant \sup_{(\tau,\xi,\eta) \in \R^3} \text{meas}( S= \{ (\tau_1,\xi_1,\eta_1) \in \text{supp}(f_{1,N_1,L_1}), \; \xi_1 \in I_1 :  \\
&\quad \quad (\tau-\tau_1,\xi-\xi_1,\eta-\eta_1) \in \text{supp}(f_{2,N_2,L_2}), \; \xi - \xi_1 \in I_2 \} )^{\frac{1}{2}} \prod_{i=1}^2 \|f_{i,N_i,L_i} \|_{L_{\tau,\xi,\eta}^2}.
\end{split}
\end{equation*}
The assumption \eqref{eq:GroupVelocityGeneral} yields
\begin{equation*}
\big| \frac{\eta_1}{\xi_1} - \frac{\eta - \eta_1}{\xi - \xi_1} \big| \gtrsim N_{\max}^{\frac{\alpha}{2}}.
\end{equation*}
Note moreover that
\begin{equation}
\label{eq:RangeEta1General}
\big| \big( \tau_1 - \omega_\alpha(\xi_1,\eta_1) \big) + \big( (\tau - \tau_1) - \omega_\alpha(\xi - \xi_1,\eta- \eta_1) \big) \big| \lesssim L_2.
\end{equation}
We estimate the measure of $S$ by fixing $\xi_1$, which amounts to a factor $|I_1|$, and counting $\tau_1$ and $\eta_1$. From \eqref{eq:GroupVelocityGeneral} and \eqref{eq:RangeEta1General} follows that for fixed $\tau$, $\xi_1$, $\xi$, $\eta$ that on $\R^2$ the measure of $\eta_1$ is estimated by 
$\big( \frac{L_2}{ N_{\max}^{\frac{\alpha}{2}}} \big)^{\frac{1}{2}}$. On $\R \times \T$ there are at most $\langle \frac{L_2}{ N_{\max}^{\frac{\alpha}{2}}} \rangle^{\frac{1}{2}}$ values of $\eta_1$ such that $(\xi_1,\eta_1,\tau_1) \in S$. This is a consequence of the mean-value theorem. Finally, with $(\xi_1,\eta_1)$ fixed, we trivially estimate 
\begin{equation*}
\text{meas} ( \{ \tau_1 : |\tau_1 - \omega_\alpha(\xi_1,\eta_1)| \leqslant L_1 \} ) \leqslant L_1.
\end{equation*}

In conclusion we obtain
\begin{equation*}
\begin{split}
\text{meas} (S) &\leqslant |I_1| \cdot \# \{ \eta_1 : \eqref{eq:GroupVelocityGeneral} \text{ and } \eqref{eq:RangeEta1General} \text{ holds } \} \cdot \big| \{ \tau_1 : |\tau_1 - \omega_\alpha(\xi_1,\eta_1) \big| \leqslant L_1 \} \big| \\
 &\lesssim N_{\min} \cdot C_{1,\D}(L_2,N_{\max}) \cdot L_1.
\end{split}
\end{equation*}
with $C_{1,\D}$ defined in \eqref{eq:ConstantBilinearStrichartz}. This finishes the proof.
\end{proof}

We shall see that the support assumptions are satisfied in the resonant case, which implies:
\begin{corollary}
\label{prop:BilinearStrichartzResonant}
Let $\D \in \{ \R^2, \R \times \T \}$, $N_3 \sim N_1 \gtrsim N_2$ and for $i=1,2$: $\text{supp}(f_{i,N_i,L_i}) \subseteq D_{N_i,L_i}$ with $L_i \ll N_1^\alpha N_2$ for $i=1,2,3$. Then the following estimate holds with $C_{1,\D}$ defined in \eqref{eq:ConstantBilinearStrichartz}:
\begin{equation}
\label{eq:BilinearStrichartzEstimate}
\begin{split}
&\quad \| 1_{D_{\alpha,N_3,L_3}} ( f_{1,N_1,L_1} * f_{2,N_2,L_2} ) \|_{L^2_{\tau,\xi,\eta}(\R \times \D^*)} \\
&\lesssim (L_1 \wedge L_2)^{\frac{1}{2}} N_{\min}^{\frac{1}{2}} C_{1,\D}(L_1 \vee L_2, N_{\max}) \prod_{i=1}^2 \| f_{i,N_i,L_i} \|_{L^2_{\tau,\xi,\eta}(\R \times \D^*)}.
\end{split}
\end{equation}
\end{corollary}
\begin{proof}
Since
\begin{equation*}
\begin{split}
&\quad |\Omega_\alpha(\xi_1,\eta_1,\xi-\xi_1,\eta-\eta_1)| \\
&= |\tau - \omega_\alpha(\xi,\eta) - (\tau_1 - \omega_\alpha(\xi_1,\eta_1)) - ((\tau-\tau_1) - \omega_\alpha(\xi-\xi_1,\eta-\eta_1)) | \\
&\ll N_{\max}^\alpha N_{\min},
\end{split}
\end{equation*}
we are in the resonant case \eqref{eq:ResonanceCondition} and we have by \eqref{eq:BoundGroupVelocityResonantCase}
\begin{equation}
\label{eq:GroupVelocityDifferenceBilinear}
\big| \frac{\eta_1}{\xi_1} - \frac{\eta - \eta_1}{\xi - \xi_1} \big| \sim N_{\max}^{\frac{\alpha}{2}}.
\end{equation}
The assumptions of Proposition \ref{prop:BilinearStrichartzGeneral} are satisfied, which yields the claim.
%
\end{proof}

We have the following alternative bilinear Strichartz estimate, which is based on the second order transversality:
\begin{equation*}
|\partial^2_{\eta} ( \omega_\alpha(\xi_1,\eta_1) + \omega_\alpha(\xi - \xi_1,\eta- \eta_1) | = 2 \Big| \frac{1}{\xi_1} + \frac{1}{\xi - \xi_1} \Big|.
\end{equation*}

 The estimate was observed by Bourgain \cite{Bourgain1993KPII} in the context of the KP-II equation. This will serve to estimate several boundary cases, in case the low frequency is very small or the resonance is very large. Define
 \begin{equation}
 \label{eq:ConstantAlternativeBilinearStrichartz}
 C_{2,\D}(L,N) = 
 \begin{cases}
 	(L N)^{\frac{1}{4}}, \quad &\D = \R^2, \\
 	\langle L N \rangle^{\frac{1}{4}}, \quad &\D = \R \times \T.
 \end{cases}
 \end{equation}
 
\begin{lemma}[{{\cite{Bourgain1993KPII}}}]
\label{lem:AlternativeBilinearStrichartzEstimate}
Let $\D \in \{ \R^2, \R \times \T \}$, and $f_{i,N_i,L_i} : \R \times \D^* \to \R_{\geqslant 0}$ with $\text{supp}(f_{i,N_i,L_i}) \subseteq D_{N_i,L_i}$ and $\pi_{\xi}(\text{supp}(f_{i,N_i,L_i})) \subseteq I_i$ with $|I_i| \leqslant K$. Then the following estimate holds:
\begin{equation}
\label{eq:AlternativeBilinearStrichartz}
\| f_{1,N_1,L_1} * f_{2,N_2,L_2} \|_{L^2_{\tau,\xi,\eta}} \lesssim K^{\frac{1}{2}} (L_1 \wedge L_2)^{\frac{1}{2}} C_{2,\D}(L_1 \vee L_2, N_{\min}) \prod_{i=1}^2 \| f_{i,N_i,L_i} \|_{L^2}.
\end{equation}
\end{lemma}

We have the following consequence of the above bilinear Strichartz estimates and the C\'ordoba--Fefferman square function estimate:
\begin{lemma}
\label{lem:CFBilinearStrichartz}
Let $\alpha \geqslant 1$, $N_1 \in 2^{\N_0}$, $N_2 \in 2^{\Z}$, $N_1 \gg N_2$, $L_1, L_2 \in 2^{\N_0}$. For $i=1,2$, let $f_{i,N_i,L_i} : \R \times \R^2 \to \R_{\geqslant 0}$ with
$\text{supp}(f_{i,N_i,L_i}) \subseteq D_{N_i,L_i}$.

Then the following estimate holds:
\begin{equation*}
\| f_{1,N_1,L_1} * f_{2,N_2,L_2} \|_{L^2_{\tau,\xi,\eta}} \lesssim \frac{N_2^{\frac{1}{2}}}{N_1^{\frac{\alpha}{4}}} \prod_{i=1}^2 L_i^{\frac{1}{2}} \| f_{i,N_i,L_i} \|_{L^2}.
\end{equation*}
\end{lemma}
\begin{proof}
If $L_{\max} \gtrsim N_1^\alpha N_2$, the estimate is a consequence of Lemma \ref{lem:AlternativeBilinearStrichartzEstimate}:
\begin{equation*}
\begin{split}
\| f_{1,N_1,L_1} * f_{2,N_2,L_2} \|_{L^2_{\tau,\xi,\eta}} &\lesssim N_2^{\frac{3}{4}} L_{\min}^{\frac{1}{2}} L_{\max}^{\frac{1}{4}} \prod_{i=1}^2 \| f_{i,N_i,L_i} \|_2 \\
&\lesssim \frac{N_2^{\frac{1}{2}}}{N_1^{\frac{\alpha}{4}}} \prod_{i=1}^2 L_i^{\frac{1}{2}} \| f_{i,N_i,L_i} \|_2.
\end{split}
\end{equation*}

In the following we suppose that $L_{\max} \ll N_1^\alpha N_2$. We carry out a dyadic decomposition (Whitney) in the transversality parameter:
\begin{equation*}
D \sim \big| \frac{\eta_1}{\xi_1} - \frac{\eta_2}{\xi_2} \big|
\end{equation*}
with the range $D \in [\big( \frac{L_{\max}}{N_2} \big)^{\frac{1}{2}}, \infty )$:
\begin{equation*}
\| f_{1,N_1,L_1} * f_{2,N_2,L_2} \|_{L^2_{\tau,\xi,\eta}} \leqslant \sum_{D \geqslant \big( \frac{L_{\max}}{N_2} \big)^{\frac{1}{2}}} \sum_{I_1^D \sim I_2^D} \| f_{1,N_1,L_1}^{I_1^D} * f_{2,N_2,L_2}^{I_2^D} \|_{L^2_{\tau,\xi,\eta}}.
\end{equation*}
Here we have broken the support of $f_{i,N_i,L_i}$ into intervals $I_i^D$ of $(\eta_i/\xi_i)$ of length $\sim D$ with
\begin{equation*}
\big| \frac{\eta_1}{\xi_1} - \frac{\eta_2}{\xi_2} \big| \sim D \text{ for } \frac{\eta_i}{\xi_i} \in I_i^D.
\end{equation*}
The contribution of $D \gtrsim N_1^{\frac{\alpha}{2}}$ is handled by Proposition \ref{prop:BilinearStrichartzGeneral}:
\begin{equation*}
\begin{split}
\sum_{D \gtrsim N_1^{\frac{\alpha}{2}}} \sum_{I_1^D \sim I_2^D} \| f_{1,N_1,L_1}^{I_1^D} * f_{2,N_2,L_2}^{I_2^D} \|_{L^2_{\tau,\xi,\eta}} &\lesssim \sum_{D \gtrsim N_1^{\frac{\alpha}{2}}} \sum_{I_1^D \sim I_2^D} \frac{N_2^{\frac{1}{2}}}{D^{\frac{1}{2}}} \prod_{i=1}^2 L_i^{\frac{1}{2}} \| f_{i,N_i,L_i}^{I_i^D} \|_{L^2_{\tau,\xi,\eta}} \\
&\lesssim \frac{N_2^{\frac{1}{2}}}{N_1^{\frac{\alpha}{4}}} \prod_{i=1}^2 L_i^{\frac{1}{2}} \| f_{i,N_i,L_i} \|_{L^2_{\tau,\xi,\eta}}.
\end{split}
\end{equation*}

For $D \ll N_1^{\frac{\alpha}{2}}$ we can use an almost orthogonal decomposition to effectively reduce the $\xi$-support of $f_{i,N_i,L_i}$. Without loss of generality we can suppose that $\pi_{\xi}(f_{i,N_i,L_i}) \subseteq \R_{>0}$ by complex conjugation and symmetry of the dispersion relation.

The convolution constraint reads
\begin{equation*}
\left\{ \begin{array}{cl}
\xi_1 + \xi_2 &= \xi_3 + \xi_4, \\
\eta_1 + \eta_2 &= \eta_3 + \eta_4, \\
\xi_1^{\alpha+1} + \frac{\eta_1^2}{\xi_1} + \xi_2^{\alpha+1} + \frac{\eta_2^2}{\xi_2} &= \xi_3^{\alpha+1} + \frac{\eta_3^2}{\xi_3} + \xi_4^{\alpha+1} + \frac{\eta_4^2}{\xi_4} + \mathcal{O}(L_{\max}).
\end{array} \right.
\end{equation*}

We rescale to unit $\xi$-frequencies by $\xi \to \frac{\xi}{N_1}$, $\eta \to \frac{\eta}{N_1^{\frac{\alpha}{2} + 1}}$ to find the following for the renormalized frequencies: \small
\begin{equation*}
\left\{ \begin{array}{cl}
\xi'_1 + \xi'_2 &= \xi'_3 + \xi'_4, \\
\eta'_1 + \eta'_2 &= \eta'_3 + \eta'_4, \\
(\xi'_1)^{\alpha+1} + \frac{(\eta'_1)^2}{\xi'_1} + (\xi'_2)^{\alpha+1} + \frac{(\eta'_2)^2}{\xi'_2} &= (\xi'_3)^{\alpha+1} + \frac{(\eta'_3)^2}{\xi'_3} \\
&\quad + (\xi'_4)^{\alpha+1} + \frac{(\eta'_4)^2}{\xi'_4} + \mathcal{O}(L_{\max}/N_1^{\alpha+1}).
\end{array} \right.
\end{equation*}
\normalsize
We subtract
\begin{equation*}
\frac{(\eta_1' + \eta_2')^2}{\xi_1'+\xi_2'} = \frac{(\eta_3'+\eta_4')^2}{\xi_3'+\xi_4'}
\end{equation*}
from the third equation to find
\begin{equation*}
\left\{ \begin{array}{cl}
\xi_1' + \xi_2' &= \xi_3' + \xi_4', \\
(\xi_1')^{\alpha+1} + (\xi_2')^{\alpha+1}- \frac{(\eta_1' \xi_2' - \eta_2' \xi_1')^2}{\xi_1' \xi_2' (\xi_1'+\xi_2')} &= (\xi_3')^{\alpha+1} + (\xi_4')^{\alpha+1} \\
&\quad \quad - \frac{(\eta_3' \xi_4' - \eta_4' \xi_3')^2}{\xi_3' \xi_4' (\xi_3'+\xi_4')} + \mathcal{O}(L_{\max}/N_1^{\alpha+1}).
\end{array} \right.
\end{equation*}
This yields
\begin{equation*}
\left\{ \begin{array}{cl}
\xi_1' + \xi_2' &= \xi_3' + \xi_4', \\
(\xi_1')^{\alpha+1} + (\xi_2')^{\alpha+1} &= (\xi_3')^{\alpha+1} + (\xi_4')^{\alpha+1} + \mathcal{O}(\frac{D^2 N_2}{N_1^{\alpha+1}} + \frac{L_{\max}}{N_1^{\alpha+1}}).
\end{array} \right.
\end{equation*}

Note that for our Whitney decomposition we always have $D^2 N_2 \gtrsim L_{\max}$. 
Note that the curve $\xi' \mapsto (\xi')^{\alpha+1}$ degenerates at the origin for $\alpha > -1$. To still find an almost orthogonality resembling the Córdoba--Fefferman square function estimate, we rewrite the second line as
\begin{equation*}
f'(\xi_{1*}) (\xi_1' - \xi_3') + f'(\xi_{2*})(\xi_2'-\xi_4') = \mathcal{O}\big( \frac{D^2 N_2}{N_1^\alpha N_1} \big).
\end{equation*}
We have $f'(\xi_{2*}) = f'(\xi_{1*}) + \Delta$ with $|\Delta| \gtrsim 1$, for which reason
\begin{equation*}
f'(\xi_{1*}) (\xi_1' - \xi_3' + \xi_2' - \xi_4') + \Delta (\xi_2'-\xi_4') = \mathcal{O}\big( \frac{D^2 N_2}{N_1^\alpha N_1} \big).
\end{equation*}
Consequently, $\xi_i' = \xi_{i+2}' + \mathcal{O}\big( \frac{D^2 N_2}{N_1^\alpha N_1} \big)$ for $i \in \{1,2\}$, and we obtain an almost orthogonal decomposition of
\begin{equation*}
\| f'_{1,N_1,L_1} * f'_{2,N_2,L_2} \|^2_{L^2_{\tau,\xi,\eta}} \lesssim \sum_{I_1', I'_2} \| f'_{1,N_1,L_1,I'_1} * f'_{2,N_2,L_2,I'_2} \|^2_{L^2}
\end{equation*}
with $\xi'$-intervals of length $\mathcal{O}(\frac{D^2 N_2}{N_1^\alpha N_2})$.
After rescaling we obtain a decomposition into intervals of length $D^2 N_2 / N_1^\alpha$.

First we handle the contribution $D \gg L_{\max}/N_2$. Applying the bilinear Strichartz estimate Proposition \ref{prop:BilinearStrichartzGeneral} gives
\begin{equation*}
\| f^{I_1}_{1,N_1,L_1} * f^{I_2}_{2,N_2,L_2} \|_{L^2_{\tau,\xi,\eta}} \lesssim \frac{D N_2^{\frac{1}{2}}}{N_1^{\frac{\alpha}{2}}} D^{-\frac{1}{2}} \prod_{i=1}^2 L_i^{\frac{1}{2}} \| f^{I_i}_{i,N_i,L_i} \|_2.
\end{equation*}
Summation in the intervals can be carried out by almost orthogonality. Summation in $L_{\max} / N_2 \ll D \lesssim N_1^{\alpha/2}$ gives
\begin{equation*}
\sum_{L_{\max} / N_2 \ll D \lesssim N_1^{\alpha/2}} \frac{D^{\frac{1}{2}} N_2^{\frac{1}{2}}}{N_1^{\frac{\alpha}{2}}} \lesssim \frac{N_2^{\frac{1}{2}}}{N_1^{\frac{\alpha}{4}}}.
\end{equation*}

For the contribution $D \sim \big( \frac{L_{\max}}{N_2} \big)^{\frac{1}{2}}$ we use Lemma \ref{lem:AlternativeBilinearStrichartzEstimate} after almost orthgonal decomposition into intervals of length $\big( \frac{L_{\max}}{N_1^{\alpha}} \big)^{\frac{1}{2}}$:
\begin{equation*}
\| f_{1,N_1,L_1}^{I_1} * f^{I_2}_{2,N_2,L_2} \|_{L^2_{\tau,\xi,\eta}} \lesssim L_{\min}^{\frac{1}{2}} \big( \frac{L_{\max}}{N_1^\alpha} \big)^{\frac{1}{2}} N_2^{\frac{1}{4}} L_{\max}^{\frac{1}{4}} \prod_{i=1}^2 \| f_{i,N_i,L_i}^{I_i} \|_2.
\end{equation*}
By the upper bound for $L_{\max}$ this gives
\begin{equation*}
\| f_{1,N_1,L_1}^{I_1} * f^{I_2}_{2,N_2,L_2} \|_{L^2_{\tau,\xi,\eta}} \lesssim \frac{N_2^{\frac{1}{2}}}{N_1^{\frac{\alpha}{4}}} \prod_{i=1}^2 \| f_{i,N_i,L_i}^{I_i} \|_2.
\end{equation*}
This handles all possible transversalities $D \in [\big( \frac{L_{\max}}{N_2} \big)^{\frac{1}{2}}, N_1^\alpha]$. The proof is complete.
\end{proof}

We have the following analog on $\R \times \T$:
\begin{lemma}
\label{lem:CFBilinearCylinder}
Let $N_i \in 2^{\Z}$, $N_2 \ll N_1$, $N_1 \gtrsim 1$, $L_i \in 2^{\N_0}$, $i=1,2$, $D^* \in 2^{\N}$. Let $\text{supp}(f_{i,N_i,L_i}) \subseteq D_{N_i,L_i}$, $i=1,2$ with
\begin{equation*}
\big| \frac{\eta_1}{\xi_1} - \frac{\eta_2}{\xi_2} \big| \lesssim D^*
\end{equation*}
for $(\xi_i,\eta_i) \in \pi_{\xi,\eta}(\text{supp}(f_{i,N_i,L_i}))$. Then the following estimate holds:
\begin{equation}
\label{eq:RefinedBilinearCylinder}
\| f_{1,N_1,L_1} * f_{2,N_2,L_2} \|_{L^2_{\tau,\xi,\eta}} \lesssim \log(D^*) N_2^{\frac{1}{2}} L_{\min}^{\frac{1}{2}} \langle L_{\max} / N_1^{\frac{\alpha}{2}} \rangle^{\frac{1}{2}} \prod_{i=1}^2 \| f_{i,N_i,L_i} \|_2.
\end{equation}
\end{lemma}
\begin{proof}
First, we suppose that $L_{\max} \gtrsim N_1^\alpha N_2$. In this case we apply Lemma \ref{lem:AlternativeBilinearStrichartzEstimate} to find
\begin{equation*}
\| f_{1,N_1,L_1} * f_{2,N_2,L_2} \|_{L^2} \lesssim N_2^{\frac{1}{2}} L_{\min}^{\frac{1}{2}} \langle L_{\max} N_2 \rangle^{\frac{1}{4}} \prod_{i=1}^2 \| f_{i,N_i,L_i} \|_2,
\end{equation*}
which implies \eqref{eq:RefinedBilinearCylinder}.

Next, suppose that $L_{\max} \ll N_1^\alpha N_2$.
We carry out a Whitney decomposition in the transversality parameter
\begin{equation*}
D \sim \big| \frac{\eta_1}{\xi_1} - \frac{\eta_2}{\xi_2} \big|
\end{equation*}
in the range $D \in [ \big( \frac{L_{\max}}{N_2} \big)^{\frac{1}{2}}, D^*]$:
\begin{equation*}
\| f_{1,N_1,L_1} * f_{2,N_2,L_2} \|_{L^2_{\tau,\xi,\eta}} \leqslant \sum_{D \in \big[ \big( \frac{L_{\max}}{N_2}\big)^{\frac{1}{2}}, D^* ]} \sum_{I_1^D \sim I_2^D} \| f_{1,N_1,L_1}^{I_{1}^D} * f_{2,N_2,L_2}^{I_{2}^D} \|_{L^2_{\tau,\xi,\eta}}.
\end{equation*}
By the argument from the proof of Lemma \ref{lem:CFBilinearStrichartz} we can effectively localize the $\xi$-support to intervals of length $\mathcal{O}(\frac{D^2 N_2}{N_1^\alpha} \big)$. We obtain for $D= \big( L_{\max} / N_2 \big)^{\frac{1}{2}}$ by applying Lemma \ref{lem:AlternativeBilinearStrichartzEstimate}
\begin{equation}
\label{eq:AuxBilStrCylI}
\| f^{I_{1}^D}_{1,N_1,L_1} * f^{I_{2}^D}_{2,N_2,L_2} \|_{L^2_{\tau,\xi,\eta}} \lesssim L_{\min}^{\frac{1}{2}} \frac{L_{\max}^{\frac{1}{2}}}{N_1^{\frac{\alpha}{2}}} \langle L_{\max} N_2 \rangle^{\frac{1}{4}} \prod_{i=1}^2 \| f_{i,N_i,L_i}^{I_{i}^D} \|_{L^2}.
\end{equation}
This is sufficient for $L_{\max} \ll N_1^\alpha N_2$.

Next, we handle the contribution $D \in \big( \big( \frac{L_{\max}}{N_2} \big)^{\frac{1}{2}}, N_1^{\alpha/2} \big]$. After localizing the $\xi$-support to intervals of length $\mathcal{O}\big( D^2 N_2 / N_1^\alpha \big)$ by almost orthogonality, we can apply Proposition \ref{prop:BilinearStrichartzGeneral} to find
\begin{equation}
\label{eq:AuxBilStrCylII}
\begin{split}
&\quad \sum_{D \in \big( \big( \frac{L_{\max}}{N_2}\big)^{\frac{1}{2}}, N_1^{\frac{\alpha}{2}} \big) ]} \sum_{I_1^D \sim I_2^D} \| f_{1,N_1,L_1}^{I_{1}^D} * f_{2,N_2,L_2}^{I_{2}^D} \|_{L^2_{\tau,\xi,\eta}} \\
&\lesssim L_{\min}^{\frac{1}{2}} \sum_{D \in \big( \big( \frac{L_{\max}}{N_2}\big)^{\frac{1}{2}}, N_1^{\frac{\alpha}{2}} \big) ]}  \frac{D N_2^{\frac{1}{2}}}{N_1^{\frac{\alpha}{2}}} \langle L_{\max} / D \rangle^{\frac{1}{2}} \prod_{i=1}^2 \| f_{i,N_i,L_i} \|_2 \\
&\lesssim L_{\min}^{\frac{1}{2}} N_2^{\frac{1}{2}} \langle L_{\max} / N_1^{\alpha/2} \rangle^{\frac{1}{2}} \prod_{i=1}^2 \| f_{i,N_i,L_i} \|_2.
\end{split}
\end{equation}

Finally, we turn to the contribution with large transversality. Here we apply Proposition \ref{prop:BilinearStrichartzGeneral} directly to find
\begin{equation}
\label{eq:AuxBilStrCylinderIII}
\begin{split}
&\quad \sum_{D \in [ N_1^{\frac{\alpha}{2}} , D^* ]} \sum_{I_1^D \sim I_2^D} \| f_{1,N_1,L_1}^{I_{1}^D} * f_{2,N_2,L_2}^{I_{2}^D} \|_{L^2_{\tau,\xi,\eta}} \\
&\lesssim N_2^{\frac{1}{2}} L_{\min}^{\frac{1}{2}} \sum_{D \in [ N_1^{\frac{\alpha}{2}} , D^* ]} \langle L_{\max} / D \rangle^{\frac{1}{2}} \prod_{i=1}^2 \| f_{i,N_i,L_i} \|_2 \\
&\lesssim \log(D^*) N_2^{\frac{1}{2}} \langle L_{\max} / N_1^{\alpha/2} \rangle^{\frac{1}{2}} \prod_{i=1}^2 \| f_{i,N_i,L_i} \|_2.
\end{split}
\end{equation}

Collecting contributions \eqref{eq:AuxBilStrCylI}-\eqref{eq:AuxBilStrCylinderIII} the proof is complete.
\end{proof}

We record the following simple estimate, which is a consequence of the Cauchy-Schwarz inequality. This will serve as substitute when \eqref{eq:GroupVelocityDifferenceBilinear} fails:
\begin{lemma}
\label{lem:CauchySchwarz}
Let $f_{i,N_i,L_i} : \R \times \D^* \to \C$ satisfy support properties $\pi_{\xi} \text{supp}(f_i) \subseteq I_i$, $i=1,2$ with $|I_i| \sim N_i$ and $\pi_{\eta} \text{supp}(f_i) \subseteq J_i$ with $|J_i| \lesssim M_i$ and $|\tau_i - \omega_\alpha(\xi_i,\eta_i)| \lesssim L_i$ for $(\tau_i,\xi_i,\eta_i) \in \text{supp}(f_{i,N_i,L_i})$ and $i=1,2$.
The following estimate holds:
\begin{equation*}
\| f_{1,N_1,L_1} * f_{2,N_2,L_2} \|_{L^2_{\tau,\xi,\eta}} \lesssim (N_{\min} M_{\min} L_{\min})^{\frac{1}{2}} \prod_{i=1}^2 \| f_{i,N_i,L_i} \|_{L^2}.
\end{equation*}
\end{lemma}

\subsection{Nonlinear Loomis-Whitney inequalities with measure}
\label{subsection:NonlinearLWGeneral}

In preparation for the proof of trilinear convolution estimates for approximate solutions to dispersion-generalized KP-I equations, we prove Theorem \ref{thm:GenearlNLW}. We impose the following assumptions on the hypersurfaces $(S_i)_{i=1,2,3}$ following \cite{KinoshitaSchippa2021,KochSteinerberger2015}:
\begin{assumption}
\label{assumption:Surfaces}
For $i=1,2,3$ there exist $0 < \beta \leqslant 1$, $b>0$, $A \geqslant 1$, $F_i \in C^{1,\beta}(\mathcal{U}_i)$, where the $\mathcal{U}_i$ denote open and convex sets in $\R^2$ and $G_i \in O(3)$ such that
\begin{itemize}
\item[(i)] the oriented surfaces $S_i$ are given by
\begin{equation*}
S_i = G_i \text{gr}(F_i), \quad \text{gr}(F_i) = \{ (x,y,z) \in \R^3 : z = F_i(x,y), \; (x,y) \in \mathcal{U}_i \};
\end{equation*}
\item[(ii)] the vector field $\mathfrak{n}_i : S_i \to \mathbb{S}^2$ of outward unit normals on $S_i$ satisfies 
the H\"older condition:
\begin{equation*}
\sup_{\sigma,\tilde{\sigma}} \frac{|\mathfrak{n}_i(\sigma) - \mathfrak{n}_i(\tilde{\sigma})|}{|\sigma - \tilde{\sigma}|^\beta} + \frac{|\mathfrak{n}_i(\sigma).(\sigma - \tilde{\sigma})|}{|\sigma - \tilde{\sigma}|^{1+\beta}} \leqslant b;
\end{equation*}
\item[(iii)] there is $A \geqslant 1$ such that for all $\sigma_i \in S_i$, $i=1,2,3$ the following estimate holds:
\begin{equation*}
A^{-1} \leqslant |\mathfrak{n}_1(\sigma_1) \wedge \mathfrak{n}_2(\sigma_2) \wedge \mathfrak{n}_3(\sigma_3) | \leqslant 1.
\end{equation*}
\end{itemize}
\end{assumption}

For convenience, we recall the statement of Theorem \ref{thm:GenearlNLW}:
\begin{theorem}[Nonlinear~Loomis--Whitney~inequalities~with~general~measure]
Let $\varepsilon>0$, $(S_i)_{i=1,2,3}$ be a collection of hypersurfaces, which satisfy Assumption \ref{assumption:Surfaces}, $\nu_{\gamma}$ be a Borel measure, which satisfies \eqref{eq:gammaBorelMeasure}, and $f_i \in L^2(S_i(\varepsilon),d\nu_\gamma)$, $i=1,2$. Then the following estimate holds:
\begin{equation*}
\| f_1 * f_2 \|_{L^2(S_3(\varepsilon),d\nu_\gamma)} \lesssim A^{\frac{1}{2}} \varepsilon^{\frac{\gamma}{2}} \prod_{i=1}^2 \| f_i \|_{L^2(S_i(\varepsilon),d\nu_\gamma)}.
\end{equation*}
\end{theorem}
\begin{proof}
We follow the argument in \cite{KinoshitaSchippa2021}.
By duality it suffices to show the trilinear expression
\begin{equation*}
\int_{S_3(\varepsilon)} (f_1 * f_2) f_3 d \nu_\gamma \lesssim A^{\frac{1}{2}} \varepsilon^{\frac{\gamma}{2}} \prod_{i=1}^3 \| f_i \|_{L^2(S_i(\varepsilon),d \nu_\gamma)}.
\end{equation*}
To this end, let $(B_{\varepsilon,j})_{j \in \N}$ be an essentially disjoint cover of $\R^3$ with balls of size $\varepsilon$. Set 
\begin{equation*}
J_{3,\varepsilon} = \{ j \in \N : B_{\varepsilon,j} \cap S_3(\varepsilon) \neq \emptyset \} \text{ and } S_{3,j}(\varepsilon) = B_{\varepsilon,j} \cap S_3(\varepsilon).
\end{equation*}
By Minkowski's inequality we find
\begin{equation*}
\big| \int_{\R^3} (f_1 * f_2)(\lambda) f_3(\lambda) d\nu_\gamma(\lambda) \big| \leqslant \sum_{j \in J_{3,\varepsilon}} \big| \int_{\R^3} (f_1 * f_2)(\lambda) f_3 \big\vert_{S_{3,j}(\varepsilon)}(\lambda) d\nu_\gamma(\lambda) \big|.
\end{equation*}
By H\"older's inequality and \eqref{eq:gammaBorelMeasure}, we find
\begin{equation*}
\big| \int_{\R^3} (f_1 * f_2)(\lambda) f_3 \big\vert_{S_{3,j_3}(\varepsilon)}(\lambda) d\nu_\gamma (\lambda) \big| \lesssim \varepsilon^{\frac{\gamma}{2}} \prod_{i=1}^3 \| f_i \|_{L^2(S_{i,j,\varepsilon})}.
\end{equation*}
Above we denote
\begin{align*}
S_{1,j,\varepsilon} &= \{ \lambda_1 \in S_1(\varepsilon) : \exists \lambda' \in B_{\varepsilon,j} : \lambda' - \lambda_1 \in S_2(\varepsilon) \}, \\
S_{2,j,\varepsilon} &= \{ \lambda_2 \in S_2(\varepsilon) : \exists \lambda' \in B_{\varepsilon,j} : \lambda' - \lambda_2 \in S_1(\varepsilon)\}.
\end{align*}
The following estimate was proved in \cite[Eq.~(16),~p.~13]{KinoshitaSchippa2021}:
\begin{equation}
\label{eq:LWTransversalitySummation}
\sum_{j \in J_{3,\varepsilon}} \chi_{S_{1,j,\varepsilon} \times S_{2,j,\varepsilon}}(\lambda_1,\lambda_2) \lesssim A.
\end{equation}
With this at hand, we can conclude as follows:
\begin{equation*}
\begin{split}
&\quad \sum_{j \in J_{3,\varepsilon}} \big| \int_{\R^3} (f_1 * f_2)(\lambda) f_3 \big\vert_{S_{3,j}(\varepsilon)}(\lambda) d \nu_\gamma(\lambda) \big| \\
&\lesssim \varepsilon^{\frac{\gamma}{2}} \sum_{j \in J_{3,\varepsilon}} \| f_1 \|_{L^2(S_{1,j,\varepsilon}, d\nu_\gamma)} \| f_2 \|_{L^2(S_{2,j,\varepsilon},d \nu_\gamma)} \| f_3\|_{L^2(S_{3,j}(\varepsilon),d\nu_\gamma)} \\
&\lesssim \varepsilon^{\frac{\gamma}{2}} \big( \sum_{ j \in J_{3,\varepsilon}} \| f_3 \|_{L^2(S_{3,j}(\varepsilon),d\nu_\gamma)}^2 \big)^{\frac{1}{2}} \big( \sum_{j \in J_{3,\varepsilon}} \| f_1 \|^2_{L^2(S_{1,j,\varepsilon}, d\nu_\gamma)} \| f_2 \|_{L^2(S_{2,j,\varepsilon},d\nu_\gamma)}^2 \big)^{\frac{1}{2}} \\
&\lesssim \varepsilon^{\frac{\gamma}{2}} \| f_3 \|_{L^2(S_3(\varepsilon),d\nu_\gamma)} \big( \int_{\R^3 \times \R^3} \sum_{j \in J_{3,\varepsilon}} \chi_{S_{1,j,\varepsilon} \times S_{2,j,\varepsilon}}(\lambda_1,\lambda_2)  \\
&\quad \qquad \quad \times |f_1(\lambda_1)|^2 |f_2(\lambda_2)|^2 d\nu_\gamma(\lambda_1) d\nu_\gamma(\lambda_2) \big)^{\frac{1}{2}} \\
&\lesssim \varepsilon^{\frac{\gamma}{2}} A^{\frac{1}{2}} \prod_{i=1}^3 \| f_i \|_{L^2(S_i(\varepsilon)}.
\end{split}
\end{equation*}
\end{proof}

\subsection{Resonance and the nonlinear Loomis--Whitney inequality}

In addition to the lower bound for the difference in group velocity in the resonant case, we shall see that we can compute the full transversality.
To show a trilinear estimate in the resonant case, consider
\begin{equation*}
\int (f_{1,N_1,L_1} * f_{2,N_2,L_2} ) f_{3,N_3,L_3} d \tau d \xi d \eta
\end{equation*}
with $\text{supp}(f_{i,N_i,L_i}) \subseteq D_{N_i, \leqslant L_i}$ with $L_i \ll N_1^\alpha N_2$. Let $N = \max(N_i)$. We suppose that $N_1 = N$ and $N_3 \sim N_1 \gtrsim N_2$. We use the scaling 
\begin{equation}
\label{eq:AnisotropicScaling}
\tau \to \frac{\tau}{ N^{\alpha+1}}, \; \xi \to \frac{\xi}{N}, \; \eta \to \frac{\eta}{ N^{\frac{\alpha}{2}+1}}
\end{equation}
to reduce to $1 \sim |\xi_1'| \sim |\xi_3'| \gtrsim |\xi_2'| \sim \frac{N_2}{N_1}$. The smallness of $|\xi_2'|$ allows us to localize $\xi_i'$ to intervals of length $\frac{N_2}{N_1}$ using almost orthogonality.

 In the resonant case, it holds, moreover
\begin{equation*}
\big| \frac{\eta_1'}{\xi_1'} - \frac{\eta_2'}{\xi_2'} \big| \sim 1.
\end{equation*}
This induces an almost orthogonal decomposition in $\eta_i'/\xi_i'$ (cf. \cite{IonescuKenigTataru2008}), by which we can suppose that $\eta_i'/\xi_i'$ is localized to intervals of size $\sim 1$. Now we use a Galilean invariance:
\begin{equation*}
\frac{\eta_i'}{\xi_i'} \to \frac{\eta_i'}{\xi_i'} - A.
\end{equation*}
This localizes $|\frac{\eta_i'}{\xi_i'}| \lesssim 1$. Considering $i=2$, we can suppose that $|\eta_2'| \lesssim N_2/N_1$. This induces an almost orthogonal decomposition in $\eta_i'$, by which we can suppose that $(\xi_i',\eta_i')$ are localized to balls of size $c(N_2/N_1)$ for some $c \ll 1$ to be chosen later. We stress that like in the proof of Proposition \ref{prop:StrichartzCylinder} for $(\xi',\eta') \in \R \times \Z/(N_1^{\frac{\alpha}{2}+1})$, we have in general for the transformed variables $(\xi',\eta' - A \xi') \notin \R \times \Z/(N_1^{\frac{\alpha}{2}+1})$ and some care is required on the cylinder when estimating the measure.


To compute the full transversality, we need to consider the wedge product of normals
\begin{equation*}
|\mathfrak{n}(\xi_1,\eta_1) \wedge \mathfrak{n}(\xi_2,\eta_2) \wedge \mathfrak{n}(\xi_3,\eta_3) |.
\end{equation*}
We compute
\begin{equation*}
\mathfrak{n}(\xi',\eta') = 
\begin{pmatrix}
-(\alpha+1) |\xi'|^{\alpha} + (\eta')^2 / (\xi')^2 \\
-2 \eta' / \xi' \\
1
\end{pmatrix}
.
\end{equation*}
The normals have modulus $\sim 1$ after rescaling and supposing that $|\eta'| \lesssim 1$. This is clear for $\mathfrak{n}(\xi_i',\eta_i')$ for $i=1,3$ since $|\xi_i'|\sim 1$. In the resonant case it holds moreover
\begin{equation*}
\big| \frac{\eta_1'}{\xi_1'} - \frac{\eta_2'}{\xi_2'} \big| \sim 1,
\end{equation*}
for which reason we have $|\mathfrak{n}(\xi_2',\eta_2')| \sim 1$.

Let $C_\alpha = \alpha+1$. We find under the convolution constraint:\small
\begin{equation*}
\begin{split}
&\quad \det(\mathfrak{n}(\xi_1,\eta_1), \mathfrak{n}(\xi_2,\eta_2), \mathfrak{n}(\xi_1+\xi_2,\eta_1+\eta_2)) \\
 &= 
\begin{vmatrix}
-C_\alpha |\xi_1|^{\alpha} + \frac{\eta_1^2}{\xi_1^2} & -C_\alpha |\xi_2|^\alpha + \frac{\eta_2^2}{\xi_2^2} & -C_\alpha |\xi_1+\xi_2|^\alpha + \frac{(\eta_1+\eta_2)^2}{(\xi_1+\xi_2)^2} \\
- \frac{2 \eta_1}{\xi_1} & - \frac{2 \eta_2}{\xi_2} & - \frac{2(\eta_1+\eta_2)}{\xi_1+\xi_2} \\
1 & 1 & 1
\end{vmatrix}
\\
&= \frac{-2}{\Omega_{KdV}} \begin{vmatrix}
-C_\alpha \xi_1 |\xi_1|^\alpha + \frac{\eta_1^2}{\xi_1} & -C_\alpha \xi_2 |\xi_2|^\alpha + \frac{\eta_2^2}{\xi_2} & - C_\alpha |\xi_1+\xi_2|^\alpha (\xi_1+\xi_2) + \frac{(\eta_1+\eta_2)^2}{\xi_1 + \xi_2} \\
\eta_1 & \eta_2 & \eta_1 + \eta_2 \\
\xi_1 & \xi_2 & \xi_1 + \xi_2
\end{vmatrix}
.
\end{split}
\end{equation*}
\normalsize
where $\Omega_{KdV} = \xi_1\xi_2(\xi_1+\xi_2)$. Now we subtract the first column and the second column from the third column to find
\small
\begin{equation*}
\begin{split}
&\quad \begin{vmatrix}
-C_\alpha \xi_1 |\xi_1|^\alpha + \frac{\eta_1^2}{\xi_1} & -C_\alpha \xi_2 |\xi_2|^\alpha + \frac{\eta_2^2}{\xi_2} & - C_\alpha |\xi_1+\xi_2|^\alpha (\xi_1+\xi_2) + \frac{(\eta_1+\eta_2)^2}{\xi_1 + \xi_2} \\
\eta_1 & \eta_2 & \eta_1 + \eta_2 \\
\xi_1 & \xi_2 & \xi_1 + \xi_2
\end{vmatrix}
\\
&=
\begin{vmatrix}
-C_\alpha \xi_1 |\xi_1|^\alpha + \frac{\eta_1^2}{\xi_1} & -C_\alpha \xi_2 |\xi_2|^\alpha + \frac{\eta_2^2}{\xi_2} & - C_\alpha \Omega_{\alpha,1} + \frac{(\eta_1+\eta_2)^2}{\xi_1 + \xi_2} - \frac{\eta_1^2}{\xi_1} - \frac{\eta_2^2}{\xi_2} \\
\eta_1 & \eta_2 & 0 \\
\xi_1 & \xi_2 & 0
\end{vmatrix}
\\
&= (\eta_1 \xi_2 - \eta_2 \xi_1) (- C_\alpha \Omega_{\alpha,1} + \frac{(\eta_1+\eta_2)^2}{\xi_1 + \xi_2} - \frac{\eta_1^2}{\xi_1} - \frac{\eta_2^2}{\xi_2}) \\
&= -(\eta_1 \xi_2 - \eta_2 \xi_1) \big( C_\alpha \Omega_{\alpha,1} + \frac{(\eta_1 \xi_2 - \eta_2 \xi_1)^2}{(\xi_1+\xi_2) \xi_1 \xi_2} \big).
\end{split}
\end{equation*}
\normalsize

This can be summarized as:
\begin{equation*}
\begin{split}
&\quad |\mathfrak{n}(\xi_1',\eta_1') \wedge \mathfrak{n}(\xi_2',\eta_2') \wedge \mathfrak{n}(\xi_1'+\xi_2',\eta'_1+\eta_2') | \\
&= \frac{2 |\eta_1' \xi_2' - \eta_2' \xi_1'|}{|\xi_1' \xi_2' (\xi_1' + \xi_2')|} \big| C_\alpha \Omega_{\alpha,1} (\xi') + \frac{(\eta_1' \xi_2' - \eta_2' \xi_1')^2}{\xi_1' \xi_2' (\xi_1'+\xi_2')} \big|.
\end{split}
\end{equation*}
Note that $\Omega_{\alpha,1}$ and $\frac{(\eta_1' \xi_2' - \eta_2' \xi_1')^2}{\xi_1' \xi_2' (\xi_1' + \xi_2')}$ have the same sign and by the resonance condition a comparable modulus:
\begin{equation*}
|\mathfrak{n}(\xi_1',\eta_1') \wedge \mathfrak{n}(\xi_2',\eta_2') \wedge \mathfrak{n}(\xi_1'+\xi_2',\eta_1'+\eta_2')| \sim \frac{|\eta_1' \xi_2' - \eta_2' \xi_1'|}{|\xi_1' \xi_2' (\xi_1'+\xi_2')|} \cdot \frac{N_2}{N_1} \sim \frac{N_2}{N_1}.
\end{equation*}
The extension off the convolution constraint for perturbations of size $c(N_2/N_1)$ is carried out by checking $|\partial_{\xi'} \mathfrak{n}(\xi_3',\eta_3')| + |\partial_{\eta'} \mathfrak{n}(\xi_3',\eta_3')| \lesssim 1$. This is immediate from $|\xi_3'| \sim 1$ and $|\eta_3'| \lesssim 1$.

\medskip

From $\mathfrak{n} \in C^2(\R^2 \backslash (\{0\} \times \R))$ the H\"older regularity assumption \ref{assumption:Surfaces} (ii) with $\beta=1$ follows. We can allow for crude estimates since the constant in the nonlinear Loomis-Whitney estimate does not depend on $b$ and $\beta$.

\medskip

Now we can formulate the trilinear estimate in the resonant case on different domains:

\begin{proposition}
Let $2^{\N_0} \ni N_1 \sim N_3 \gtrsim N_2 \in 2^{\Z}$, $L_i \in 2^{\N_0}$, and $f_{i,N_i,L_i} \in L^2(\R^3)$ with $\text{supp}(f_{i,N_i,L_i}) \subseteq D_{\alpha,N_i,\leqslant L_i}$, $i=1,2,3$. Suppose that $L_{\max} \ll N_1^\alpha N_2$. Then the following estimate holds:
\begin{equation}
\label{eq:LoomisWhitneyEuclidean}
\big| \int_{\R^3} (f_{1,N_1,L_1} * f_{2,N_2,L_2} ) f_{3,N_3,L_3} d \xi d\eta d\tau \big| \lesssim N_1^{\frac{1}{2}-\frac{3 \alpha}{4}} N_2^{-\frac{1}{2}} \prod_{i=1}^3 L_i^{\frac{1}{2}} \| f_{i,N_i,L_i} \|_{L^2_{\tau,\xi,\eta}(\R^3)}.
\end{equation}

\medskip

Under the above assumptions, if $f_{i,N_i,L_i} \in L^2(\R \times \R \times \Z/\gamma)$ for some $\gamma \in (1/2,1]$, then the following estimate holds:
\begin{equation}
\label{eq:LoomisWhitneyCylinder}
\big| \int \big( f_{1,N_1,L_1} * f_{2,N_2,L_2} \big) f_{3,N_3,L_3} d \xi (d \eta)_\gamma d \tau \big| \lesssim C(\underline{N},\underline{L})\prod_{i=1}^3 \| f_{i,N_i,L_i} \|_{L^2_{\tau, \xi, \eta}(\R \times \R \times \Z/\gamma)}
\end{equation}
with
\begin{equation*}
C(\underline{N},\underline{L}) = \big( \frac{N_1}{N_2} \big)^{\frac{1}{2}} (L_{\min} / N_1^{\alpha/2})^{\frac{1}{2}} \langle L_{\text{med}} / N_1^{\alpha/2} \rangle^{\frac{1}{2}} (L_{\max} / N_1^{\alpha/2} )^{\frac{1}{2}}.
\end{equation*}

\medskip

If $2^{\N_0} \ni N_1 \sim N_3 \gtrsim N_2 \in 2^{\N_0}$, $L_i$ like above, and $f_{i,N_i,L_i} \in L^2(\R \times \Z \times \Z/\gamma)$ with $\text{supp}(f_{i,N_i,L_i}) \subseteq D_{\alpha, N_i, \leqslant L_i}$, $i=1,2,3$. Then the following estimate holds:
\begin{equation}
\label{eq:LoomisWhitneyLattice}
\big| \int \big( f_{1,N_1,L_1} * f_{2,N_2,L_2} \big) f_{3,N_3,L_3} (d \xi)_1 (d\eta)_1 d\tau \big| \lesssim C(\underline{N},\underline{L}) \prod_{i=1}^3 \| f_{i N_i, L_i} \|_{L^2_\tau L_{\xi,\eta}^2(\Z \times \Z/\gamma)}.
\end{equation}
with
\begin{equation*}
C(\underline{N},\underline{L}) = L_{\min}^{\frac{1}{2}} \langle L_{\text{med}} / N^{\alpha-1} \rangle^{\frac{1}{2}} \langle L_{\max} / N^{\alpha - 1} \rangle^{\frac{1}{2}}.
\end{equation*}
\end{proposition}

\begin{remark}
The estimate \eqref{eq:LoomisWhitneyEuclidean} recovers \cite[Lemma~4.2]{SanwalSchippa2023}. We shall be brief in the proof.
For $\alpha = 2$, \eqref{eq:LoomisWhitneyCylinder} recovers Robert's estimate \cite[Proposition~5.7]{Robert2018}. 
Moreover, \eqref{eq:LoomisWhitneyLattice} recovers the estimate of Zhang \cite{Zhang2016}.
\end{remark}

\begin{proof}
We start with the proof of the estimate in Euclidean space. To ease notation, let $N = N_1$. Firstly, we decompose $f_{i,j_i,k_i} = \sum_c f^c_{i,j_i,k_i}$ into $\sim \langle L_i \rangle$ functions, which are in the $1$-neighbourhood of a translation of the characteristic surface.  Write below $g_{i,N_i,L_i} = f^c_{i,N_i,L_i}$ to ease notation. We use the anisotropic rescaling to find
\begin{equation*}
\begin{split}
&\quad \int (g_{1,N_1,L_1} * g_{2,N_2,L_2} ) g_{3,N_3,L_3} d \tau d \xi d \eta \\
&= N^{2(\alpha+1)} N^2 N^{\alpha + 2} \int (g'_{1,N_1,L_1} * g'_{2,N_2,L_2} ) g'_{3,N_3,L_3} d \tau' d \xi' d \eta'.
\end{split}
\end{equation*}
This makes the rescaled functions $f^c_{i,N_i,L_i}$ (or a translate thereof) supported in the $N^{-(\alpha+1)}$-neighbourhood of the characteristic surface. After the almost orthogonal decomposition from above, we can suppose that $|\eta_i'| \lesssim 1$. We have computed for the transversality:
\begin{equation*}
|\mathfrak{n}(\xi_1',\eta_1') \wedge \mathfrak{n}(\xi_2',\eta_2') \wedge \mathfrak{n}(\xi_3',\eta_3') | \sim \frac{N_2}{N_1}.
\end{equation*}
In the above $(\xi_i',\eta_i')$ are contained in balls of size $c(N_2 / N_1)$, which constitute an almost orthogonal decomposition of the spatial frequency support.

\medskip

Hence, we can apply Theorem \ref{thm:GenearlNLW} with Lebesgue measure and $\varepsilon = N^{-(\alpha+1)}$ to find
\begin{equation*}
| \int (g'_{1,N_1,L_1} * g'_{2,N_2,L_2} ) g'_{3,N_3,L_3} d \tau' d \xi' d \eta' | \lesssim \big( \frac{N_1}{N_2} \big)^{\frac{1}{2}} N^{-\frac{3(\alpha+1)}{2}} \prod_{i=1}^3 \| g'_{i,N_i,L_i} \|_{L^2}. 
\end{equation*}
The claim follows from reversing the scaling \eqref{eq:AnisotropicScaling}, which yields a factor of\\ $N^{-\frac{3(\alpha+1)}{2}} N^{-\frac{3}{2}} N^{-\frac{3(\alpha+2)}{4}}$, and summing over $c_i$ with $\# c_i \sim L_i$ with Cauchy-Schwarz. This finishes the proof in the Euclidean case.

\smallskip

We turn to the cylinder case.
First we suppose that $L_{\text{med}} \ll N_1^{\alpha/2}$. If $L_{\max} \leqslant N_1^{\alpha/2}$, we do not decompose. If $L_{\max} \geqslant N_1^{\alpha/2}$, we decompose $f_{i,N_i,L_i}$ with $L_i = L_{\max}$ into layers having modulation size $N_1^{\alpha/2}$ such that in the following we suppose $L_{\max} \lesssim N_1^{\alpha/2}$. We use the anisotropic scaling 
\begin{equation*}
\tau \to \tau' = \frac{\tau}{N_1^{\alpha+1}}, \quad \xi \to \xi' = \frac{\xi}{ N}, \quad \eta \to \eta'=\frac{\eta}{N_1^{\frac{\alpha}{2}+1}},
\end{equation*}
which leads to
\begin{equation*}
\begin{split}
&\quad \big| \int_{\R \times \R \times \Z/\gamma} (f_{1,N_1,L_1} * f_{2,N_2,L_2} ) f_{3,N_3,L_3} d \xi (d \eta) d\tau \big| \\ &= N_1^2 N_1^{2(\alpha+1)} \big| \int_{\R \times \R \times \Z/(\gamma N_1^{\frac{\alpha}{2}+1})} (f'_{1,N_1,L_1} * f'_{2,N_2,L_2} ) f_{3,N_3,L_3}' d\xi' (d\eta') d\tau' \big|.
\end{split}
\end{equation*}
Write $L_i' = L_i / N_1^{\alpha+1}$. We carry out the Galilean transform
\begin{equation*}
(\xi'',\eta'') = (\xi',\eta' - A' \xi')
\end{equation*}
with $A' = N_1^{-\frac{\alpha}{2}} A$ such that given $\xi''$ we have $\eta'' + A' \xi'' \in \Z / N_1^{\frac{\alpha}{2}+1}$. Moreover, we have $|\eta''| \lesssim 1$ and we can carry out like above an almost orthogonal decomposition of $(\xi'',\eta'')$ into balls of size $c(N_2/N_1)$.

Next, we decompose the supports of $f'_{i,N_i,L_i}$ into balls $\theta_i$ of size $L_{\max}' = \frac{L_{\max}}{N_1^{\alpha+1}} \lesssim N_1^{-\frac{\alpha}{2} - 1}$. 

Then, we obtain
\begin{equation}
\label{eq:DecompositionThetaCylinderI}
\begin{split}
&\quad \sum_{\theta_3} \big| \int \sum_{\theta_1 \sim_{\theta_3} \theta_2} \big( f'_{1,\theta_1} * f'_{2,\theta_2} \big) f'_{3,\theta_3} \big| \\
&\leqslant \sum_{\theta_3} \sum_{\theta_1 \sim_{\theta_3} \theta_2} \| f'_{1,\theta_1} * f'_{2,\theta_2} \|_{L^2_{\tau',\xi',\eta'}} \| f'_{3,\theta_3} \|_{L^2_{\tau',\xi',\eta'}}.
\end{split}
\end{equation}

By symmetry we can suppose that $L_1' = L_{\min}'$. We estimate the convolution with the Cauchy-Schwarz inequality. Let $I_{L_{\max}'}$ denote the interval of length $L'_{\max} = L_{\max}/(N_1^{\alpha+1})$, which contains $\pi_{\eta''}(\text{supp}(f'_{1,\theta_1}))$. We let $J_{L_{\max}'}$ denote the interval which contains $\pi_{\xi''}(\text{supp}(f'_{1,\theta_1})$. We count the number of $\eta'$ such that $\eta' - A' \xi' \in I_{L_{\max}'}$ for $\xi'' \in I'_{L_{\max}'}$. This is similar to the proof of Proposition \ref{prop:StrichartzCylinder}:
\begin{equation*}
\# \{ \eta' \in \Z/(N_1^{\frac{\alpha}{2}+1}) : (\xi'',\eta'') \in I_{L_{\max}'} \times I'_{L_{\max}'} \} \lesssim \langle A' \cdot L'_{\max} \cdot N_1^{\frac{\alpha}{2}+1} \rangle \sim \langle A' \cdot \frac{L_{\max}}{N_1^{\frac{\alpha}{2}}} \rangle.
\end{equation*}
Now, for fixed $\eta'$ from the above set, we estimate 
\begin{equation*}
\text{meas}( \{ \xi'' \in I_{L_{\max}'} : (\xi'',\eta'') \in I'_{L'_{\max}} \times I_{L_{\max}} \} \lesssim \frac{L'_{\max}}{A'}.
\end{equation*}
Given $(\xi'',\eta'')$ from above, we finally have
\begin{equation*}
\text{meas} ( \{ \tau' \in \R : (\tau',\xi'',\eta'') \in \theta_1 \} ) \lesssim L'_{\min}.
\end{equation*}
For this reason we find
\begin{equation*}
\| f'_{1,\theta_1} * f'_{2,\theta_2} \|_{L^2_{\tau',\xi',\eta'}} \lesssim (L_{\min}' (L'_{\max}/A') \langle A' \cdot \frac{L_{\max}}{N_1^{\frac{\alpha}{2}}} \rangle)^{\frac{1}{2}} \| f'_{1,\theta_1} \|_{L^2_{\tau',\xi',\eta'}} \| f'_{2,\theta_2} \|_{L^2_{\tau',\xi',\eta'}}.
\end{equation*}

%
By two more applications of the Cauchy-Schwarz inequality to carry out the summation over $\theta_i$, we incur a factor of $A^{\frac{1}{2}} \sim \Big(\frac{N_1}{N_2}\Big)^{\frac{1}{2}}$ and obtain
\begin{equation*}
\begin{split}
&\quad \sum_{\theta_3} \sum_{\theta_1 \sim_{\theta_3} \theta_2} \| f'_{1,\theta_1} * f'_{2,\theta_2} \|_{L^2_{\tau',\xi',\eta'}} \| f'_{3,\theta_3} \|_{L^2_{\tau',\xi',\eta'}} \\
&\lesssim (L'_{\min} L'_{\max})^{\frac{1}{2}} \sum_{\theta_3} \sum_{\theta_1 \sim_{\theta_3} \theta_2} \| f'_{1,\theta_1} \|_{L^2_{\tau',\xi',\eta'}} \| f'_{2,\theta_2} \|_{L^2_{\tau',\xi',\eta'}} \| f'_{3,\theta_3} \|_{L^2_{\tau',\xi',\eta'}} \\
&\lesssim (L'_{\min} L_{\max}')^{\frac{1}{2}} (N_1 /N_2)^{\frac{1}{2}} \prod_{i=1}^3 \| f'_{i,N_i,L_i} \|_{L^2_{\tau',\xi',\eta'}}.
\end{split}
\end{equation*}

Now we reverse the scaling which yields a factor $N_1^{-\frac{3}{2}} N_1^{-\frac{3}{2}(\alpha+1)}$. Taking into account the scaling factor from above gives
\begin{equation*}
N_1^2 N_1^{2(\alpha+1)} \big( \frac{N_1}{N_2} \big)^{\frac{1}{2}} N_1^{-\frac{3}{2}} N_1^{-\frac{3}{2}(\alpha+1)} N_1^{-(\alpha+1)} (L_{\min} L_{\max})^{\frac{1}{2}} = N_1^{\frac{1}{2}-\frac{\alpha}{2}} N_2^{-\frac{1}{2}} (L_{\min} L_{\max})^{\frac{1}{2}}.
\end{equation*}
This proves \eqref{eq:LoomisWhitneyCylinder} in case $L_{\text{med}} \ll N_1^{\alpha/2}$.

\medskip

We turn to the case $L_{\text{med}} \gtrsim N_1^{\alpha/2}$ and suppose that $L_{\min} \ll N_1^{\alpha/2}$. We decompose the modulation of $f_{i_j}$ with $L_{i_1} = L_{\text{med}}$ and $L_{i_2} = L_{\text{max}}$ into layers of thickness $N_1^{\frac{\alpha}{2}}$, such that we can apply the previous result with $L_{\text{med}} \sim L_{\max} \sim N_1^{\alpha/2}$. The claim follows then from the Cauchy-Schwarz inequality which incurs factors
\begin{equation*}
(L_{\text{med}} / N_1^{\alpha/2})^{\frac{1}{2}} (L_{\text{max}} / N_1^{\alpha/2})^{\frac{1}{2}}.
\end{equation*}

Finally, we turn to the case $L_{\min} \gtrsim N_1^{\alpha/2}$. In this case we decompose all functions into layers of modulation with size $N_1^{\alpha/2}$. Then we can apply the previous arguments and finally, we can apply the Cauchy-Schwarz inequality, which incurs a factor of
\begin{equation*}
\big( \prod_{i=1}^3 L_i \big)^{\frac{1}{2}} / N_1^{\frac{3 \alpha}{2}}.
\end{equation*}
This completes the proof of the Loomis-Whitney inequality on the cylinder.

\medskip

For the domain $\R \times \Z \times \Z/\gamma$ we additionally use a bilinear Strichartz estimate as simple applications of Cauchy-Schwarz inequality do not seem to suffice; see Remark \ref{rem:CauchySchwarz}.
Decompose the functions $f_{i,N_i,L_i}$ into functions $g_{i,N_i,L_i}$ with modulation of thickness $N_1^{\alpha-1}$ such that after rescaling the functions have modulation $\lesssim N_1^{-2}$. Then it follows like in the proof of Theorem \ref{thm:GenearlNLW}
\begin{equation}
\label{eq:DecompositionLWPeriodic}
\begin{split}
&\quad \int (g_{1,N_1,L_1} * g_{2,N_2,L_2} ) g_{3,N_3,L_3} d \tau d\xi d\eta \\
 &= N^{2(\alpha+1)} \int_{\R \times \Z / N \times \Z / N^2} (g'_{1,N_1,L_1} * g'_{2,N_2,L_2} ) g'_{3,N_3,L_3} d \tau' d\xi' d\eta' \\
&\leqslant N^{2(\alpha+1)} \sum_{\theta_3} \| ( g'_{1,N_1,L_1} * g'_{2,N_2,L_2} )_{\theta_3} \|_{L^2_\tau L^2_{\xi,\eta}} \| g'_{3,N_3,L_3,\theta_3} \|_{L^2_\tau L^2_{\xi,\eta}}.
\end{split}
\end{equation}
Here $\theta_3$ denote $N^{-1}$-balls.

We have by the bilinear Strichartz estimate in the resonant case and an almost orthogonal decomposition into $N^{-1}$-balls denoted by $\theta_1$, $\theta_2$:
\begin{equation*}
\begin{split}
&\quad \| ( g'_{1,N_1,L_1} * g'_{2,N_2,L_2} )_{\theta_3} \|_{L^2_\tau L^2_{\xi,\eta}} \\
&\lesssim (L_1')^{\frac{1}{2}} (1+ (L_2')^{\frac{1}{2}}/2^{1/2}) \sum_{\theta_1 \sim_{\theta_3} \theta_2} \| g'_{1,N_1,L_1,\theta_1} \|_{L^2_\tau L^2_{\xi,\eta}} \| g'_{2,N_2,L_2,\theta_2} \|_{L^2_{\tau} L^2_{\xi,\eta}}.
\end{split}
\end{equation*}
Now we apply the Cauchy Schwarz inequality in $\theta_3$, which incurs a factor of $\langle L'_2 A N_1 \rangle \sim 1$ by $L'_2 \lesssim N_2 N_1^{-2}$ and $A \sim N_1 /N_2$. For this reason we find with $L_1' \lesssim N^{-2}$:
\begin{equation*}
\begin{split}
&\quad N^{2(\alpha+1)} \sum_{\theta_3} \| ( g'_{1,N_1,L_1} * g'_{2,N_2,L_2} )_{\theta_3} \|_{L^2_\tau L^2_{\xi,\eta}} \| g'_{3,N_3,L_3,\theta_3} \|_{L^2_\tau L^2_{\xi,\eta}} \\
&\lesssim (L_1')^{\frac{1}{2}} N^{2 \alpha + 1} \prod_{i=1}^3 \| g'_{i,N_i,L_i} \|_{L^2}.
\end{split}
\end{equation*}
Scaling back we obtain
\begin{equation*}
N^{2\alpha +1} \prod_{i=1}^3 \| g'_{i,N_i,L_i} \|_{L^2} \lesssim N^{\frac{\alpha-1}{2}} \prod_{i=1}^3 \| g_{i,N_i,L_i} \|_{L^2}.
\end{equation*}
Finally, we need to carry out the summation over thin modulation layers, which gives by the above
\begin{equation*}
\eqref{eq:DecompositionLWPeriodic} \lesssim  L_{\min}^{\frac{1}{2}} \langle L_{\text{med}} / N^{\alpha-1} \rangle^{\frac{1}{2}} \langle L_{\max} / N^{\alpha-1} \rangle^{\frac{1}{2}} \prod_{i=1}^3 \| f_{i N_i, L_i} \|_{L^2_\tau L^2}.
\end{equation*}

The proof is complete.
\end{proof}

\begin{remark}
\label{rem:CauchySchwarz}
On $\T^2_\gamma$ the decomposition to $\varepsilon = N^{-2}$ balls does not recover Zhang's estimate \cite{Zhang2016}. The computation is carried out for the KP-I equation for the sake of illustration.
We decompose $f_{i,N_i,L_i}$ into layers of modulation with thickness $N_1$ such that after the anisotropic rescaling
\begin{equation*}
\tau \to \tau / N^3, \; \xi \to \xi/N, \; \eta \to \eta/N^2
\end{equation*}
we find functions with modulation $\lesssim N^{-2} = \varepsilon$. We compute
\begin{equation*}
\begin{split}
&\quad \int_{\R \times \Z^2} \big( f_{1,N_1,L_1} * f_{2,N_2,L_2} ) f_{3,N_3,L_3} d \tau d \xi d\eta \\
&= N^6 \int_{\R \times \Z / N \times \Z / N^2} (f'_{1,N_1,L_1} * f'_{2,N_2,L_2} ) f'_{3,N_3,L_3} d \tau d\xi d\eta \\
&= N^6 \sum_{\theta_3 : N^{-2} - \text{ball}} \int \big( f'_{1,N_1,L_1} * f'_{2,N_2,L_2} )_{\theta_3} f'_{3,N_3,L_3,\theta_3} d\tau' d\xi' d\eta' \\
&\leqslant N^6 \sum_{\theta_3} \| (f'_{1,N_1,L_1} * f'_{2,N_2,L_2} )_{\theta_3} \|_{L^2_\tau L^2_{\xi,\eta}} \| f'_{3,N_3,L_3,\theta_3} \|_{L^2_\tau L^2_{\xi,\eta}}.
\end{split}
\end{equation*}
Applying Cauchy-Schwarz yields
\begin{equation*}
\| (f'_{1,N_1,L_1} * f'_{2,N_2,L_2})_{\theta_3} \|_{L^2_\tau L^2_{\xi,\eta}} \lesssim N^{-1} \sum_{\theta_1 \sim_{\theta_3} \theta_2} \| f'_{1,N_1,L_1,\theta_1} \|_{L^2} \| f'_{2,N_2,L_2,\theta_2} \|_{L^2}.
\end{equation*}
In this case, the application of Cauchy-Schwarz only yields a factor $\varepsilon^{\frac{1}{2}}= N^{-1}$ because we are dealing with two counting measures. The summation of $\theta_3$-balls incurs a factor of $A^{\frac{1}{2}} \sim \Big(\frac{N_1}{N_2}\Big)^{\frac{1}{2}}$ such that
\begin{equation*}
\begin{split}
&\lesssim N^6 N^{-1} \big( \frac{N_1}{N_2} \big)^{\frac{1}{2}} \prod_{i=1}^3 \| f'_{i,N_i,L_i} \|_{L^2_\tau L^2_{\xi,\eta}} \\
&\lesssim N^6 N^{-1} \big( N_1 / N_2 \big)^{\frac{1}{2}} N^{-\frac{9}{2}} \prod_{i=1}^3 \| f_{i,N_i,L_i} \|_{L^2_\tau L^2_{\xi,\eta}} \\
&\lesssim N_2^{-\frac{1}{2}} N_1 \prod_{i=1}^3 \langle L_i /N \rangle^{\frac{1}{2}} \| f_{i,N_i,L_i} \|_{L^2_\tau L^2_{\xi,\eta}}.
\end{split}
\end{equation*}
This leaves a gap in the case of $N_2 \ll N_1$ compared to the result of Zhang \cite{Zhang2016}. 
\end{remark}

\section{Ill-posedness results for KP-I equations}
\label{section:C2IllPosedness}
In this section, we show that \eqref{eq:FKPI} posed on $\D \in \{ \R^2, \R \times \T, \T^2 \}$ is $C^2$ ill-posed for different choices of $\alpha$, namely one cannot use Picard iteration (the fixed point argument) to solve the problem. On $\R^2$ we prove the following optimal result, which improves on the argument from \cite{SanwalSchippa2023}. Moreover, we show that on $\R \times \T$ for semilinear KP-I equations, the regularity for local well-posedness is strictly subcritical.

\subsection{Sharp $C^2$-illposedness for KP-I equations on $\R^2$}  We prove that \eqref{eq:FKPIR2} is not analytically well-posed in $H^{s_1,s_2}$ for $\alpha <\frac{5}{2}$.

First, we state a preliminary result:
\begin{lemma}
	\label{lem:ComputationFT}
	 The spatial Fourier transform of
\begin{equation*}
v(t) = \int_0^t U_\alpha(t-s) \partial_x (U_\alpha(s) \phi_1 U_\alpha(s) \phi_2)ds
\end{equation*} 
is given by
	\begin{equation}
		\begin{split}
		\hat{v}(t,\xi,\eta) = \xi e^{it\omega_{\alpha}(\xi,\eta)} \int_{*} \frac{1-e^{-it\Omega_{\alpha}(\xi_1,\xi_2,\eta_1,\eta_2)}}{\Omega_{\alpha}(\xi_1,\xi_2,\eta_1,\eta_2)}\hat{\phi}_1(\xi_1,\eta_1)\hat{\phi}_2(\xi_2,\eta_2)d\xi_1 d\eta_1,
		\end{split}
	\end{equation}
where $*$ denotes the convolution constraint $(\xi,\eta ) = (\xi_1,\eta_1) + (\xi_2,\eta_2)$ with $(\xi_i,\eta_i) \in D_i$.
\begin{proof}
	With $(\xi,\eta) = (\xi_1,\eta_1) + (\xi_2,\eta_2)$, we have 
	\begin{equation*}
		\begin{split}
		&\mathcal{F}_{x,y}\Big(U_{\alpha}(t-s)\partial_x\big(U_{\alpha}(s)\phi_1 U_{\alpha}(s)\phi_2\big)\Big) (t,\xi,\eta)\\
		 &~~=e^{i(t-s)\Omega_{\alpha}(\xi,\eta)}(i\xi)\mathcal{F}_{x,y}\big(U_{\alpha}(s)\phi_1U_{\alpha}(s)\phi_2\big)(t,\xi,\eta)\\
			&~~=e^{i(t-s)\Omega_{\alpha}(\xi,\eta)}(i\xi) \int_{*} e^{is\omega_{\alpha}(\xi_1,\eta_1)}\hat{\phi}(\xi_1,\eta_1) e^{is\omega_{\alpha}(\xi_2,\eta_2)} \hat{\phi}_2(\xi_2,\eta_2)d\xi_1d\eta_1\\
						&~~=e^{it\Omega_{\alpha}(\xi,\eta)} (i\xi)\int_{*} e^{is(\omega_{\alpha}(\xi_1,\eta_1)+\omega_{\alpha}(\xi_2,\eta_2)-\omega_{\alpha}(\xi,\eta))} \hat{\phi}_1(\xi_1,\eta_1)\hat{\phi}_2(\xi_2,\eta_2)d\xi_1d\eta_1\\
						&~~=i\xi e^{it\Omega_{\alpha}(\xi,\eta)} \int_{*} e^{-is\Omega_{\alpha}(\xi_1,\xi_2,\eta_1,\eta_2)}\hat{\phi}_1(\xi_1,\eta_1)\hat{\phi}_2(\xi_2,\eta_2)d\xi_1d\eta_1.
			\end{split}
	\end{equation*}
This implies after integrating in time
\begin{equation*}
	\begin{split}
		\hat{v}(t,\xi,\eta) &= i\xi e^{it\Omega_{\alpha}(\xi,\eta)} \int_0^t \int_{*} e^{-is\Omega_{\alpha}(\xi_1,\xi_2,\eta_1,\eta_2)}\hat{\phi}_1(\xi_1,\eta_1)\hat{\phi}_2(\xi_2,\eta_2)d\xi_1d\eta_1 ds\\
		&=\xi e^{it\Omega_{\alpha}(\xi,\eta)} \int_{*} \frac{1-e^{-it\Omega_{\alpha}(\xi_1,\xi_2,\eta_1,\eta_2)}}{\Omega_{\alpha}(\xi_1,\xi_2,\eta_1,\eta_2)} \hat{\phi}_1(\xi_1,\eta_1)\hat{\phi}_2(\xi_2,\eta_2)d\xi_1d\eta_1.
	\end{split}
\end{equation*}
\end{proof}
\end{lemma}

\begin{theorem}
\label{thm:C2IllposednessR2}
Let $1 \leqslant \alpha < \frac{5}{2}$, $(s_1,s_2)\in \R^2$. Then, there does not exist a $T>0$ such that there is a function space $X_T \hookrightarrow C([-T,T];H^{s_1,s_2}(\R^2;\R))$ in which \eqref{eq:FKPI} admits a local solution with a $C^2$-differentiable flow-map $\Gamma_t$:
	\begin{equation*}
			\Gamma_t:H^{s_1,s_2}(\R^2;\R) \to H^{s_1,s_2}(\R^2;\R), \quad u_0 \mapsto u(t), ~~t \in [-T,T].
	\end{equation*}
\end{theorem}

\begin{proof}	
We recall that the resonance function is given by
\begin{equation*}
	\Omega_{\alpha}(\xi_1,\xi_2,\eta_1,\eta_2) = \Omega_{\alpha}^1(\xi_1,\xi_2) - \Omega_{\alpha}^2(\xi_1,\xi_2,\eta_1,\eta_2),
\end{equation*}
where 
\begin{equation*}
	\begin{split}
	\Omega_{\alpha}^1(\xi_1,\xi_2) &= |\xi_1+\xi_2|^{\alpha}(\xi_1+\xi_2) - |\xi_1|^{\alpha}\xi_1 - |\xi_2|^{\alpha}\xi_2,\\
	\Omega_{\alpha}^2(\xi_1,\xi_2,\eta_1,\eta_2) &= \frac{(\eta_1\xi_2-\eta_2\xi_1)^2}{\xi_1\xi_2(
		\xi_1+\xi_2)}.
	\end{split}
\end{equation*}

We define functions via their Fourier transforms as follows:
\begin{equation}
	\begin{split}
	\hat{\phi}_1(\xi_1,\eta_1) &= \frac{1_{D_1}(\xi_1,\eta_1)}{N^{-\frac{\alpha}{2}+\frac{1}{4}}}, \quad D_1 = [N^{-\frac{\alpha-1}{2}}, 2N^{-\frac{\alpha-1}{2}}] \times [-N^{-\frac{\alpha}{2}}, N^{-\frac{\alpha}{2}}],\\
	\hat{\phi}_2(\xi_2,\eta_2) &= \frac{1_{D_2}(\xi_2,\eta_2)}{N^{s_1 +(1+\frac{\alpha}{2})s_2}N^{-\frac{\alpha}{2}+\frac{1}{4}}}, \\
D_2 &= [N,N+ N^{-\frac{\alpha-1}{2}}]  \times [\sqrt{1+\alpha}N^{\frac{\alpha}{2}+1}, \sqrt{1+\alpha}N^{\frac{\alpha}{2}+1} +N^{-\frac{\alpha}{2}}].
	\end{split}
\end{equation}
	From \cite[Lemma~3.1]{SanwalSchippa2023}, we estimate the size of the resonance function by
	\begin{equation*}
		|\Omega_{\alpha}| \sim 1.
	\end{equation*}

In Lemma \ref{lem:ComputationFT} we computed
\begin{equation*}
	|\hat{v}(t,\xi,\eta)| \sim t \frac{N}{N^{s_1} N^{(1+\frac{\alpha}{2})s_2}}  
\end{equation*}
for a significant measure of $(\xi,\eta) = (\xi_1,\eta_1) + (\xi_2,\eta_2)$ with $(\xi_i,\eta_i) \in D_i$. 
For $0<t\leqslant c \ll 1$, the Sobolev norm of $v(t)$ is given by
\begin{equation*}
	\|v(t)\|_{H^{s_1,s_2}(\R^2)} \sim tN^{\frac{5}{4}-\frac{\alpha}{2}}.
\end{equation*}
For $\Gamma_t$ to be $C^2$-differentiable, it needs to hold
\begin{equation*}
	1 \sim \|\phi_1\|_{H^{s_1,s_2}(\R^2)} \|\phi_2\|_{H^{s_1,s_2}(\R^2)} \gtrsim t N^{\frac{5}{4}-\frac{\alpha}{2}}, 
	\end{equation*}
which requires $\alpha \geqslant \frac{5}{2}$ for $N\gg1$. Clearly, $\phi_i$ are not real-valued, but letting $u^*_0 = \phi_1 + \phi_2$ and symmetrizing the Fourier transform we find real-valued initial data $u_0$ with comparable Sobolev norm. The above estimates remain unchanged, which completes the proof.
\end{proof}

\subsection{$C^2$-illposedness for KP-I equations on $\R \times \T$}
  Next, we treat the case of the cylinder $\R \times \T$. 
\begin{theorem}
	\label{thm:C2IllposedCylinder}
	Let $\alpha <5$, $(s_1,s_2) \in \R^2$. Then there does not exist any time $T>0$ such that there is a function space $X_T \hookrightarrow C([-T,T];H^{s_1,s_2}(\R \times \T))$ in which \eqref{eq:FKPI} has a unique local solution with a $C^2$-differentiable flow map $\Gamma_t$:
	\begin{equation*}
		\Gamma_t:H^{s_1,s_2}(\R \times \T) \to H^{s_1,s_2}(\R \times \T), \quad u_0 \mapsto u(t), ~~t \in [-T,T].
	\end{equation*} 
\end{theorem}

We remark that the cases $\alpha \leqslant 1$ can be readily seen to be $C^2$-illposed by comparison with the dispersion-generalized Benjamin-Ono equation:
\begin{equation*}
\left\{ \begin{array}{cl}
\partial_t u + \partial_x D_x^\alpha u &= u \partial_x u, \quad (t,x) \in \R \times \R, \\
u(0) &= u_0 \in H^{s}(\R).
\end{array} \right.
\end{equation*}
This evolution is recovered by considering initial data to KP-equations on $\R \times \T$, which do not depend on the periodic coordinate.

In the following we suppose that $\alpha \geqslant 1$. As an ansatz to prove Theorem \ref{thm:C2IllposedCylinder}, we consider the following functions with parameters $\gamma$ and $\beta(N)$ to be determined:
\begin{equation}
	\label{eq:PreliminaryFunctions}
	\begin{split}
	\hat{\phi}(\xi_1,\eta_1) &= \gamma^{-\frac{1}{2}} 1_{D_1}(\xi_1,\eta_1), ~~D_1 = [\gamma, 2\gamma] \times \{0\},\\
	\hat{\phi}_2(\xi_2,\eta_2) &= \gamma^{-\frac{1}{2}} N^{-s_1}(\beta(N))^{-s_2} 1_{D_2}(\xi_2,\eta_2), ~~D_2 = [N-\gamma, N]\times \{\beta(N)\},
	\end{split}
\end{equation}
where $N \gg 1$, $\beta(N)\sim N^{\frac{\alpha}{2}+1}$, and $\gamma = \gamma(N) \ll 1$ will be chosen later. We find an upper bound for the size of the resonance function. With $\xi_1, \xi_2>0$, we can compute $\Omega_{\alpha}^1$ invoking the mean value theorem:
\begin{equation*}
\Omega_{\alpha}^1(\xi_1,\xi_2) = (\xi_1+\xi_2)^{\alpha+1} - \xi_2^{\alpha+1}-\xi_1^{\alpha+1}= (\alpha + 1)\xi_1\xi_*^{\alpha} - \xi_1^{\alpha+1},
\end{equation*}
for $\xi_* \in (\xi_2,\xi_2+\xi_1)$. For $\eta_1=0$ which is relevant for our setting, we have
\begin{equation*}
	\Omega_{\alpha}^2(\xi_1,\xi_2,\eta_1,\eta_2) = \frac{\eta_2^2\xi_1}{\xi_2(\xi_1+\xi_2)}=:f(\xi_1).
\end{equation*}
Using Taylor's theorem, we have
\begin{equation}
f(\xi_1) = f(\xi_1^*) + f'(\xi_1^*)(\xi_1-\xi_1^*)+R,
\end{equation}
where $R$ includes all the lower order terms. This gives
\begin{equation*}
	f(\xi_1) = f(\xi_1^*) + \frac{\eta_2^2}{\xi_2(\xi_1^*+\xi_2)}(\xi_1-\xi_1^*) -\frac{\eta_2^2(\xi_1^*)}{\xi_2(\xi_1^*+\xi_2)^2}(\xi_1 - \xi_1^*) + R.
\end{equation*}
From the definition of $D_1$ and $D_2$ in \eqref{eq:PreliminaryFunctions}, we have that the first two terms on the right-hand side above have similar size while the remainder term is smaller in size than the first two terms. Now we choose $\xi_1^*, \xi_1, \xi_2,\xi^*,\eta_2 \in \R$ such that the leading order terms in $\Omega_1^{\alpha}$ and $\Omega_2^{\alpha}$ cancel, namely
\begin{equation*}
	\begin{split}
	(\alpha+1)\xi_1\xi_*^{\alpha} - \xi_1^{\alpha+1} = f(\xi_1^*) + \frac{\eta_2^2(\xi_1-\xi_1^*)}{\xi_2(\xi_1^* + \xi_2)}
= \frac{\eta_2^2\xi_1}{\xi_2(\xi_1^* + \xi_2)},
	\end{split}
\end{equation*}
which gives
\begin{equation*}
	\eta_2^2 = \big((\alpha+1)\xi_*^{\alpha} - \xi_1^{\alpha}\big) \xi_2(\xi_1^* + \xi_2).
\end{equation*}
With this choice, for the size of the resonance function, we conclude
\begin{equation*}
	|\Omega_{\alpha}| \sim \Big| \frac{\eta_2^2 \xi_1^*(\xi_1-\xi_1^*)}{\xi_2(\xi_1^*+\xi_2)^2}\Big| \sim N^{\alpha-1}\gamma^2.
\end{equation*}
However, with the choice
\begin{equation*}
	\eta_2 =:\beta(N,\gamma) = \Big(\big((\alpha+1)\xi_*^{\alpha} - \xi_1^{\alpha}\big) \xi_2(\xi_1^* + \xi_2) \Big)^{\frac{1}{2}},
\end{equation*}
we cannot ascertain that $\beta(N,\gamma) \in \N$. To this end, we find a rational approximation of $\eta_2$. Using Dirichlet's approximation theorem, we have for $\beta(N, \gamma) \in \R$, the existence of infinitely many rational numbers $\frac{p}{q}$ such that
\begin{equation}
	| \beta(N, \gamma)-\frac{p}{q}| <\frac{1}{q^2}.
\end{equation} 
The above implies that
\begin{equation}
	\label{eq:Diophantine}
	|q\beta(N,\gamma) - p| <\frac{1}{q}.
\end{equation}
Since we have a natural number ($q\beta(N,\gamma)>0$) approximation for $q\beta(N,\gamma)$ (and not $\beta(N,\gamma)$), we make use of the anisotropic scaling \eqref{eq:AnisotropicScaling} for \eqref{eq:FKPI}: for any $\lambda \in \R$
\begin{equation*}
	\tau \to \lambda^{\alpha}\tau, \quad \xi \to \lambda \xi, \quad \eta \to \lambda^{\frac{\alpha}{2}+1}\eta. 
\end{equation*}
With this, for $q \in \N$, we introduce the following scaling
\begin{equation}
\label{eq:AnisotropicScalingQ}
	\tau \to\tau':= q^{2\frac{\alpha+1}{\alpha+2}} \tau, \quad \xi \to \xi':=q^{\frac{2}{\alpha+2}}\xi, \quad \eta \to \eta':=q\eta.
\end{equation}
It is easy to check that with the scaling \eqref{eq:AnisotropicScalingQ}, the resonance function scales as
\begin{equation*}
	\Omega_{\alpha} \to \Omega'_{\alpha} = q^{2\frac{\alpha+1}{\alpha+2}} \Omega_{\alpha}.
\end{equation*}
Now we work with the scaled variables ($\xi'_i, \eta'_i)$, $i=1,2$ and ensure that the natural number approximation \eqref{eq:Diophantine} of $q\beta(N,\gamma)$ does not affect the size of the resonance function. Consider the difference
\begin{equation*}
	\begin{split}
	 |\Omega'_{\alpha}(\xi'_1, \xi'_2, 0, q\beta(N,\gamma)) - \Omega'_{\alpha}(\xi'_1,\xi_2',0,p)|&= q^{2\frac{\alpha+1}{\alpha+2}} | \Omega_{\alpha}(\xi_1,\xi_2,0,\beta) - \Omega_{\alpha}(\xi_1,\xi_2,0,\frac{p}{q})|\\
	 &=q^{2\frac{\alpha+1}{\alpha+2}}\frac{\xi_1}{\xi_2(\xi_1+\xi_2)}\Big|\beta^2 - \frac{p^2}{q^2}\Big|\\
	 &\lesssim q^{2\frac{\alpha+1}{\alpha+2}}\frac{\gamma}{N^2} \Big| \beta-\frac{p}{q}\Big|  \Big|\beta +\frac{p}{q}\Big|\\
	 &\lesssim q^{2\frac{\alpha+1}{\alpha+2}} \frac{\gamma}{N^2} \frac{N^{\frac{\alpha}{2}+1}}{q^2}.
	 \end{split}
\end{equation*}
We require the size of the above expression to be negligible compared to \\$|\Omega'_{\alpha}(\xi'_1,\xi_2',0, q\beta(N,\gamma))|$, i.e.,
\begin{equation*}
	q^{2\frac{\alpha+1}{\alpha+2}} \frac{\gamma}{N^2} \frac{N^{\frac{\alpha}{2}+1}}{q^2} \ll q^{2\frac{\alpha+1}{\alpha+2}} N^{\alpha-1}\gamma^2
	\end{equation*}
which requires 
\begin{equation}
	\label{eq:ConditionOnq}
	q \gg N^{-\frac{\alpha}{4}} \gamma^{-\frac{1}{2}}.
	\end{equation}

Furthermore, we choose $\gamma = \gamma(N)$ such that the size of the new resonance function $\Omega'_{\alpha}$ is small, namely
\begin{equation*}
	q^{2\frac{\alpha+1}{\alpha+2}} N^{\alpha-1}\gamma^2 \sim 1,
	\end{equation*}
which gives
\begin{equation}
	\label{eq:ChoiceOfGamma}
	\gamma \sim N^{\frac{1}{2}-\frac{\alpha}{2}} q^{-\frac{\alpha+1}{\alpha+2}}.
\end{equation}
We remark that this is consistent with \eqref{eq:ConditionOnq} since taking the two conditions together
\begin{equation*}
q^{\frac{\alpha+3}{2(\alpha+2)}} \gg N^{-\frac{1}{4}}.
\end{equation*}
This is automatically satisfied for $\alpha \geqslant 1$, $N \in 2^{\N_0}$, and $q \in \N$.


With the choice of $q$ and $\gamma$ made above, we are set to prove Theorem \ref{thm:C2IllposedCylinder}.

\begin{proof}[Proof of Theorem \ref{thm:C2IllposedCylinder}]
	We define the  input functions $\phi_1$ and $\phi_2$ via their spatial Fourier transforms as follows:
\begin{equation*}
	\begin{split}
		\hat{\phi}_1(\xi_1,\eta_1) &= q^{-\frac{1}{\alpha+2}} \gamma^{-\frac{1}{2}} 1_{D_1}(\xi_1,\eta_1), \text{ where } D_1 = [q^{\frac{2}{\alpha+2}}\gamma, 2q^{\frac{2}{\alpha+2}}\gamma] \times \{0\},\\
		\hat{\phi}_2(\xi_2,\eta_2) &= q^{-\frac{1}{\alpha+2}} \gamma^{-\frac{1}{2}} (q^{\frac{2}{\alpha+2}}N)^{-s_1} p^{-s_2}1_{D_2}(\xi_2,\eta_2), \\
		&\quad \text{ where } D_2 = [q^{\frac{2}{\alpha+2}} (N-\gamma), q^{\frac{2}{\alpha+2}}N] \times \{p\}.
	\end{split}
\end{equation*}
In the above definition we choose $q$ such that $[q^{\frac{2}{\alpha+2}}\gamma, 2q^{\frac{2}{\alpha+2}}\gamma] \subseteq [0,1]$, i.e., $ q^{\frac{2}{\alpha+2}} \gamma \lesssim 1$ which follows from \eqref{eq:ChoiceOfGamma} and $\alpha \geqslant 1$.
It is straight-forward to check that $\|\phi_i\|_{H^{s_1,s_2}} \sim 1$ for $i=1,2$. We consider the contribution to the second Picard iterate given by the Duhamel integral:
\begin{equation*}
	v(t) = \int_0^t S_{\alpha}(t-s)\partial_x\big(S_{\alpha}(s)\phi_1 S_{\alpha}(s)\phi_2\big)ds.
	\end{equation*}
To show  that the flow map $\Gamma_t$ is not $C^2$ differentiable at the origin, we show that
\begin{equation*}
	\|v(t)\|_{H^{s_1,s_2}} \to \infty \text{ as } N \to \infty.
\end{equation*}
From Lemma \ref{lem:ComputationFT}, we have
\begin{equation*}
	\hat{v}(t,\xi,\eta) = \xi e^{it\omega_{\alpha}(\xi,\eta)} \int_{*} \frac{1-e^{-it\Omega_{\alpha}(\xi_1,\xi_2,\eta_1,\eta_2)}}{\Omega_{\alpha}(\xi_1,\xi_2,\eta_1,\eta_2)} \hat{\phi}_1(\xi_1,\eta_1)\hat{\phi}_2(\xi_2,\eta_2) d\xi_1d\eta_1,
\end{equation*}
where $*$ denotes the convolution constraint and we have the counting measure in the $\eta$ variable. To compute the $H^{s_1,s_2}$ norm of $v$, we estimate the size of $\hat{v}$ for $0<t\leqslant c \ll 1$:
\begin{equation*}
	|\hat{v}(t,\xi,\eta)| \sim t \frac{q^{\frac{2}{\alpha+2}}N}{(q^{\frac{2}{\alpha+2}}N)^{s_1}p^{s_2}}.
\end{equation*}
Thus,
\begin{equation*}
	\|v(t)\|_{H^{s_1,s_2}}^2 \sim \int \langle \xi\rangle^{2s_1}\langle \eta \rangle^{2s_2}|\hat{v}(t,\xi,\eta)|^2d\xi d\eta \sim t^2 q^{\frac{4}{\alpha+2}} N^2 q^{\frac{2}{\alpha+2}}\gamma.
\end{equation*}
For the flow map $\Gamma_t$ to be $C^2$, we require
\begin{equation*}
	1\sim \|\phi_1\|_{H^{s_1,s_2}} \|\phi_2\|_{H^{s_1,s_2}} \gtrsim \|v(t)\|_{H^{s_1,s_2}} \sim t q^{\frac{3}{\alpha+2}}N\gamma^{\frac{1}{2}} \sim t N^{-\frac{1}{4} \frac{5-\alpha}{\alpha+3}} N^{\frac{5-\alpha}{4}},
\end{equation*}
where we used \eqref{eq:ConditionOnq} and \eqref{eq:ChoiceOfGamma} to obtain the last term in the above display. Hence, we conclude that $\Gamma_t$ can be $C^2$-differentiable only for $\alpha \geqslant 5$.
\end{proof}

\subsection{Subcritical illposedness for semilinear KP-I equations on $\R \times \T$}

We supplement the analytic well-posedness result proved in Theorem \ref{thm:SemilinearWPRT}:

\begin{theorem}
    \label{thm:IllposednessForSmallS}
    Let $\alpha >5$, and $s<\frac{1-\alpha}{4}$. Then, for initial data in $H^{s,0}(\R \times \T,\R)$, \eqref{eq:FKPI} is ill-posed on $\R\times \T$, i.e., the data-to-solution map
    \begin{equation*}
        u_0 \mapsto u(t)
    \end{equation*}
    fails to be continuous.
    \begin{proof}
        The proof relies on \cite[Proposition~1]{BejenaruTao2006}. The starting point is an analytic data-to-solution mapping assigning more regular initial data from a ball in $H^{\bar{s},0}$, $\bar{s} \in (\frac{1-\alpha}{4},0)$ to solutions of higher regularity:
        \begin{equation*}
        		(B_{R,\bar{s}}, \| \cdot \|_{H^{\bar{s},0}}) \to (B_{R}, X_{\bar{s}}), \quad u_0 \mapsto u.
        \end{equation*}
        This is provided by Theorem \ref{thm:SemilinearWPRT}.
        
\smallskip        
        
        \cite[Proposition~1]{BejenaruTao2006} now states that 
a possible well-posedness for $s'<\frac{1-\alpha}{4}$ implies continuity of
\begin{equation*}
    \begin{split}
   (B_{R,\bar{s}}, \| \cdot \|_{H^{s',0}}) \to C_T H^{s',0}, \quad
    u_0 &\mapsto \int_0^t S_\alpha(t-s) (\partial_x(S_\alpha(s) u_0 S_\alpha(s) u_0)) ds.
    \end{split}
\end{equation*}

The above display is the second Picard iterate, and we furnish initial data for which the estimate
\begin{equation}
\label{eq:Boundedness}
    \Big \| \int_0^t S_\alpha(t-s) (\partial_x(S_\alpha(s) u_0 S_\alpha(s) u_0)) ds \Big\|_{H^{s,0}(\R \times \T)} \lesssim \|u_0\|_{H^{s,0}(\R \times \T)}^2
\end{equation}
holds only for $s\geqslant \frac{1-\alpha}{4}$. We rely on a High $\times$ High $\to$ Low resonant interaction, which is crucial to apply \cite[Proposition~1]{BejenaruTao2006} because the sequence of initial data must be contained in $(B_{R,\bar{s}}, \| \cdot \|_{H^{\bar{s},0}})$. Define 
\begin{equation*}
  \hat{u}_0(\xi,\eta) = 1_{[A,A+A^{1-\alpha}]}(\xi) \cdot \delta_0(\eta) + 1_{[1-A,1-A+A^{1-\alpha}]}(\xi) \cdot \delta_m(\eta), 
\end{equation*}
where $m \in \N$ and $1\ll A \in \R_{>0}$ is such that for some output frequencies $\xi$, we have $\xi \sim 1$. It is straightforward to check that
\begin{equation*}
    \|u_0\|_{H^{s,0}} \sim A^s A^{\frac{1-\alpha}{2}}.
\end{equation*}
Again we remark that $u_0$ is not real-valued, which can be accomplished by symmetrization of the Fourier transform like in previous sections.

\smallskip

Furthermore, choosing $m \sim A^{\frac{\alpha}{2}}$ and $A$ suitably, we have
\begin{equation*}
    |\Omega_\alpha(\xi_1,\xi-\xi_1,\eta_1,\eta-\eta_1)| \sim 1.
\end{equation*}
Indeed, it suffices to check this for one specific $\xi_1 \in A$ and $\xi \sim 1$. The estimate follows then from simple derivative estimates.

\medskip

By Lemma \ref{lem:ComputationFT} and that $|\xi|\sim 1$ we have 
\begin{equation*}
  \Big \| \int_0^t S_\alpha(t-s) (\partial_x(S_\alpha(s) u_0 S_\alpha(s) u_0)) ds \Big\|_{H^{s,0}} \sim N^{\frac{3-3\alpha}{2}}.   
\end{equation*}
For \eqref{eq:Boundedness} to hold, we require
\begin{equation*}
   N^{\frac{3-3\alpha}{2}} \lesssim N^{2s} N^{1-\alpha},
\end{equation*}
which is true only for $s\geqslant \frac{1-\alpha}{4}$. This completes the proof.
    \end{proof}
\end{theorem}

\subsection{Ill-posedness on tori}
Finally, we prove the following result on $\T \times \T_\gamma$.

\begin{theorem}
	\label{thm:C2IllposedTorus}
	Let $\alpha \in \R_{>0}$, $(s_1,s_2) \in \R^2$. Then, there is $\gamma = \gamma(\alpha) \in (1/2,1]$ such that there does not exist any time $T>0$, for which there is a function space $X_T \hookrightarrow C([-T,T];H^{s_1,s_2}(\T \times \T_\gamma))$ in which \eqref{eq:FKPI} has a unique local solution with a $C^2$-differentiable flow map $\Gamma_t$:
	\begin{equation*}
		\Gamma_t:H^{s_1,s_2}(\T \times \T_\gamma) \to H^{s_1,s_2}(\T \times \T_\gamma), \quad u_0 \mapsto u(t), ~~t \in [-T,T].
	\end{equation*}
\end{theorem}
\begin{proof}
Considering initial data, which do not depend on the second variable, the evolution becomes
\begin{equation*}
\partial_t u + \partial_x D_x^\alpha u = u \partial_x u.
\end{equation*}	
For this equation it is easy to see, choosing initial data
\begin{equation*}
\hat{u}_0(\xi) = \delta_1(\xi) + \delta_N(\xi)
\end{equation*}
for $N \gg 1$, the data-to-solution mapping fails to be $C^2$ for $\alpha < 1$. 

\medskip

In the following we suppose that $\alpha \geqslant 1$ and
consider the following functions:
		\begin{equation}
  \label{eq:PreliminaryChoiceTorus}
			\begin{split}
				\hat{\phi}_1(\xi_1,\eta_1) &=  1_{D_1}(\xi_1,\eta_1), \quad D_1 = \{ 1 \} \times \{ 0 \},\\
				\hat{\phi}_2(\xi_2,\eta_2) &= N^{-s_1} p^{-s_2} 1_{D_2}(\xi_2,\eta_2), \quad D_2 = \{ N \} \times \{ p \}.
			\end{split}
		\end{equation}
We choose $p$ such that
\begin{equation*}
\Omega_\alpha(1,0,N,p) = (N+1)^{\alpha+1} - 1 - N^{\alpha+1} - \frac{p^2}{(N+1)N} = 0.
\end{equation*}
Note that this requires a choice of $\gamma \in (\frac{1}{2},1]$ such that $p \in \gamma^{-1} \Z$, i.e., fixing the ratio of the periods.

Clearly, for the Fourier transform of the Duhamel integral
\begin{equation*}
|\hat{v}(t,\xi,\eta)| \sim \frac{N t}{N^{s_1}	 p^{s_2}}, \quad \xi = N+1, \; \eta = p.
\end{equation*}
Hence,
\begin{equation*}
\| v(t) \|_{H^{s_1,s_2}} \gtrsim N t.
\end{equation*}
This implies that for $\alpha \geqslant 1$, the data-to-solution mapping of \eqref{eq:FKPI} fails to be $C^2$ for some $\gamma = \gamma(\alpha)$.


\end{proof}

\section{Improved quasilinear local well-posedness for the KP-I equation on the Euclidean plane}
\label{section:QuasilinearLWPKPIR2}

In this section we prove the following result on local well-posedness for the KP-I equation on $\R^2$:
\begin{equation}
\label{eq:KPIR2}
\left\{ \begin{array}{cl}
\partial_t u + \partial_x^3 u - \partial_x^{-1} \partial_y^2 u &= u \partial_x u, \quad (t,x,y) \in \R \times \R^2, \\
u(0) &= u_0 \in H^{s,0}(\R^2).
\end{array} \right.
\end{equation}

\begin{theorem}
\label{thm:ImprovedQuasilinearLWPKPIR2}
\eqref{eq:KPIR2} is locally well-posed provided that $s>\frac{1}{2}$ in the following sense: The data-to-solution mapping $S_T^\infty: H^{\infty,0}(\R^2) \to C_T H^{\infty,0}(\R^2)$ extends continuously to $S_T^s: H^{s,0}(\R^2) \to C_T H^{s,0}(\R^2)$ with $T=T(\| u_0 \|_{H^{s,0}(\R^2)})$ depending lower semicontinuously on the initial data and $T(u) \gtrsim 1$ as $u \downarrow 0$.
\end{theorem}

We begin with short-time bilinear estimates stated in Proposition \ref{prop:ShorttimeBilinearKPIR2} and then show energy estimates in short-time function spaces in Proposition \ref{prop:ShorttimeEnergyIterationKPIR2}. We choose the frequency-dependent time localization $T=T(N)=N^{-1}$, which is chosen to precisely ameliorate the derivative loss in the nonlinear estimate. Moreover, we consider Sobolev spaces with low frequency weight to estimate the differences of solutions:
\begin{equation*}
\| f \|_{\bar{H}^{-\frac{1}{2}}}^2 = \int_{\R^2} (1+|\xi|^{-1})^{1+2\delta} \langle \xi \rangle^{-1} |\hat{f}(\xi,\eta)|^2 d\xi d\eta.
\end{equation*}
$\delta$ is chosen depending on $s$:
\begin{equation*}
\delta = \frac{s-\frac{1}{2}}{2} \wedge \frac{1}{8}.
\end{equation*}

The choice of frequency-dependent time localization and low frequency weights can be traced back to the work of Ionescu--Kenig--Tataru \cite{IonescuKenigTataru2008}.
Short-time function spaces to estimate the difference of the solution with low frequency weight are denoted by
$\bar{F}^{-\frac{1}{2}}$, $ \bar{B}^{-\frac{1}{2}}$.

\subsection{Outline of the proof.}
The general argument can be found already in \cite{IonescuKenigTataru2008}, so we shall be brief.
Solutions $u$ to \eqref{eq:KPIR2} are iterated in short-time function spaces as follows with $T \in (0,1]$, $s \geqslant s' > \frac{1}{2}$:
\begin{equation*}
\left\{ \begin{array}{cl}
\| u \|_{F^s(T)} &\lesssim \| u \|_{B^s(T)} + \| \partial_x(u^2) \|_{\mathcal{N}^s(T)}, \\
\| \partial_x(u^2) \|_{\mathcal{N}^s(T)} &\lesssim \| u \|_{F^s(T)} \| u \|_{F^{s'}(T)}, \\
\| u \|^2_{B^s(T)} &\lesssim \| u(0) \|^2_{H^{s,0}(\R^2)} + \| u \|^2_{F^s(T)} \| u \|_{F^{s'}(T)}.
\end{array} \right.
\end{equation*}
By a standard bootstrap argument this implies a priori estimates
\begin{equation*}
\| u \|_{F^s(1)} \lesssim \| u(0) \|_{H^{s,0}(\R^2)}
\end{equation*}
for $s \geqslant s' > \frac{1}{2}$ only depending on smallness of the $H^{s',0}(\R^2)$-norm. And by subcriticality of this norm, smallness can always be accomplished by rescaling.

\smallskip

Differences of solutions $v = u_1 - u_2$ with $u_i$ solutions to \eqref{eq:KPIR2} are propagated as follows with $s > \frac{1}{2}$ and $T \in (0,1]$:
\begin{equation*}
\left\{ \begin{array}{cl}
\| v \|_{\bar{F}^{-\frac{1}{2}}(T)} &\lesssim \| v \|_{\bar{B}^{-\frac{1}{2}}(T)} + \| \partial_x(v(u_1+u_2)) \|_{\bar{\mathcal{N}}^{-\frac{1}{2}}(T)}, \\
\| \partial_x (v (u_1+u_2)) \|_{\bar{\mathcal{N}}^{-\frac{1}{2}}(T)} &\lesssim \| v \|_{\bar{F}^{-\frac{1}{2}}(T)} (\| u_1 \|_{F^{s}(T)} + \| u_2 \|_{F^s(T)} ), \\
\| v \|^2_{\bar{B}^{-\frac{1}{2}}(T)} &\lesssim \| v(0) \|_{\bar{H}^{-\frac{1}{2}}}^2 + \| v \|_{\bar{F}^{-\frac{1}{2}}(T)}^2 ( \| u_1 \|_{F^{s}(T)} + \| u_2 \|_{F^s(T)} ).
\end{array} \right.
\end{equation*}
For small initial data this implies by limiting properties of the function spaces as $T \downarrow 0$ and another continuity argument
\begin{equation*}
\| v \|_{\bar{F}^{-\frac{1}{2}}(1)} \lesssim \| v(0) \|_{\bar{H}^{-\frac{1}{2}}}.
\end{equation*}

\smallskip

We conclude local well-posedness from invoking the Bona-Smith argument. Differences of solutions in $H^{s,0}(\R^2)$ for $s > \frac{1}{2}$ are estimated as
\begin{equation*}
\left\{ \begin{array}{cl}
\| v \|_{F^{s}(T)} &\lesssim \| v \|_{B^s(T)} + \| \partial_x(v(u_1+u_2)) \|_{F^s(T)}, \\
\| \partial_x(v(u_1+u_2)) \|_{\mathcal{N}^s(T)} &\lesssim \| v \|_{F^s(T)} (\| u_1 \|_{F^s(T)} + \| u_2 \|_{F^s(T)} ), \\
\| v \|^2_{B^s(T)} &\lesssim \| v(0) \|^2_{H^{s,0}(\R^2)} + \| v \|^3_{F^s(T)} \\
&\quad + \| v \|_{F^s(T)} \| v \|_{\bar{F}^{-\frac{1}{2}}(T)} \| u_2 \|_{F^{s+\frac{1}{2}}(T)}.
\end{array} \right.
\end{equation*}
Consider $u_0 \in H^{s,0}(\R^2)$ and the frequency-truncation $P_{\leqslant H} u_0$. The solution to \eqref{eq:KPIR2} emanating from $P_{\leqslant H} u_0$ is denoted by $u^H$.

\smallskip

We have for the difference of the solutions $v = u - u^H$:
\begin{equation*}
\begin{split}
\| v \|^2_{F^s(T)} &\lesssim \| v(0) \|^2_{H^{s,0}(\R^2)} + \| v \|^2_{F^s(T)} (\| u \|^2_{F^s(T)} + \| u^{H} \|^2_{F^{s}(T)}) \\
&\quad + \| v \|^3_{F^s(T)} + \| v \|_{F^s(T)} \| v(0) \|_{\bar{H}^{-\frac{1}{2}}} \| u^H(0) \|_{F^{2s+\frac{1}{2}}(T)}.
\end{split}
\end{equation*}
Now we plug in the a priori estimates obtained previously to find
\begin{equation*}
\begin{split}
\| v \|^2_{F^s(T)} &\lesssim \| v(0) \|^2_{H^{s,0}(\R^2)} + \| v \|^2_{F^s(T)} (\| u_0 \|^2_{H^{s,0}(\R^2)} + \| u^{H}(0) \|^2_{H^{\frac{1}{2},0}(\R^2)}) \\
&\quad + \| v \|^3_{F^s(T)} + \| v \|_{F^s(T)} \| P_{\geqslant H} u(0) \|_{H^{-\frac{1}{2},0}} \| u^H(0) \|_{H^{2s+\frac{1}{2},0}}.
\end{split}
\end{equation*}
This implies for small enough initial data in $H^{s,0}$:
\begin{equation*}
\| u - u^H \|_{F^s(T)} \lesssim \| P_{\geqslant H} u(0) \|_{H^{s,0}(\R^2)}.
\end{equation*}
Next, consider $u_{0n} \to u_0$ in $H^{s,0}(\R^2)$. We expand the difference of solutions:
\begin{equation}
\label{eq:DifferenceSolutionsEstimate}
\| u_n - u \|_{C_T H^{s,0}(\R^2)} \leqslant \| u_n - u_n^H \|_{C_T H^{s,0}} + \| u_n^H - u^H \|_{C_T H^{s,0}} + \| u_n^H - u \|_{C_T H^{s,0}}.
\end{equation}
By the above, the first and third term is estimated as
\begin{equation*}
\| u_n - u_n^H \|_{C_T H^{s,0}} + \| u_n^H - u \|_{C_T H^{s,0}} \lesssim \| u_{0n}^{\geqslant H} \|_{H^{s,0}} + \| u_0^{\geqslant H} \|_{H^{s,0}}.
\end{equation*}
We can choose $H$ large enough such that for $n \geqslant n_0$
\begin{equation*}
\| u_{0n}^{\geqslant H} \|_{H^{s,0}} + \| u_0^{\geqslant H} \|_{H^{s,0}} \leqslant \delta.
\end{equation*}
Consequently, by the local well-posedness of \eqref{eq:KPIR2} in $H^{\frac{3}{2}+\varepsilon,0}$ (see e.g. \cite[Section~8]{MolinetSautTzvetkov2007}), which is used to estimate the center term in \eqref{eq:DifferenceSolutionsEstimate}, we find
\begin{equation*}
\limsup_{n \to \infty} \| u_n - u \|_{C_T H^{s,0}(\R^2)} \lesssim \delta.
\end{equation*}
Since $\delta >0$ was arbitrary, the proof of local well-posedness is complete.

\subsection{Short-time bilinear estimates}

\begin{proposition}
\label{prop:ShorttimeBilinearKPIR2}
Let $\alpha = 2$, and $r_2 \geqslant r_1 > \frac{1}{2}$. Then, the following estimate holds:
\begin{equation}
\label{eq:ShorttimeBilinearEstimateKPIA}
\| \partial_x(u_1 u_2) \|_{\mathcal{N}^{r_2}(T)} \lesssim \| u_1 \|_{F^{r_1}(T)} \| u_2 \|_{F^{r_2}(T)} + \| u_1 \|_{F^{r_2}(T)} \| u_2 \|_{F^{r_1}(T)}.
\end{equation}
Secondly, for $s > \frac{1}{2}$ the following estimate holds:
\begin{equation}
\label{eq:ShorttimeBilinearEstimateKPIB}
\| \partial_x (uv) \|_{\bar{\mathcal{N}}^{-\frac{1}{2}}(T)} \lesssim \| u \|_{F^s(T)} \| v \|_{\bar{F}^{-\frac{1}{2}}(T)}.
\end{equation}
\end{proposition}
\begin{proof}
By definition of the function spaces, the claim reduces to proving summable estimates:
\begin{equation}
\label{eq:DyadicLocalizedBilinearEstimate}
\| P_N \partial_x( P_{N_1} u_1 P_{N_2} u_2) \|_{\mathcal{N}_N(T)} \lesssim C(N,N_1,N_2) \| P_{N_1} u_1 \|_{F_{N_1}(T)} \| P_{N_2}u_2 \|_{F_{N_2}(T)}.
\end{equation}
In the following, we reduce to estimates for functions localized in modulation taking into account the time localization. The arguments are standard (cf. \cite{IonescuKenigTataru2008,rsc2019}), and we shall be brief.

\subsubsection{Low $\times$ High $\to$ High interaction ($N_2 \ll N_1\sim N$)}
In the following we consider extensions $P_{N_i} \tilde{u}_i$ of $P_{N_i} u_i$ such that
\begin{equation*}
\| P_{N_i} \tilde{u}_i \|_{F_{N_i}} \leqslant 2 \| P_{N_i} u_i \|_{F_{N_i}(T)}.
\end{equation*}
To lighten notations, we shall again denote $\tilde{u}_i \to u_i$.

We use the definition of the $\mathcal{N}_N$-norm to bound the left hand side of \eqref{eq:DyadicLocalizedBilinearEstimate} by
\begin{equation*}
\begin{split}
&\sup_{t_N \in \R} \| (\tau - \omega_\alpha(\xi,\eta)+iN)^{-1} \xi 1_{A_N}(\xi) \mathcal{F}_{t,x,y}[u_{1,N_1} \cdot \eta_0(N (t-t_N))] \\
&\quad * \mathcal{F}_{t,x,y}[u_{2,N_2} \cdot \eta_0(N(t-t_N))] \|_{X_N}.
\end{split}
\end{equation*}
We define
\begin{equation*}
f_{i,N_i} = \mathcal{F}_{t,x,y}[u_{i,N_i} \cdot \eta_0(N(t-t_N))].
\end{equation*}
Consider $L_1,L_2 \geqslant N$ and $f_{i,N_i,L_i}: \R \times \R^2 \to \R_+$ and for $L_i = N$ in $D_{N,\leqslant L}$. By the function space properties it suffices to obtain an estimate
\begin{equation}
\label{eq:ConstantHighHighLowKPIR2}
\sum_{L \geqslant N} L^{-\frac{1}{2}} N \| 1_{D_{N,L}} (f_{1,N_1,L_1} * f_{2,N_2,L_2}) \|_{L^2_{\tau,\xi,\eta}} \lesssim (1 \vee N_2)^{\frac{1}{2}} (N_2 \wedge 1)^{\varepsilon} \prod_{i=1}^2 L_i^{\frac{1}{2}} \| f_i \|_{L^2_{\tau,\xi,\eta}}.
\end{equation} 
Once the above display is established, estimates \eqref{eq:DyadicLocalizedBilinearEstimate} and consequently \eqref{eq:ShorttimeBilinearEstimateKPIA} and \eqref{eq:ShorttimeBilinearEstimateKPIB} follow from summation and properties of the function spaces. Note that by the time localization we have $L,L_i \geqslant N$. Below we shall distinguish the cases of
\begin{enumerate}
\item Very low frequencies: $N \lesssim 1$,
\item Resonant case: $L_{\max} \ll N_1^2 N_2$,
\item Non-resonant case: $L_{\max} \gtrsim N_1^2 N_2$.
\end{enumerate}

\noindent \textbf{(1) Very low frequencies} $N \sim N_1 \lesssim 1$. We carry out an additional homogeneous frequency decomposition. In this case we apply the bilinear Strichartz estimate from Lemma \ref{lem:AlternativeBilinearStrichartzEstimate} to find
\begin{equation*}
\begin{split}
\| 1_{D_{N,L}} (f_{1,N_1,L_1} * f_{2,N_2,L_2}) \|_{L^2_{\tau,\xi,\eta}} 
&\lesssim (L_1 \wedge L_2)^{\frac{1}{2}} (L_1 \vee L_2)^{\frac{1}{4}} N_2^{\frac{3}{4}}
\prod_{i=1}^2\| f_{i,N_i,L_i} \|_{L^2}.
\end{split}
\end{equation*}

\noindent \textbf{(2) Resonant case} $N \gg 1$, $L_{\max} \ll N_1^2 N_2$. If $N_2 \sim 1$, we carry out an additional homogeneous frequency decomposition. Recall that in the resonant case, we have the transversality bound
\begin{equation*}
\big| \frac{\eta_1}{\xi_1} - \frac{\eta - \eta_1}{\xi -\xi_1} \big| \sim N,
\end{equation*}
which allows for the bilinear Strichartz estimate invoking Proposition \ref{prop:BilinearStrichartzGeneral}:
\begin{equation*}
\begin{split}
\sum_{N \leqslant L \ll N_1^2 N_2} L^{-\frac{1}{2}} \| 1_{D_{N,L}} (f_{1,N_1,L_1} * f_{2,N_2,L_2} ) \|_{L^2_{\tau,\xi,\eta}} &\leqslant N^{-\frac{1}{2}} \| f_{1,N_1,L_1} * f_{2,N_2,L_2} \|_{L^2_{\tau,\xi,\eta}} \\
&\lesssim N_1^{-1} N_2^{\frac{1}{2}} \prod_{i=1}^2 L_i^{\frac{1}{2}} \| f_{i,N_i,L_i} \|_{L^2_{\tau,\xi,\eta}}.
\end{split}
\end{equation*}
This implies \eqref{eq:ConstantHighHighLowKPIR2} for $r_i$, $s >\frac{1}{2}$.

\noindent \textbf{(3) Non-resonant case.} $N \gg 1$, $L_{\max} \gtrsim N_1^2 N_2$. If $N_2 \sim 1$, we carry out an additional homogeneous dyadic decomposition in $N_2 \in 2^{\Z}$. We distinguish between $L=L_{\max}$ and $L_i = L_{\max}$. In case $L=L_{\max}$ we obtain for $N_2 \gtrsim N_1^{-1}$ by applying the bilinear Strichartz estimate in Lemma \ref{lem:AlternativeBilinearStrichartzEstimate}
\begin{equation*}
\begin{split}
&\quad \sum_{L \geqslant N_1^2 N_2 } L^{-\frac{1}{2}} \| 1_{D_{N,L}} (f_{1,N_1,L_1} * f_{2,N_2,L_2}) \|_{L^2_{\tau,\xi,\eta}} \\
&\lesssim (N_1^2 N_2)^{-\frac{1}{2} } N_2^{\frac{3}{4}} (L_1 \wedge L_2)^{\frac{1}{2}} (L_1 \vee L_2)^{\frac{1}{4}} \prod_{i=1}^2 \| f_{i,N_i,L_i} \|_{L^2_{\tau,\xi,\eta}} \\
&\lesssim N_1^{-1} (N_2 / N_1)^{\frac{1}{4}} \prod_{i=1}^2 L_i^{\frac{1}{2}} \| f_{i,N_i,L_i} \|_{L^2_{\tau,\xi,\eta}}. 
\end{split}
\end{equation*}
This implies \eqref{eq:ShorttimeBilinearEstimateKPIA} and \eqref{eq:ShorttimeBilinearEstimateKPIB} for $r_i,s \geqslant 0$.

\smallskip
For $N_2 \ll N_1^{-1}$ the same argument gives
\begin{equation*}
\begin{split}
&\quad \sum_{L \geqslant N_1} L^{-\frac{1}{2}} \| 1_{D_{N,L}} (f_{1,N_1,L_1} * f_{2,N_2,L_2}) \|_{L^2_{\tau,\xi,\eta}} \\
&\lesssim N_2^{\frac{3}{4}} N_1^{-\frac{3}{4}} \prod_{i=1}^2 L_i^{\frac{1}{2}} \| f_{i,N_i,L_i} \|_{L^2_{\tau,\xi,\eta}}.
\end{split}
\end{equation*}
This shows \eqref{eq:ShorttimeBilinearEstimateKPIA} and \eqref{eq:ShorttimeBilinearEstimateKPIB} for $r_i,s > - \frac{1}{2}$.
The Low-High-High-interaction is settled.

\medskip

\subsubsection{High $\times$ High $\to$ High interaction} $N \sim N_1 \sim N_2 \gg 1$. We distinguish between resonant case $L_{\max} \ll N_1^3$ and non-resonant case $L_{\max} \gtrsim N_1^3$.

\noindent \textbf{(1) Resonant case} $L_{\max} \ll N_1^3$. In this case, we can apply the nonlinear Loomis-Whitney inequality \eqref{eq:LoomisWhitneyEuclidean} after invoking duality to find
\begin{equation*}
\sum_{N_1 \leqslant L \ll N_1^3} L^{-\frac{1}{2}} \| 1_{D_{N,L}} (f_{1,N_1,L_1} * f_{2,N_2,L_2}) \|_{L^2_{\tau,\xi,\eta}} \lesssim N_1^{-\frac{3}{2}} \log(N_1) \prod_{i=1}^2 L_i^{\frac{1}{2}} \| f_{i,N_i,L_i} \|_{L^2_{\tau,\xi,\eta}}.
\end{equation*}
This shows \eqref{eq:ShorttimeBilinearEstimateKPIA} and \eqref{eq:ShorttimeBilinearEstimateKPIB} for $r_i,s>-\frac{1}{2}$.

\noindent \textbf{(2) Non-resonant case} $L_{\max} \gtrsim N_1^3$. Suppose by symmetry (up to logarithmic loss) that $L = L_{\max}$. In this case we apply H\"older's inequality and two $L^4$-Strichartz estimates from Proposition \ref{prop:L4StrichartzEstimates}:
\begin{equation*}
\sum_{L \gtrsim N_1^3} L^{-\frac{1}{2}} \|1_{D_{N,L}} (f_{1,N_1,L_1} * f_{2,N_2,L_2}) \|_{L^2_{\tau,\xi,\eta}} \lesssim N_1^{-\frac{3}{2}} \prod_{i=1}^2 L_i^{\frac{1}{2}} \| f_{i,N_i,L_i} \|_{L^2_{\tau,\xi,\eta}}.
\end{equation*}
This implies again \eqref{eq:ShorttimeBilinearEstimateKPIA} and \eqref{eq:ShorttimeBilinearEstimateKPIB} for $r_i,s>-\frac{1}{2}$.

\subsubsection{High $\times$ High $\to$ Low interaction} $N \ll N_1 \sim N_2$. If $N \lesssim 1$, we carry out an additional homogeneous frequency decomposition in $N$. In the following we suppose that $N_1 \gg 1$ as the case $N \lesssim N_1 \lesssim 1$ can be handled like above via $L^4$-Strichartz estimates from Proposition \ref{prop:L4StrichartzEstimates}. The reduction to modulation localized estimates differs from the previous cases because the time localization of $\mathcal{N}_N$-norm does not suffice to estimate the $F_{N_i}$-norm. We add time-localization, which incurs a factor of $N_1 (1 \vee N)^{-1}$:
\begin{equation*}
\begin{split}
&\quad \| (\tau - \omega_\alpha(\xi,\eta) + i N)^{-1} \xi 1_{A_N}(\xi) \mathcal{F}_{t,x,y}[ u_{1,N_1} u_{2,N_2} \cdot \eta^2_0(N(t-t_N))] \|_{X_N} \\
&= \| (\tau - \omega_\alpha(\xi,\eta) + i N)^{-1} \xi 1_{A_N}(\xi) \mathcal{F}_{t,x,y}[u_{1 ,N_1} \eta_0(N(t-t_N)) \\
&\quad \cdot u_{2,N_2} \eta_0(N(t-t_N)) \sum_{k \in \Z} \gamma^2(N_1 t - k)] \|_{X_N}.
\end{split}
\end{equation*}
Consequently, it suffices to estimate $N_1 (1 \vee N)^{-1}$ terms of the form
\begin{equation*}
\begin{split}
&\quad \| (\tau - \omega_\alpha(\xi,\eta) + iN)^{-1} \xi 1_{A_N}(\xi) \mathcal{F}_{t,x,y}[u_{1,N_1} \gamma(N_1 t - k) \eta_0(N(t-t_N))] \\
&\quad * \mathcal{F}_{t,x,y}[u_{2,N_2} \gamma(N_1 t - k) \eta_0(N(t-t_N))] \|_{X_N}.
\end{split}
\end{equation*}

\noindent \textbf{(1) Very small frequencies} ($N\lesssim N_1 \lesssim 1$): As mentioned above, this case can be handled using two $L^4$ Strichartz estimates as in the previous estimate.\\
The case $N \lesssim 1 \lesssim N_1\sim N_2$ is considered in the following cases.

\noindent \textbf{(2) Resonant case} ($L_{\max} \ll N_1^2 N$): In this case, we apply the nonlinear Loomis-Whitney inequality \eqref{eq:LoomisWhitneyEuclidean} to find
\begin{equation*}
\begin{split}
&\quad \sum_{N \leqslant L \ll N_1^2 N} L^{-\frac{1}{2}} \| 1_{D_{N,L}} (f_{1,N_1,L_1} * f_{2,N_2,L_2}) \|_{L^2_{\tau,\xi,\eta}} \\
&\lesssim N_1^{-1} N^{-\frac{1}{2}} \log(N_1) \prod_{i=1}^2 L_i^{\frac{1}{2}} \| f_{i,N_i,L_i} \|_2.
\end{split}
\end{equation*}
This implies \eqref{eq:ShorttimeBilinearEstimateKPIA}  for $r_2 \geq r_1 > 0$ taking into account derivative loss and time localization.

For \eqref{eq:ShorttimeBilinearEstimateKPIB} we need to handle the low frequency weight. The above estimate suffices for $s>\frac{1}{2}$ when $N \gtrsim N_1^{-2}$. If $N \ll N_1^{-2}$ we apply Lemma \ref{lem:AlternativeBilinearStrichartzEstimate} to find
\begin{equation}
\label{eq:AuxShorttimeEstimateKPIR2}
\begin{split}
&\quad \sum_{N \leqslant L \ll N_1^2 N} L^{-\frac{1}{2}} \| 1_{D_{N,L}} (f_{1,N_1,L_1} * f_{2,N_2,L_2}) \|_{L^2_{\tau,\xi,\eta}} \\
&\lesssim N^{\frac{3}{4}} \log(N_1) N_1^{-\frac{3}{4}} \prod_{i=1}^2 L_i^{\frac{1}{2}} \| f_{i,N_i,L_i} \|_2.
\end{split}
\end{equation}
This shows \eqref{eq:ShorttimeBilinearEstimateKPIB} also for very low $N \ll 1$.

\noindent \textbf{(3) Non-resonant case} $L_{\max} \gtrsim N_1^2 N$. Suppose that $L = L_{\max}$. The other cases only deviate logarithmically by additional summation in the modulation. We apply two $L^4$-Strichartz estimates from Proposition \ref{prop:L4StrichartzEstimates} to find
\begin{equation*}
\begin{split}
&\quad \sum_{L \gtrsim N_1^2 N} L^{-\frac{1}{2}} \| 1_{D_{N,L}} (f_{1,N_1,L_1} * f_{2,N_2,L_2}) \|_{L^2_{\tau,\xi,\eta}} \lesssim N_1^{-1} N^{-\frac{1}{2}} \prod_{i=1}^2 L_i^{\frac{1}{2}} \| f_{i,N_i,L_i} \|_2.
\end{split}
\end{equation*}
This shows \eqref{eq:ShorttimeBilinearEstimateKPIA} and \eqref{eq:ShorttimeBilinearEstimateKPIB} for $r_i,s>1/2$ and $N \gg 1$.

Next, we suppose that $N \ll 1 \ll N_1 \sim N_2$. We carry out an additional homogeneous decomposition in frequency $N \in 2^{-\N}$ to localize the derivative. The above display suffices for \eqref{eq:ShorttimeBilinearEstimateKPIA}.
As long as $N \gtrsim N_1^{-2}$ the above display still suffices for \eqref{eq:ShorttimeBilinearEstimateKPIB}. If $N \ll N_1^{-2}$, we apply a bilinear Strichartz estimate  from Lemma \ref{lem:AlternativeBilinearStrichartzEstimate} to find \eqref{eq:AuxShorttimeEstimateKPIR2}, which also finishes the proof of \eqref{eq:ShorttimeBilinearEstimateKPIB}. The proof of Proposition \ref{prop:ShorttimeBilinearKPIR2} complete.
\end{proof}

\subsection{Short-time energy estimates}

Next, we turn to the iteration of the energy norm. We show the following:
\begin{proposition}
\label{prop:ShorttimeEnergyIterationKPIR2}
Let $T \in (0,1]$, $s_1 \geqslant s_2 > \frac{1}{2}$, and $u \in C([0,T],H^{\infty,0}(\R^2))$ be a smooth solution to \eqref{eq:FKPI} with $\alpha = 2$. Then the following estimate holds:
\begin{equation}
\label{eq:ShorttimeEnergyIterationKPI}
\| u \|^2_{B^{s_1}(T)} \lesssim \| u_0 \|_{H^{s_1,0}}^2 + \| u \|_{F^{s_1}(T)}^2 \| u \|_{F^{s_2}(T)}.
\end{equation}
Let $s>\frac{1}{2}$, and $v = u_1-u_2$ for $u_i \in C([0,T],H^{\infty,0}(\R^2))$. Then the following estimates hold:
\begin{align}
\label{eq:ShorttimeEnergyDifferenceIterationKPIA}
\| v \|^2_{\bar{B}^{-\frac{1}{2}}(T)} &\lesssim \| v(0) \|_{\bar{H}^{-\frac{1}{2}}}^2 + \| v \|^2_{\bar{F}^{-\frac{1}{2}}(T)} ( \| u_1 \|_{F^{s}(T)} + \| u_2 \|_{F^s(T)} ), \\
\label{eq:ShorttimeEnergyDifferenceIterationKPIB}
\| v \|^2_{B^{s}(T)} &\lesssim \| v(0) \|_{H^{s,0}}^2 + \| v \|^3_{F^s(T)} + \| v \|_{F^s(T)} \| v \|_{\bar{F}^{-\frac{1}{2}}(T)} \| u_2 \|_{F^{2s+\frac{1}{2}}(T)}.
\end{align}
\end{proposition}
\begin{proof}
We begin with the proof of \eqref{eq:ShorttimeEnergyIterationKPI}.
We invoke the fundamental theorem of calculus to write for $N \in 2^{\N_0}$, $N \gg 1$ and $t \in [0,T]$:
\begin{equation}
\label{eq:FundamentalTheoremKPI}
\| P_N u(t) \|^2_{L^2} = \| P_N u(0) \|_{L^2}^2 + \int_0^t \int_{\R^2} P_N u \partial_x (P_N (u^2)) dx ds.
\end{equation}
To obtain more favorable estimates, we integrate by parts to assign the derivative always to the lowest frequency term. The arguments are standard, and we shall be brief (cf. \cite{IonescuKenigTataru2008,rsc2019,SanwalSchippa2023}).

\medskip

We use a paraproduct decomposition:
\begin{equation*}
P_N (u^2) = 2 P_N ( u P_{\ll N} u) + P_N( P_{\gtrsim N} u P_{\gtrsim N} u).
\end{equation*}
We decompose
\begin{equation*}
\int_0^t \int_{\R^2} P_N (u P_{\ll N} u) \partial_x P_N u dx ds = \sum_{\substack{K \ll N, \\ K \in 2^{\Z}}} \int_0^t \int_{\R^2} P_N (u P_K u) \partial_x P_N u dx ds.
\end{equation*}
Write
\begin{equation*}
P_N(u P_K u) = P_N u P_K u + [ P_N(u P_K u) - P_N u P_K u].
\end{equation*}
For the first term we can integrate by parts to assign the derivative on the lowest frequency:
\begin{equation*}
\int_{\R^2} P_N u P_K u (\partial_x P_N u) dx =  - \frac{1}{2} \int_{\R^2} (P_N u)^2 (\partial_x P_K u) dx.
\end{equation*}

For the second term we consider the bilinear Fourier multiplier:
\begin{equation*}
m(\xi_1,\xi_2) = (\xi_1+\xi_2) [ \chi_N(\xi_1+\xi_2) \chi_K(\xi_2) - \chi_N(\xi_1) \chi_K(\xi_2) ].
\end{equation*}
It is a consequence of the mean-value theorem that
\begin{equation*}
\big| \partial_{\xi_1}^{\alpha_1} \partial_{\xi_2}^{\alpha_2} m(\xi_1,\xi_2) \big| \lesssim K N^{-\alpha_1} K^{-\alpha_2} \tilde{\chi}_N(\xi_1) \tilde{\chi}_K(\xi_2)
\end{equation*}
for suitable mild enlargements of $\chi_N$, $\chi_K$. Consequently, by Fourier series expansion the second term can effectively be regarded as
\begin{equation*}
\int_0^t \int_{\R^2} \big[ P_N (u P_K u) - P_N u P_K u \big] \partial_x P_N u dx ds \sim \int_0^t \int_{\R^2} \tilde{P}_N u P_N u \partial_x P_K u dx ds.
\end{equation*}
The reduction points out that we can always regard the derivative acting on the lowest frequency.

$\bullet$ High $\times$ Low $\to$ High interaction ($N_2 \ll N_1\sim N_3$): We show the estimate
\begin{equation}
\label{eq:HighLowHighEnergyKPIR2}
\sum_{\substack{N_2 \ll N_1 \sim N_3, \\ N_1 \gtrsim 1}} N_1^{2s_1} \int_0^t \int_{\R^2} P_{N_1} u \partial_x P_{N_2} u P_{N_3} u dx ds \lesssim \| u \|^2_{F^{s_1}(T)} \| u \|_{F^{s_2}(T)}.
\end{equation}
Firstly we smoothly decompose the interval $[0,t] \subseteq [0,T]$ into intervals of length $\lesssim N_1^{-1}$:
\begin{equation*}
1 = \sum_{k \in \Z} \gamma^3(N_1 t- k)
\end{equation*}
for a suitable $\gamma \in C^\infty_c(\R)$. This achieves the required time-localization to carry out estimates in the short-time Fourier restriction spaces. In the following we focus on estimates
\begin{equation}
\label{eq:AuxEnergyEstimateKPI}
\begin{split}
&\quad \big| \int_{\R} \int_{\R^2} (\gamma(N_1 s - k) P_{N_1} u) (\gamma (N_1 s - k) P_{N_2} u) (\gamma(N_1 s - k) P_{N_3} u) dx ds \big| \\
&\lesssim N_1^{-1} N_2^{-\frac{1}{2}} \prod_{i=1}^3 \| P_{N_i} u \|_{F_{N_i}}.
\end{split}
\end{equation}
Estimates \eqref{eq:HighLowHighEnergyKPIR2} follow from taking into account time localization, which incurs a factor $T N_1$ and the derivative loss, which draws a factor $N_2$.\footnote{We remark that the boundary terms actually require a separate estimate due to the sharp time-cutoff. There are at most $4$ terms, which effectively gains a factor $N^{-1}$ from the time localization. We refer to the literature for further details (cf. \cite{IonescuKenigTataru2008}).}

We turn to establishing \eqref{eq:AuxEnergyEstimateKPI}: By an application of Plancherel's theorem and the properties of the function spaces, it suffices to establish estimates
\begin{equation}
\label{eq:AuxEnergyEstimateModulationLocalizedKPI}
\int (f_{1,N_1,L_1} * f_{2,N_2,L_2} ) f_{3,N_3,L_3} d\xi d\eta d\tau \lesssim N_1^{-1} N_2^{-\frac{1}{2}} \prod_{i=1}^3 L_i^{\frac{1}{2}} \| f_{i,N_i,L_i} \|_2
\end{equation}
for $f_{i,N_i,L_i} : \R \times \R^2 \to \R_{>0}$ with $\text{supp}(f_{i,N_i,L_i}) \subseteq D_{N_i,L_i}$, $L_i \geqslant N_1$. We record this in the following lemma which is proved in \cite{IonescuKenigTataru2008}. We include the proof for self-containedness.
\begin{lemma}
\label{lem:EnergyTrilinearEstimateKPI}
Let $N_1 \sim N_2 \gtrsim N_3$ for $N_1 \in 2^{\N_0}$ and $N_3 \in 2^{\Z}$. Then \eqref{eq:AuxEnergyEstimateModulationLocalizedKPI} holds true.
\end{lemma}
 \begin{proof}
First, we handle the resonant case: $\max_i(L_i) \ll N_1^2 N_2$. An application of the nonlinear Loomis-Whitney inequality \eqref{eq:LoomisWhitneyEuclidean} gives
\begin{equation*}
\int (f_{1,N_1,L_1} * f_{2,N_2,L_2}) f_{3,N_3,L_3} d\xi d\eta d\tau \lesssim N_1^{-1} N_2^{-\frac{1}{2}} \prod_{i=1}^3 L_i^{\frac{1}{2}} \| f_{i,N_i,L_i} \|_2,
\end{equation*}
which is \eqref{eq:AuxEnergyEstimateModulationLocalizedKPI}.

\smallskip

Next, suppose we are in the non-resonant case: $\max(L) \gtrsim N_1^2 N_2$. Suppose that $L_1 = L_{\max}$. The other cases can be handled similarly. Applying H\"older's inequality and two $L^4_{t,x,y}$-Strichartz estimates from Proposition \ref{prop:L4StrichartzEstimates} yields again \eqref{eq:AuxEnergyEstimateModulationLocalizedKPI}:
\begin{equation*}
\begin{split}
\big| \int (f_{1,N_1,L_1} * f_{2,N_2,L_2} ) f_{3,N_3,L_3} d\xi d\eta d\tau \big| &\leqslant \| f_{1,N_1,L_1} \|_{L^2_{\tau,\xi,\eta}} \prod_{i=2}^3 \| \mathcal{F}^{-1}_{t,x,y} [f_{i,N_i,L_i} ] \|_{L^4_{t,x,y}} \\
&\lesssim (N_1^2 N_2)^{-\frac{1}{2}} \prod_{i=1}^3 L_i^{\frac{1}{2}} \| f_{i,N_i,L_i} \|_{L_{\tau,\xi,\eta}^2}.
\end{split}
\end{equation*}
\end{proof}

$\bullet$ High $\times$ High $\to$ High interaction ($N_1 \sim N_2 \sim N_3 \gg 1$): The estimate \eqref{eq:AuxEnergyEstimateKPI} follows again from applying Lemma \ref{lem:EnergyTrilinearEstimateKPI}.

$\bullet$ High $\times$ High $\to$ Low interaction ($N_1 \ll N_2 \sim N_3$): We obtain from the above reductions and an application of Lemma \ref{lem:EnergyTrilinearEstimateKPI}
\begin{equation*}
\begin{split}
&\quad \sum_{1 \lesssim N_1 \ll N_2 \sim N_3} N_1^{2s_1}  \big| \int_{0}^t \int_{\R^2} (\partial_x P_{N_1} u) P_{N_2} u P_{N_3} u dx ds \big| \\
&\lesssim \sum_{1 \lesssim N_1 \ll N_2 \sim N_3} N_1^{2s_1} N_1^{\frac{1}{2}} \prod_{i=1}^3 \| P_{N_i} u \|_{F_{N_i}}.
\end{split}
\end{equation*}
Dyadic summation using the Cauchy-Schwarz inequality implies \eqref{eq:ShorttimeEnergyIterationKPI} for the claimed regularities. This finishes the proof of \eqref{eq:ShorttimeEnergyIterationKPI}.

$\hfill \Box$

\medskip

\subsection{Energy estimates for the difference of solutions} We turn to the proof of \eqref{eq:ShorttimeEnergyDifferenceIterationKPIA}. Invoking the fundamental theorem of calculus for a solution to the difference equation on frequencies $N \gg 1$ yields
\begin{equation*}
\| P_N v(t) \|_{L^2}^2 = \|P_N v(0) \|^2_{L^2} + \int_0^t \int_{\R^2} P_N v \partial_x (P_N(v(u_1+u_2))) dx ds.
\end{equation*}
We take again a paraproduct decomposition:
\begin{equation}
\label{eq:ParaproductDifferences}
\begin{split}
P_N(v(u_1+u_2)) &= P_N(v P_{\ll N} (u_1+u_2)) + P_N (P_{\gtrsim N} v P_{\gtrsim N}(u_1+u_2)) \\
&\quad + P_N (P_{\ll N} v (u_1+u_2)).
\end{split}
\end{equation}

The first term can be handled by integration by parts like above. The second term is different because we need to estimate one function with high frequency at negative Sobolev regularity. It suffices to show for $i=1,2$ and $s>\frac{1}{2}$:
\begin{equation*}
\sum_{1 \lesssim N_1 \lesssim N_2 \sim N_3} N_1^{-1} \big| \int_0^t \int_{\R^2} (\partial_x P_{N_1} v) P_{N_2} v P_{N_3} u_i dx ds \big| \lesssim \| v \|^2_{\bar{F}^{-\frac{1}{2}}(T)} \| u_i \|_{F^s(T)}.
\end{equation*}
This is reduced to dyadic estimates for $1 \lesssim N_1 \lesssim N_2 \sim N_3$:
\begin{equation}
\label{eq:EnergyDyadicEstimateKPIAux}
\big| \int_0^t \int_{\R^2} P_{N_1} v P_{N_2} v P_{N_3} u_i dx ds \big| \lesssim N_1^{-\frac{1}{2}} \| P_{N_1} v \|_{F_{N_1}} \| P_{N_2} v \|_{F_{N_2}} \| P_{N_3} u_i \|_{F_{N_3}}
\end{equation}
for extensions $v$, $u_i$. Time localization incurs a factor $N_2$, by which we reduce to the convolution estimate (after changing to Fourier variables, modulation localization, and taking into account the properties of the function spaces):
\begin{equation*}
\int (f_{1,N_1,L_1} * f_{2,N_2,L_2} ) f_{3,N_3,L_3} d\xi d\eta d\tau \lesssim N_1^{-\frac{1}{2}} N_2^{-1} \prod_{i=1}^3 L_i^{\frac{1}{2}} \| f_{i,N_i,L_i} \|_2
\end{equation*}
with $1 \lesssim N_1 \lesssim N_2 \sim N_3$ and $L_i \geqslant N_2$. This is the content of Lemma \ref{lem:EnergyTrilinearEstimateKPI}.

\medskip

The third term is different because the derivative does not act on the lowest frequency and we cannot integrate by parts due to lack of symmetry. Here the estimate at negative regularity comes to rescue. We need to obtain summable estimates for $K \ll N$, $N \gg 1$:
\begin{equation*}
\big| \int_0^t \int_{\R^2} P_N v P_K v \tilde{P}_N u_i dx ds \big| \lesssim K^{-\frac{1}{2}} \| P_N v \|_{F_N} \| P_K v \|_{F_K} \| \tilde{P}_N u_i \|_{F_N}.
\end{equation*}
Note that we have not included the factor coming from time localization.


\smallskip

For $K \gtrsim 1$ the above estimate has already been verified in \eqref{eq:EnergyDyadicEstimateKPIAux}.

\smallskip

Suppose that $K \lesssim 1$. Let $N_1 \sim N_3 \gtrsim N_2 \sim K$. In the resonant case $L_{\max} \ll N_1^2 N_2$ we can use a bilinear Strichartz estimate from Proposition \ref{prop:BilinearStrichartzResonant} to find with above notations:
\begin{equation*}
\int (f_{1,N_1,L_1} * f_{2,N_2,L_2}) f_{3,N_3,L_3} d\xi d\eta d\tau \lesssim N_1^{-1} N_2^{\frac{1}{2}} \prod_{i=1}^3 L_i^{\frac{1}{2}} \| f_{i,N_i,L_i} \|_2.
\end{equation*}
In the non-resonant case $L_{\max} \gtrsim N_1^2 N_2$ we can use a different bilinear Strichartz estimate Lemma \ref{lem:AlternativeBilinearStrichartzEstimate} on $f_{1,N_1,L_1} * f_{2,N_2,L_2}$ to find with $L_2 = L_{\max}$:
\begin{equation*}
\begin{split}
\int (f_{1,N_1,L_1} * f_{2,N_2,L_2}) f_{3,N_3,L_3} d\xi d\eta d\tau &\lesssim (N_1^2 N_2)^{-\frac{1}{4}} N_2^{\frac{3}{4}} N_1^{-\frac{1}{2}} \prod_{i=1}^3 L_i^{\frac{1}{2}} \| f_{i,N_i,L_i} \|_2 \\
&\lesssim N_1^{-1} N_2^{\frac{1}{2}} \prod_{i=1}^3 \| f_{i,N_i,L_i} \|_2.
\end{split}
\end{equation*}
This estimate is acceptable. Applying the analog argument for $L_i = L_{\max}$, $i \in \{1,3\}$ gives a better estimate.
With this estimate we find by dyadic summation for any $s>\frac{1}{2}$:
\small
\begin{equation}
\label{eq:DyadicSummationDifferencesKPI}
\big| \sum_{\substack{K \ll N, \\ N \gtrsim 1}} N^{-1} \int_0^t \int_{\R^2} P_N v \partial_x (P_K v \tilde{P}_N(u_1+u_2) ) dx ds \big| \lesssim \| v \|_{\bar{F}^{-\frac{1}{2}}(T)}^2 ( \| u_1 \|_{F^s(T)} + \| u_2 \|_{F^s(T)} ).
\end{equation}
\normalsize
This finishes the proof of \eqref{eq:ShorttimeEnergyDifferenceIterationKPIA}.

\medskip

We turn to the proof of \eqref{eq:ShorttimeEnergyDifferenceIterationKPIB}. To this end, we rewrite
\begin{equation*}
\begin{split}
&\quad \| P_N v(t) \|_{L^2}^2 \\
 &= \|P_N v(0) \|^2_{L^2} + \int_0^t \int_{\R^2} P_N v \partial_x (P_N(v(u_1+u_2))) dx ds \\
&= \|P_N v(0) \|^2_{L^2} + \int_0^t \int_{\R^2} P_N v \partial_x (P_N(v^2)) dx ds + \int_0^t \int_{\R^2} P_N v \partial_x (P_N(v u_2)) dx ds.
\end{split}
\end{equation*}
The estimate of the first integral in the last line is like for \eqref{eq:ShorttimeEnergyIterationKPI}. For the second integral we carry out a paraproduct decomposition like in \eqref{eq:ParaproductDifferences}:
\begin{equation*}
P_N(v \, u_2) = P_N(v \, P_{\ll N} u_2) + P_N (P_{\gtrsim N} v P_{\gtrsim N} u_2) + P_N (P_{\ll N} v \, u_2).
\end{equation*}
The first term can be handled by integration by parts, the second term can likewise be treated like in the estimate \eqref{eq:ShorttimeEnergyIterationKPI}. For the last term we use the argument from \eqref{eq:DyadicSummationDifferencesKPI} to find for $s > \frac{1}{2}$:
\begin{equation*}
\big| \sum_{\substack{K \ll N, \\ N \gtrsim 1}} N^{2s} \int_0^t \int_{\R^2} P_N v \partial_x (P_K v \tilde{P}_N u_2)  dx ds \big| \lesssim \| v \|_{\bar{F}^{-\frac{1}{2}}(T)} \| v \|_{F^s(T)} \| u_2 \|_{F^{2s+\frac{1}{2}}(T)}.
\end{equation*}
The proof of \eqref{eq:ShorttimeEnergyDifferenceIterationKPIB} is complete.
\end{proof}

\section{Improved quasilinear local well-posedness for the dispersion-generalized KP-I equation on the Euclidean plane}
\label{section:LWPFKPIR2}
Next, we show improved local well-posedness for quasilinear fractional KP-I equations with $\alpha \in (2,5)$ on $\R^2$:
\begin{equation}
\label{eq:FKPIR2}
\left\{ \begin{array}{cl}
\partial_t u + \partial_x D_x^\alpha u - \partial_x^{-1} \partial_y^2 u &= u \partial_x u, \quad (t,x,y) \in \R \times \R^2, \\
u(0) &= u_0 \in H^{s,0}(\R^2)
\end{array} \right.
\end{equation}
The second and third author showed local well-posedness for $s(\alpha,\varepsilon) \geqslant 5-2\alpha + \varepsilon$ in \cite{SanwalSchippa2023}. Here we improve the result as follows: Define
\begin{equation}
\label{eq:RegularityFKPIR2}
s_1(\alpha,\varepsilon,\R^2) =
\begin{cases}
6 - \frac{11 \alpha}{4} + \varepsilon, &\quad \alpha \in (2,\frac{24}{11}], \\
0, &\quad \alpha \in (\frac{24}{11},\frac{5}{2}),
\end{cases}
\end{equation}

In \cite{SanwalSchippa2023} we showed a priori estimates for $s \geqslant s_1(\alpha,\varepsilon,\R^2)$. In the following we show local well-posedness for $s \geqslant s_1(\alpha,\varepsilon,\R^2)$.
\begin{theorem}
\label{thm:ImprovedQuasilinearLWPFKPIR2}
Let $\alpha \in (2,\frac{5}{2})$. Then \eqref{eq:FKPIR2} is locally well-posed provided that $s \geqslant s_1(\alpha,\varepsilon,\R^2)$ for any $\varepsilon > 0$. The data-to-solution mapping $S_T^\infty : H^{\infty,0}(\R^2) \to C_T H^{\infty,0}$ extends to a mapping $S_T^s: H^{s,0}(\R^2) \to C_T H^{s,0}(\R^2)$ for $s \geqslant s_1(\alpha,\varepsilon,\R^2)$ with $T=T(\| u_0 \|_{H^{s,0}})$ depending lower semicontinuously on $ \| u_0 \|_{H^{s,0}}$ and $T(u) \gtrsim 1$ as $u \downarrow 0$. 
\end{theorem}

\subsection{Outline of the proof.}

We use again short-time Fourier restriction like in Section \ref{section:QuasilinearLWPKPIR2}. The present choice of frequency-dependent time localization is like in \cite{SanwalSchippa2023}: $T=T(N)=N^{-(5-2\alpha+\varepsilon)}$. This interpolates between $T=T(N)=N^{-1}$ for the KP-I equation \eqref{eq:KPIR2} and semilinear local well-posedness for \eqref{eq:FKPIR2} with $\alpha=\frac{5}{2}$, which will be established in Section \ref{section:SemilinearWellposedness}. There is a significant difference between the KP-I equation and the dispersion-generalized KP-I equations: The nonlinear Loomis-Whitney inequality recovers more than one derivative in the resonant case and a logarithmic loss from dyadically summing the modulation can be compensated. This makes the nonlinear Loomis-Whitney inequality also useful for the short-time bilinear estimate. 

\smallskip

Presently, we improve the regularity by estimating the differences of solutions in $H^{-\frac{1}{2},0}(\R^2)$ and use a weight for the low frequencies:
\begin{equation*}
\| f \|^2_{\bar{H}^{-\frac{1}{2},0}} = \int_{\R^2} (1+|\xi|^{-1})^{1+2\varepsilon} \langle \xi \rangle^{-1} |\hat{f}(\xi,\eta)|^2 d\xi d\eta.
\end{equation*}
We denote the derived short-time function spaces with low frequency weight as $\bar{F}^{-\frac{1}{2}}(T)$ or $\bar{\mathcal{N}}^{-\frac{1}{2}}(T)$, respectively.

 We show the following short-time bilinear estimates:
\begin{proposition}
\label{prop:ShorttimeBilinearEstimatesFKPIR2}
Let $T \in (0,1]$, and $\alpha \in (2,\frac{5}{2})$. There is $\varepsilon >0$ such that the following estimate holds for $r_1 \geqslant 0$:
\begin{equation}
\label{eq:ShorttimeBilinearEstimateSolutionsFKPIR2}
\| \partial_x(u v) \|_{\mathcal{N}^{r_1}(T)} \lesssim \| u \|_{F^{r_1}(T)} \| v \|_{F^{0}(T)} + \| u \|_{F^{0}(T)} \| v \|_{F^{r_1}(T)}.
\end{equation}
Secondly, the following estimate holds for $r \geqslant s_1(\alpha,2\varepsilon,\R^2)$:
\begin{equation}
\label{eq:ShorttimeBilinearEstimateDifferencesFKPIR2}
\| \partial_x (u v) \|_{\bar{\mathcal{N}}^{-\frac{1}{2}}(T)} \lesssim \| u \|_{F^r(T)} \| v \|_{\bar{F}^{-\frac{1}{2}}(T)}.
\end{equation}
\end{proposition}

\begin{remark}
We presently remark on the choice of parameters: $\varepsilon$ in the time localization is chosen small enough such that \eqref{eq:ShorttimeBilinearEstimateSolutionsFKPIR2} holds (cf. \cite[Proposition~5.2]{SanwalSchippa2023}). We choose the low frequency weight with the same $\varepsilon$. This allows us to quantify the well-posedness for $s \geqslant s_1(\alpha,5\varepsilon,\R^2)$.
\end{remark}

The energy estimates in short-time Fourier restriction norms read as follows:
\begin{proposition}
\label{prop:EnergyEstimatesFKPIR2}
Let $T \in (0,1]$, and $\alpha \in (2,\frac{5}{2})$. Let $u \in C([0,T];H^{\infty,0}(\R^2))$ be a smooth solution to \eqref{eq:FKPIR2}. For $r \geqslant s_1(\alpha,5 \varepsilon,\R^2)$ the following estimate holds:
\begin{equation}
\label{eq:EnergyEstimateSolutionFKPIR2}
\| u \|^2_{B^{r}(T)} \lesssim \| u_0 \|_{H^{r,0}(\R^2)}^2 + \| u \|^2_{F^{r}(T)} \| u \|_{F^{s_1}(T)}.
\end{equation}

\smallskip

Let $v = u_1 - u_2$ be a difference of two smooth solutions to \eqref{eq:FKPIR2}. For $r \geqslant s_1(\alpha,\varepsilon,\R^2)$ the following estimates hold:
\begin{align}
\label{eq:EnergyEstimateDifferencesFKPIR2A}
\| v \|^2_{\bar{B}^{-\frac{1}{2}}(T)} &\lesssim \| v(0) \|^2_{\bar{H}^{-\frac{1}{2}}(\R^2)} + \| v \|^2_{\bar{F}^{-\frac{1}{2}}(T)} ( \| u_1 \|_{F^{r}(T)} + \| u_2 \|_{F^r(T)} ), \\
\label{eq:EnergyEstimateDifferencesFKPIR2B}
\| v \|^2_{B^r(T)} &\lesssim \| v(0) \|^2_{H^{r,0}(\R^2)} + \| v \|^3_{F^r(T)} + \| v \|_{\bar{F}^{-\frac{1}{2}}(T)} \| v \|_{F^r(T)} \| u_2 \|_{F^{2r+\frac{1}{2}}(T)}.
\end{align}
\end{proposition}

The estimates for solutions were proved in \cite{SanwalSchippa2023}: \eqref{eq:ShorttimeBilinearEstimateSolutionsFKPIR2} was proved in \cite[Proposition~5.2]{SanwalSchippa2023}, and \eqref{eq:EnergyEstimateSolutionsCylinder} was proved in \cite[Proposition~5.8]{SanwalSchippa2023}. The proof of Theorem \ref{thm:ImprovedQuasilinearLWPKPIR2} can be carried out in three steps:
\begin{itemize}
\item Proof of a priori estimates for $s \geqslant s_1(\alpha,\varepsilon)$, which has already been carried out in \cite{SanwalSchippa2023},
\item Proof of Lipschitz continuous dependence in $\bar{H}^{-\frac{1}{2}}$ for solutions in $F^r(T)$, $r \geqslant s_1(\alpha,\varepsilon,\R^2)$,
\item Conclusion of local well-posedness via Bona--Smith approximation.
\end{itemize}
The second and third step were described in the context of the KP-I equation in Section \ref{section:QuasilinearLWPKPIR2}. We presently omit the details to avoid repetition.

\subsection{A trilinear convolution estimate}

The following trilinear estimate will play a central role in establishing the short-time estimates.
\begin{lemma}
\label{lem:TrilinearConvolutionDispersionGeneralized}
Let $\alpha \geqslant 2$, $N_i \in 2^{\Z}$ with $N_2 \ll N_1 \sim N_3$, $N_1 \gtrsim 1$, and $L_i \gtrsim (1 \vee N_i^{5-2\alpha+\varepsilon})$. Let $f_{i,N_i,L_i}: \R \times \R^2 \to \R_{\geqslant 0}$ with $\text{supp}(f_{i,N_i,L_i}) \subseteq D_{N_i,L_i}$ for $i \in \{1,2,3\}$. Then the following estimate holds:
\begin{equation}
\label{eq:TrilinearConvolutionEstimateR2}
\int (f_{1,N_1,L_1} * f_{2,N_2,L_2}) f_{3,N_3,L_3} d\xi d\eta d\tau \lesssim N_2^{-\frac{1}{2}} N_1^{\frac{1}{2}-\frac{3 \alpha}{4}} \prod_{i=1}^3 L_i^{\frac{1}{2}} \| f_{i,N_i,L_i} \|_2.
\end{equation}
\end{lemma}
\begin{proof}
In the resonant case $L_{\max} \ll N_1^\alpha N_2$ this is a consequence of the nonlinear Loomis-Whitney inequality \eqref{eq:LoomisWhitneyEuclidean}. In case $L_{\max} \gtrsim N_1^\alpha N_2$ we distinguish between $L_2 = L_{\max}$ and $(L_1 = L_{\max}$ or $L_3 = L_{\max})$.

In case $L_2 = L_{\max}$ we apply H\"older's inequality and two $L^4_{t,x,y}$-Strichartz estimates from Proposition \ref{prop:L4StrichartzEstimates} to find
\begin{equation*}
\begin{split}
\int (f_{1,N_1,L_1} * f_{2,N_2,L_2} ) f_{3,N_3,L_3} d\xi d\eta d\tau &\leqslant \| f_{2,N_2,L_2} \|_{L^2_{\tau,\xi,\eta}} \prod_{i=1,3} \| \mathcal{F}^{-1}_{t,x,y}[f_{i,N_i,L_i} ] \|_{L^4_{t,x,y}} \\
&\lesssim (N_1^\alpha N_2)^{-\frac{1}{2}} N_1^{\frac{2-\alpha}{4}} \prod_{i=1}^3 L_i^{\frac{1}{2}} \| f_{i,N_i,L_i} \|_2,
\end{split}
\end{equation*}
which is \eqref{eq:TrilinearConvolutionEstimateR2}.

\smallskip

Next, suppose $L_3 = L_{\max} \gtrsim N_1^\alpha N_2$. We apply H\"older's inequality and Lemma \ref{lem:CFBilinearStrichartz} to find
\begin{equation*}
\begin{split}
\int (f_{1,N_1,L_1} * f_{2,N_2,L_2} ) f_{3,N_3,L_3} d\xi d\eta d\tau &\leqslant \| f_{3,N_3,L_3} \|_{L^2_{\tau,\xi,\eta}} \| f_{1,N_1,L_1} * f_{2,N_2,L_2} \|_{L^2_{\tau,\xi,\eta}} \\
&\lesssim (N_1^\alpha N_2)^{-\frac{1}{2}} \frac{N_2^{\frac{1}{2}}}{N_1^{\frac{\alpha}{4}}} \prod_{i=1}^3 L_i^{\frac{1}{2}} \| f_{i,N_i,L_i} \|_2.
\end{split}
\end{equation*}
This completes the proof.
\end{proof}

\subsection{Short-time bilinear estimates}
\label{subsection:ShorttimeBilinearEstimatesR2}
In this section we finish the proof of Proposition \ref{prop:ShorttimeBilinearEstimatesFKPIR2} by showing \eqref{eq:ShorttimeBilinearEstimateDifferencesFKPIR2}.

As usually, the claim follows from estimates for dyadically localized frequencies:
\begin{equation*}
\| P_N \partial_x (P_{N_1} u P_{N_2} v) \|_{\mathcal{N}^{-\frac{1}{2}}(T)} \lesssim C(N,N_1,N_2) \| P_{N_1} u \|_{F_{N_1}(T)} \| P_{N_2} v \|_{F_{N_2}(T)}
\end{equation*}
for $N,N_1,N_2 \in 2^{\Z}$. For $N \gtrsim N_2$ we can use the estimates from the proof of \eqref{eq:EnergyEstimateSolutionFKPIR2}. 

It remains to check the \emph{High $\times$ High $\to$ Low} interaction. In this case the necessity to estimate the high frequency at low regularity leads to the condition $ r \geqslant s_1(\alpha,\varepsilon,\R^2)$. We shall consider the following cases:
\begin{enumerate}
\item Very low frequencies,
\item Resonant case,
\item Non-resonant case.
\end{enumerate}

\noindent \textbf{(1) Very small frequencies} ($N \lesssim 1 \ll N_1$): We consider two subcases:\\
(i) $L_{\max} \ll N_1^\alpha N$: Since $(N \vee 1) =1$, the time localization factor becomes $N_1^{5-2\alpha+\varepsilon}$ and we need to account for the weight at negative Sobolev regularity. We obtain as a consequence of the nonlinear Loomis-Whitney inequality \eqref{eq:LoomisWhitneyEuclidean}
\begin{equation*}
\begin{split}
&\quad N^{\frac{1}{2}-\delta} N_1^{5-2 \alpha + \varepsilon} \sum_{1 \lesssim L \ll N_1^\alpha N} L^{-\frac{1}{2}} \| 1_{D_{N,L}} (f_{1,N_1,L_1} * f_{2,N_2,L_2}) \|_{L^2_{\tau,\xi,\eta}} \\
&\lesssim N_1^{-\frac{1}{2}} N^{-\delta} N_1^{s_1(\alpha,\varepsilon,\R^2)} \prod_{i=1}^2 L_i^{\frac{1}{2}} \| f_{i,N_i,L_i} \|_2 \lesssim N_1^{-\frac{1}{2}} N_1^{s_1(\alpha,\varepsilon+\delta,\R^2)} \prod_{i=1}^2 L_i^{\frac{1}{2}} \| f_{i,N_i,L_i} \|_2
\end{split}
\end{equation*}
as long as $N \gtrsim N_1^{-2}$. For $N \lesssim N_1^{-2}$ we can apply a bilinear Strichartz estimate from Lemma \ref{lem:AlternativeBilinearStrichartzEstimate}
after invoking duality to find
\begin{equation*}
\begin{split}
&\quad N^{\frac{1}{2}-\delta} N_1^{5-2 \alpha + \varepsilon} \sum_{1 \lesssim L \ll N_1^\alpha N} L^{-\frac{1}{2}} \| 1_{D_{N,L}} (f_{1,N_1,L_1} * f_{2,N_2,L_2}) \|_{L^2_{\tau,\xi,\eta}} \\
&\lesssim N^{\frac{5}{4}-\delta} N_1^{\frac{5-2\alpha+\varepsilon}{2}} \log(N_1) \prod_{i=1}^2 L_i^{\frac{1}{2}} \| f_{i,N_i,L_i} \|_2.
\end{split}
\end{equation*}
This clearly suffices.

\noindent (ii) $L_{\max} \gtrsim N_1^\alpha N$: We further distinguish between $L_{\max} = L$ and $L_{\max}= L_1 (\text{ or }L_2$).\\
$\bullet~ L_{\max} = L$: We apply H\"older's inequality and two $L^4$-Strichartz estimates from Proposition \ref{prop:L4StrichartzEstimates}:
\begin{equation*}
\begin{split}
&\quad N^{\frac{1}{2}-\delta} N_1^{5-2 \alpha + \varepsilon} L^{-\frac{1}{2}} \| 1_{D_{N,L}} (f_{1,N_1,L_1} * f_{2,N_2,L_2}) \|_{L^2_{\tau,\xi,\eta}} \\
&\lesssim N^{-\delta} N_1^{5-2 \alpha + \varepsilon} (N_1^\alpha N)^{-\frac{1}{2}} N_1^{\frac{2-\alpha}{4}} \prod_{i=1}^2 L_i^{\frac{1}{2}} \| f_{i,N_i,L_i} \|_{L^2} \\
&\lesssim N^{-\delta} N_1^{\frac{11}{2} - \frac{11 \alpha}{4} + \varepsilon} \prod_{i=1}^2 L_i^{\frac{1}{2}} \| f_{i,N_i,L_i} \|_{L^2}.
\end{split}
\end{equation*}
Like above this suffices for $N \gtrsim N_1^{-2}$. For $N \ll N_1^{-2}$ we can argue again using a bilinear Strichartz estimate from Lemma \ref{lem:AlternativeBilinearStrichartzEstimate}.

\smallskip

$\bullet~  L_{\max}=L_1 (\text{or }L_2)$ We assume $L_{\text{med}} \ll N_1^\alpha N $. Invoking duality and Lemma \ref{lem:CFBilinearStrichartz} gives
\begin{equation*}
\begin{split}
&\quad N^{\frac{1}{2}-\delta} N_1^{5-2 \alpha + \varepsilon} \sum_{1 \lesssim L \ll N_1^\alpha N} L^{-\frac{1}{2}} \| 1_{D_{N,L}} (f_{1,N_1,L_1} * f_{2,N_2,L_2}) \|_{L^2_{\tau,\xi,\eta}} \\
&\lesssim N^{\frac{1}{2}-\delta} N_1^{5-2 \alpha + \varepsilon} N_1^{\frac{1}{2}-\frac{3\alpha}{4}} N^{-\frac{1}{2}} \prod_{i=1}^2 L_i^{\frac{1}{2}} \| f_{i,N_i,L_i} \|_2 \\
&\lesssim N_1^{-\frac{1}{2}} N^{-\delta} N_1^{s_1(\alpha,\varepsilon,\R^2)} \prod_{i=1}^2 L_i^{\frac{1}{2}} \| f_{i,N_i,L_i} \|_2.
\end{split}
\end{equation*}
This estimate suffices for $N_1^{-2} \lesssim N \lesssim 1$. For $N \lesssim N_1^{-2}$ we can again conclude by Lemma \ref{lem:AlternativeBilinearStrichartzEstimate}.

\smallskip

\noindent \textbf{(2) Resonant case} ($1 \lesssim N \ll N_1, L_{\max}\ll N_1^{\alpha}N$): We find from increasing time localization and modulation localization that it suffices to show estimates
\begin{equation}
\label{eq:LowHighHighShorttimeBilinearEstimateFKPIR2}
\begin{split}
&\quad N^{\frac{1}{2}} \big( \frac{N_1}{N} \big)^{5-2 \alpha + \varepsilon} \sum_{L \geqslant N^{5-2\alpha + \varepsilon}} L^{-\frac{1}{2}} \| 1_{D_{N,L}} (f_{1,N_1,L_1} * f_{2,N_2,L_2} ) \|_{L^2_{\tau,\xi,\eta}} \\
&\lesssim N_1^{-\frac{1}{2}} N^{-c(\varepsilon)} N_1^{s_1(\alpha,C \varepsilon)} \prod_{i=1}^2 L_i^{\frac{1}{2}} \| f_{i,N_i,L_i} \|_2
\end{split}
\end{equation}
for $L_i \geqslant N_1^{5-2\alpha + \varepsilon}$.
First, we suppose that $L_{\max} \ll N_1^\alpha N$. Since this case is resonant, we can apply the nonlinear Loomis-Whitney inequality \eqref{eq:LoomisWhitneyEuclidean} (as a special case of Lemma \ref{lem:TrilinearConvolutionDispersionGeneralized}) to find
\begin{equation*}
\begin{split}
&\quad N^{\frac{1}{2}} \big( \frac{N_1}{N} \big)^{5-2 \alpha + \varepsilon} L^{-\frac{1}{2}} \| 1_{D_{N,L}} (f_{1,N_1,L_1} * f_{2,N_2,L_2}) \|_{L^2_{\tau,\xi,\eta}} \\
&\lesssim N_1^{\frac{1}{2}-\frac{3\alpha}{4}} \big( \frac{N_1}{N} \big)^{5-2\alpha+\varepsilon} \prod_{i=1}^2 L_i^{\frac{1}{2}} \| f_{i,N_i,L_i} \|_2.
\end{split}
\end{equation*}
This estimate implies \eqref{eq:LowHighHighShorttimeBilinearEstimateFKPIR2} with $C = 1$.

\smallskip

\noindent \textbf{(3) Non-resonant case} ($L_{\max} \gtrsim N_1^\alpha N$, $1 \lesssim N \ll N_1 \sim N_2$): Depending on $L_{\max}$, we have the following sub-cases: \\
(i) $L_{\max}=L$. We find from applying H\"older's inequality and two $L^4$-Strichartz  estimates from Proposition \ref{prop:L4StrichartzEstimates}:
\begin{equation*}
\begin{split}
&\quad N^{\frac{1}{2}} \big( \frac{N_1}{N} \big)^{5-2 \alpha + \varepsilon} \sum_{L \geqslant N_1^\alpha N} L^{-\frac{1}{2}} \| 1_{D_{N,L}} (f_{1,N_1,L_1} * f_{2,N_2,L_2}) \|_{L^2_{\tau,\xi,\eta}} \\
&\lesssim N^{\frac{1}{2}} \big( \frac{N_1}{N} \big)^{5-2\alpha+\varepsilon}  (N_1^\alpha N)^{-\frac{1}{2}} N_1^{\frac{2-\alpha}{4}} \prod_{i=1}^2 L_i^{\frac{1}{2}} \| f_{i,N_i,L_i} \|_2 \\
&= N_1^{-\frac{1}{2}} N^{-(5-2\alpha+\varepsilon)}  N_1^{s_1(\alpha,\varepsilon,\R^2)}  \prod_{i=1}^2 L_i^{\frac{1}{2}} \| f_{i,N_i,L_i} \|_2
\end{split}
\end{equation*}
This implies \eqref{eq:LowHighHighShorttimeBilinearEstimateFKPIR2}.\\
(ii) $L_{\max} = L_1$ and $L_{\text{med}} \ll N_1^\alpha N$: In this case, we invoke duality and apply Lemma \ref{lem:TrilinearConvolutionDispersionGeneralized} to find
\begin{equation}
\label{eq:AuxShorttimeDifferenceEstimate}
\begin{split}
&\quad N^{\frac{1}{2}} \big( \frac{N_1}{N} \big)^{5-2 \alpha + \varepsilon} \sum_{N^{5-2\alpha+\varepsilon} \leqslant L \leqslant N_1^\alpha N} L^{-\frac{1}{2}} \| 1_{D_{N,L}} (f_{1,N_1,L_1} * f_{2,N_2,L_2}) \|_{L^2_{\tau,\xi,\eta}} \\
&\lesssim N^{\frac{1}{2}} \big( \frac{N_1}{N} \big)^{5-2\alpha+\varepsilon} \log(N_1) N_1^{\frac{1}{2}-\frac{3\alpha}{4}} N^{-\frac{1}{2}} \prod_{i=1}^2 L_i^{\frac{1}{2}} \| f_{i,N_i,L_i} \|_2.
\end{split}
\end{equation}
This implies again \eqref{eq:LowHighHighShorttimeBilinearEstimateFKPIR2}.\\
The easier case when $L_{\text{med}} \gtrsim N_1^\alpha N_2$ \eqref{eq:LowHighHighShorttimeBilinearEstimateFKPIR2} follows from applying Lemma \ref{lem:AlternativeBilinearStrichartzEstimate}.
This completes the proof of Proposition \ref{prop:ShorttimeBilinearEstimatesFKPIR2}.

$\hfill \Box$

\subsection{Energy estimates}

This section is devoted to the proof of Proposition \ref{prop:EnergyEstimatesFKPIR2}. Energy estimates for solutions with the claimed regularity were already proved in \cite{SanwalSchippa2023}. We need to establish estimates for differences of solutions $v =u_1-u_2$:
\begin{equation}
\label{eq:AuxDifferenceEnergyEstimate}
\| v \|^2_{\bar{B}^{-\frac{1}{2}}(T)} \lesssim \| v(0) \|^2_{\bar{H}^{-\frac{1}{2}}(\R^2)} + \| v \|^2_{\bar{F}^{-\frac{1}{2}}(T)} ( \| u_1 \|_{F^s(T)} + \| u_2 \|_{F^s(T)})
\end{equation}
for
\begin{equation*}
s \geqslant s_1(\alpha,\varepsilon,\R^2) =
\begin{cases}
6 - \frac{11 \alpha}{4} + \varepsilon, \quad &\alpha \in (2, \frac{24}{11}], \\
0, \quad &\alpha \in (\frac{24}{11},\frac{5}{2}).
\end{cases}
\end{equation*}
We turn to the proof. For small frequencies we have by the definition of the function spaces
\begin{equation*}
\| P_{\lesssim 1} v \|_{\bar{B}^{-\frac{1}{2}}(T)} \lesssim \| P_{\lesssim 1} v(0) \|_{\bar{H}^{-\frac{1}{2}}}.
\end{equation*}

For $N \gg 1$ we invoke the fundamental theorem of calculus:
\begin{equation*}
\| P_N v(t) \|^2_{L^2} = \| P_N v(0) \|^2_{L^2} +  \int_0^t \int_{\R^2} P_N v \partial_x (P_N ( v (u_1+u_2)) dx ds.
\end{equation*}
For $i=1,2$, we take a paraproduct decomposition:
\begin{equation}
\label{eq:ParaproductDifferencesFKPIR2}
P_N (v u_i) = P_N (v P_{\ll N} u_i ) + P_N (P_{\ll N} v \cdot u_i) + P_N (P_{\gtrsim N} v P_{\gtrsim N} u_i).
\end{equation}
The first term can be estimated like in \eqref{eq:EnergyEstimateSolutionFKPIR2} by integration by parts and a commutator estimate. The details are omitted. The second term is different because we cannot integrate by parts and for the third term we need to estimate $P_{\gtrsim N} v$ at negative regularity.

\medskip

We turn to the estimate of the second term in \eqref{eq:ParaproductDifferencesFKPIR2}. After adding time localization and taking the negative regularity into account, it remains to estimate
\begin{equation*}
N_1^{5-2\alpha + \varepsilon} \iint_{\R \times \R^2} \gamma^3(N_1 (t-t_I)) P_{N_1} v P_{N_2} v P_{N_3} u dx dt.
\end{equation*}
Above $u$ denotes a solution $u_i$, $N_1 \sim N_3 \gg N_2$, and $\gamma \in C^\infty_c(\R)$ denotes a suitable bump function adapted to the unit interval.

After changing to Fourier space and modulation localization, this reduces again to establishing convolution estimates 
\begin{equation*}
\int (f_{1,N_1,L_1} * f_{2,N_2,L_2}) f_{3,N_3,L_3} d\xi d\eta d\tau \lesssim N_1^{\frac{1}{2}-\frac{3\alpha}{4}} N_2^{-\frac{1}{2}} \prod_{i=1}^3 L_i^{\frac{1}{2}} \| f_{i,N_i,L_i} \|_2.
\end{equation*}
This is the content of Lemma \ref{lem:CFBilinearStrichartz}. The above display gives \eqref{eq:AuxDifferenceEnergyEstimate} for $s \geqslant s_1(\alpha,2\varepsilon,\R^2)$.

\smallskip

Lastly, we estimate the third term in \eqref{eq:ParaproductDifferencesFKPIR2}. After the usual time and modulation localization, we need to prove for $1 \lesssim N_1 \lesssim N_2$ and $L_i \gtrsim N_2^{5-2\alpha + \varepsilon}$:
\small
\begin{equation*}
N_2^{5-2\alpha+\varepsilon} \int (f_{1,N_1,L_1} * f_{2,N_2,L_2}) f_{3,N_3,L_3} d\xi d\eta d\tau \lesssim N_1^{-\frac{1}{2}} N_2^{-\frac{1}{2}} N_2^{s_2(\alpha,\varepsilon,\R^2)} \prod_{i=1}^3 L_i^{\frac{1}{2}} \| f_{i,N_i,L_i} \|_2.
\end{equation*}
\normalsize
Summaation of the above in the spatial frequencies leads to \eqref{eq:AuxDifferenceEnergyEstimate} for $s \geqslant s_1(\alpha,2\varepsilon,\R^2)$. The above display follows from Lemma \ref{lem:TrilinearConvolutionDispersionGeneralized}.
The proof of \eqref{eq:AuxDifferenceEnergyEstimate} is complete.
%
%

\medskip

It remains to prove \eqref{eq:EnergyEstimateDifferencesFKPIR2B}:
\begin{equation*}
\| v \|^2_{B^r(T)} \lesssim \| v(0) \|^2_{H^{r,0}(\R^2)} + \| v \|^3_{F^r(T)} + \| v \|_{\bar{F}^{-\frac{1}{2}}(T)} \| v \|_{F^r(T)} \| u_2 \|_{F^{2r+\frac{1}{2}}(T)}
\end{equation*}
for $r \geqslant s_1(\alpha,\varepsilon,\R^2)$.
To this end, we invoke again the fundamental theorem of calculus to find for $N \gg 1$
\begin{equation*}
\begin{split}
&\quad \| P_N v(t) \|_{L^2}^2 \\
 &= \|P_N v(0) \|^2_{L^2} + \int_0^t \int_{\R^2} P_N v \partial_x (P_N(v(u_1+u_2))) dx ds \\
&= \|P_N v(0) \|^2_{L^2} + \int_0^t \int_{\R^2} P_N v \partial_x (P_N(v^2)) dx ds + \int_0^t \int_{\R^2} P_N v \partial_x (P_N(v u_2)) dx ds.
\end{split}
\end{equation*}
The first integral can be estimated like solutions in \eqref{eq:EnergyEstimateSolutionFKPIR2}. It suffices to estimate the second integral. We carry out a paraproduct decomposition:
\begin{equation*}
P_N(v u_2) = P_N(v P_{\ll N} u_2) + P_N(P_{\ll N} v \cdot u_2) + P_N(P_{\gtrsim N} v \cdot P_{\gtrsim N} u_2).
\end{equation*}
The first term can be estimated like in \eqref{eq:EnergyEstimateSolutionFKPIR2} by integration by parts. We turn to the second term, for which we need to show estimates after adding time localization and dyadic frequency localization for $N_1 \sim N_3 \gg N_2$, $N_1 \gtrsim 1$:
\begin{equation*}
N_1^{5-2\alpha + \varepsilon} \iint_{\R \times \R^2} \gamma^3(N_1 (t-t_I)) P_{N_1} v P_{N_2} v P_{N_3} u dx dt.
\end{equation*}
After changing to Fourier space and modulation localization, it suffices to show
\begin{equation*}
\int (f_{1,N_1,L_1} * f_{2,N_2,L_2} ) f_{3,N_3,L_3} d\xi d\eta d\tau \lesssim N_1^{\frac{1}{2}-\frac{3\alpha}{4}} N_2^{-\frac{1}{2}} \prod_{i=1}^3 L_i^{\frac{1}{2}} \| f_{i,N_i,L_i} \|_2.
\end{equation*}
This establishes \eqref{eq:EnergyEstimateDifferencesFKPIR2B} for $r \geqslant s_1(\alpha,2\varepsilon,\R^2)$.
The above display is again Lemma \ref{lem:TrilinearConvolutionDispersionGeneralized}.

For the last term we need to estimate for
$N_1 \ll N_2 \sim N_3$, $N_1 \gtrsim 1$:
\begin{equation*}
N_2^{5-2\alpha + \varepsilon} \iint_{\R \times \R^2} \gamma^3(N_1 (t-t_I)) P_{N_1} v P_{N_2} v P_{N_3} u dx dt.
\end{equation*}
Like above we change to Fourier space and carry out a modulation localization such that it suffices to show
\begin{equation*}
\int (f_{1,N_1,L_1} * f_{2,N_2,L_2} ) f_{3,N_3,L_3} d\xi d\eta d\tau \lesssim N_2^{\frac{1}{2}-\frac{3\alpha}{4}} N_1^{-\frac{1}{2}} \prod_{i=1}^3 L_i^{\frac{1}{2}} \| f_{i,N_i,L_i} \|_2
\end{equation*}
as this will imply \eqref{eq:EnergyEstimateDifferencesFKPIR2B} for $r \geqslant s_1(\alpha,2\varepsilon,\R^2)$. The above display is once more established by Lemma \ref{lem:TrilinearConvolutionDispersionGeneralized}. The proof is complete.

$\hfill \Box$

\section{Quasilinear well-posedness on the cylinder}
\label{section:LWPFKPICylinder}
In this section we show the following result on quasilinear local well-posedness on the cylinder. Recall the definition of $s_1(\alpha,\varepsilon,\R \times \T)$:
\begin{equation*}
s_1(\alpha,\varepsilon,\R \times \T) = \begin{cases}
 \frac{11}{6} - \frac{2 \alpha}{3} + \varepsilon, &\quad 2 < \alpha \leqslant \frac{11}{4} \text{ and } \varepsilon > 0,\\
  0, &\quad \frac{11}{4} < \alpha \leqslant 5.
  \end{cases}
\end{equation*}

\begin{theorem}
\label{thm:QuasilinearLWPCylinder}
Let $\D = \R \times \T$, $\alpha \in (2,5)$, and $ s \geqslant s_1(\alpha,\varepsilon,\D)$ with $s_1$ defined in \eqref{eq:RegSolution} and $\varepsilon > 0$. Then, there is a data-to-solution mapping $S_T^s:H^{s,0}(\D) \to C_T H^{s,0}(\D)$, which assigns real-valued initial data $u_0 \in H^{s,0}(\D)$ to solutions $u \in C_T H^{s,0}$ to \eqref{eq:FKPI} and $T=T(\| u_0 \|_{H^{s,0}})$ is lower semi-continuous. It holds $T(\| u_0 \|_{H^{s,0}}) \gtrsim 1$ for $\| u_0 \|_{H^{s,0}(\D)} \downarrow 0$.
\end{theorem}

\subsection{Outline of the proof.}

Like in the previous section on KP-I equations posed on $\R^2$, we use short-time Fourier restriction. 

\smallskip

 We have the following short-time bilinear estimates:
\begin{proposition}
\label{prop:ShorttimeBilinearEstimateGeneral}
Let $\alpha \in (2,5)$ and $T=T(N)=N^{-\frac{5-\alpha}{3}-\varepsilon(\alpha)}$ with
\begin{equation*}
0< \varepsilon < \frac{4 \alpha}{6} - \frac{4}{3}.
\end{equation*}
 Then the estimate
 \begin{equation}
 \label{eq:ShorttimeEstimateI}
 \| \partial_x(u v) \|_{\mathcal{N}^{s}(T)} \lesssim T^{\kappa} ( \| u \|_{F^s(T)} \| v \|_{F^0(T)} + \| u \|_{F^0(T)} \| v \|_{F^s(T)}
 \end{equation}
  holds for $s \geqslant 0$. 
  
\medskip  
  
  Secondly, the estimate
 \begin{equation}
 \label{eq:AsymmetricShorttimeBilinearEstimate}
\| \partial_x(u v) \|_{\bar{\mathcal{N}}^{-\frac{1}{2}}(T)} \lesssim T^{\kappa} \| u \|_{F^{s}(T)} \| v \|_{ \bar{F}^{-\frac{1}{2}}(T)}
\end{equation}
holds for $s \geqslant s_1(\alpha,\varepsilon,\R \times \T)$. 
\end{proposition}
Note that we obtain the following as a special case:
\begin{equation*}
\| \partial_x (u_1 u_2) \|_{\mathcal{N}
^0(T)} \lesssim T^{\kappa} \| u_1 \|_{F^0(T)} \| u_2 \|_{F^0(T)}.
\end{equation*}
The proof will be given in Subsection \ref{section:ShorttimeBilinearEstimatesCylinder}. The veracity of \eqref{eq:ShorttimeEstimateI} will give an upper bound for $\varepsilon$.

A further ingredient is the short-time energy propagation. We show energy estimates for the solution:
\begin{equation}
\label{eq:EnergyEstimateSolutionsCylinder}
\| u \|^2_{B^s(T)} \lesssim \| u_0 \|_{H^{s,0}}^2 + T^{\delta} \| u \|_{F^s(T)}^2 \| u \|_{F^{s_1}(T)}
\end{equation}
for $s \geqslant s_1(\alpha,5\varepsilon,\R \times \T)$ and $\delta = \delta(\varepsilon) > 0$.

\smallskip

Secondly, we prove energy estimates for the difference of solutions. Let $u_i$, $i=1,2$ be two solutions to \eqref{eq:FKPI} and let $v = u_1 - u_2$. We observe that $v$ solves:
\begin{equation*}
\partial_t v - \partial_x D_x^\alpha v - \partial_{x}^{-1} \partial_y^2 v =  \frac{\partial_x(v (u_1+u_2))}{2}.
\end{equation*}
We shall estimate $v$ in weighted Sobolev norms of negative regularity. Define
\begin{equation*}
\| f \|^2_{\bar{H}^{-\frac{1}{2}}} = \int_{\R \times \Z} (1+|\xi|^{-1})^{1+2 \varepsilon} \langle \xi \rangle^{-1} |\hat{f}(\xi,\eta)|^2 d\xi (d\eta)_1. 
\end{equation*}

\smallskip


The energy estimates for solutions read as follows:
\begin{proposition}
\label{prop:EnergyEstimateSolutionsCylinder}
Suppose that $s \geqslant s_1$.
Then the estimate \eqref{eq:EnergyEstimateSolutionsCylinder} holds true.
\end{proposition}

We formulate the estimates for differences of solutions:
\begin{proposition}
\label{prop:EnergyEstimatesDifference}
Let $\alpha \in (2,5)$, and $s \geqslant s_1$. Then choosing $\varepsilon$ small enough, for $T=T(N)=N^{-\frac{5-\alpha}{3}-\varepsilon}$, the following estimate holds for some $\delta = \delta(\varepsilon) > 0$:
\begin{equation}
\label{eq:DifferenceEnergyEstimate}
\| v \|_{\bar{B}^{-\frac{1}{2}}(T)}^2 \lesssim \| v(0) \|_{\bar{H}^{-\frac{1}{2}}}^2 + T^{\delta} \| v \|_{\bar{F}^{-\frac{1}{2}}(T)}^2 \big( \| u_1 \|_{F^{s}(T)} + \| u_2 \|_{F^{s}(T)} \big).
\end{equation}
Let $r \geqslant s_1(\alpha,5\varepsilon,\R \times \T)$. The following estimate holds for some $\delta = \delta(\varepsilon) > 0$:
\begin{equation}
\label{eq:DifferenceEnergyEstimateII}
\| v \|_{B^r(T)}^2 \lesssim \| v(0) \|_{H^{r,0}}^2 + T^{\delta} \| v \|_{F^{r}(T)}^3 + T^{\delta} \| v \|_{\bar{F}^{-\frac{1}{2}}(T)} \| v \|_{F^r(T)} \| u_2 \|_{F^{2r+\frac{1}{2}}(T)} .
\end{equation}
\end{proposition}

With these estimates at hand, the proof follows along the same lines of Sections \ref{section:QuasilinearLWPKPIR2} and \ref{section:LWPFKPIR2}. Details are omitted to avoid repetition.

\subsection{A trilinear convolution estimate}

In this section we show the analog of Lemma \ref{lem:TrilinearConvolutionDispersionGeneralized}:
\begin{lemma}
\label{lem:TrilinearConvolutionEstimateCylinder}
Let $\alpha \geqslant 2$, $N_i \in 2^{\Z}$ with $N_2 \lesssim N_1 \sim N_3$, $N_1 \gtrsim 1$, and $L_{\text{med}} \gtrsim (1 \vee N_1^{\frac{5-\alpha}{3}+\varepsilon})$. Let $f_{i,N_i,L_i}: \R \times \R \times \Z \to \R_{\geqslant 0}$ with $\text{supp}(f_{i,N_i,L_i}) \subseteq D_{N_i,L_i}$ for $i \in \{1,2,3\}$. Then the following estimate holds for some $0<\delta\leqslant \delta'$ :
\begin{equation}
\label{eq:TrilinearConvolutionEstimateCylinder}
\int (f_{1,N_1,L_1} * f_{2,N_2,L_2}) f_{3,N_3,L_3} d\xi d\eta d\tau \lesssim N_2^{-\frac{1}{2}+\delta} N_1^{-\frac{1}{3}-\frac{\alpha}{3}-\frac{\varepsilon}{4}+\delta} L_{\max}^{-\delta} \prod_{i=1}^3 L_i^{\frac{1}{2}} \| f_{i,N_i,L_i} \|_{L^2}.
\end{equation}
\end{lemma}
\begin{proof}
Firstly, we can suppose that $N_2 \gtrsim N_1^{-C}$ for some large $C$ since the estimate of tiny frequencies $N_2 \ll N_1^{-C}$ is immediate from Lemma \ref{lem:AlternativeBilinearStrichartzEstimate}.
\smallskip
Next, suppose that $L_{\max} \ll N_1^\alpha N_2$. In this case \eqref{eq:TrilinearConvolutionEstimateCylinder} is immediate from the nonlinear Loomis-Whitney inequality \eqref{eq:LoomisWhitneyCylinder} and the lower bound on $L_{\text{med}}$:
\begin{equation*}
\begin{split}
&\quad \int (f_{1,N_1,L_1} * f_{2,N_2,L_2}) f_{3,N_3,L_3} d\xi d\eta d\tau \\
&\lesssim N_2^{-\frac{1}{2}} N_1^{\frac{1}{2}-\frac{\alpha}{2}} (L_{\min} L_{\max})^{\frac{1}{2}} \langle L_{\text{med}} / N_1^{\alpha/2} \rangle^{\frac{1}{2}} \prod_{i=1}^3 \| f_{i,N_i,L_i} \|_{L^2}.
\end{split}
\end{equation*}

Next, suppose that $L_{\max} \gtrsim N_1^\alpha N_2$. In case $L_2 = L_{\max}$ we carry out a dyadic decomposition in the transversality
\begin{equation*}
D \sim \big| \frac{\eta_1}{\xi_1} - \frac{\eta_3}{\xi_3} \big|.
\end{equation*}

We remark that we have by convolution constraint for $(\xi_i,\eta_i) \in \pi_{\xi,\eta}\text{supp}(f_{i,N_i,L_i})$:
\begin{equation*}
\big| \frac{\eta_1}{\xi_1} - \frac{\eta_2}{\xi_2} \big| \lesssim \big( \frac{L_{\max}}{N_2} \big)^{\frac{1}{2}}.
\end{equation*}
Indeed, we have
\begin{equation*}
\Omega_{\alpha} = (\tau - \omega_\alpha(\xi,\eta)) - (\tau_1 - \omega_{\alpha}(\xi_1,\eta_1)) - (\tau_2 - \omega_\alpha(\xi_2,\eta_2)).
\end{equation*}
This gives the claimed bound on the transversality:
\begin{equation*}
\frac{|\eta_1 \xi_2 - \eta_2 \xi_1|^2}{\xi_1 \xi_2 (\xi_1 + \xi_2)} \lesssim L_{\max} \Rightarrow \big| \frac{\eta_1}{\xi_1} - \frac{\eta_2}{\xi_2} \big| \lesssim \big( \frac{L_{\max}}{N_2} \big)^{\frac{1}{2}}.
\end{equation*}

The summation in $D$ incurs a factor of $\log(L_{\max})+\log(N_1)$ by the lower bound on $N_2$. For $D \lesssim N_1^{\frac{\alpha}{2}}$ we can apply H\"older's inequality and two $L^4$-Strichartz estimates from Proposition \ref{prop:StrichartzCylinder} to find
\begin{equation*}
\begin{split}
&\quad \int (f_{1,N_1,L_1} * f_{2,N_2,L_2}) f_{3,N_3,L_3} d\xi d\eta d\tau \\
 &\leqslant \| f_{2,N_2,L_2} \|_{L^2_{\tau,\xi,\eta}} \prod_{i=1,3} \| \mathcal{F}^{-1}_{t,x,y} [f_{i,N_i,L_i}] \|_{L^4_{t,x,y}} \\
&\lesssim_\varepsilon (N_1^\alpha N_2)^{-\frac{1}{2}+\delta} L_{\max}^{\frac{1}{2}-\delta} N_1^{\frac{2-\alpha}{4}+\varepsilon} \| f_{2,N_2,L_2} \|_{L^2} \prod_{i=1,3} \| f_{i,N_i,L_i} \|_{L^2}.
\end{split}
\end{equation*}
Here we use an almost orthogonal decomposition to
\begin{equation*}
\big| \frac{\eta_i}{\xi_i} - A \big| \lesssim N_1^{\frac{\alpha}{2}}    
\end{equation*}
such that Proposition \ref{prop:StrichartzCylinder} becomes applicable.

For $D \gtrsim N_1^{\frac{\alpha}{2}}$ we can apply a bilinear Strichartz estimate from Proposition \ref{prop:BilinearStrichartzGeneral} to find
\begin{equation*}
\begin{split}
&\quad \int (f_{1,N_1,L_1} * f_{2,N_2,L_2} ) f_{3,N_3,L_3} d\xi d\eta d\tau \\
&\leqslant \| f_{2,N_2,L_2} \|_{L^2} \| f_{1,N_1,L_1} * \tilde{f}_{3,N_3,L_3} \|_{L^2_{\tau,\xi,\eta}} \\
&\lesssim (N_1^\alpha N_2)^{-\frac{1}{2}} L_{\max}^{\frac{1}{2}} L_{\min}^{\frac{1}{2}} N_2^{\frac{1}{2}} \langle L_{\text{med}} / N_1^{\frac{\alpha}{2}} \rangle^{\frac{1}{2}} \prod_{i=1}^3 L_i^{\frac{1}{2}} \| f_{i,N_i,L_i} \|_2 \\
&\lesssim N_1^{-\frac{5}{6}-\frac{\alpha}{3}-\frac{\varepsilon}{2}} \prod_{i=1}^3 L_i^{\frac{1}{2}} \| f_{i,N_i,L_i} \|_2.
\end{split}
\end{equation*}

Next, we consider $L_1 = L_{\max} \gtrsim N_1^\alpha
 N_2$ (which covers as well $L_3 = L_{\max}$ by symmetry). We apply H\"older's inequality and Lemma \ref{lem:CFBilinearCylinder} to find
 \begin{equation*}
 \begin{split}
&\quad \int (f_{1,N_1,L_1} * f_{2,N_2,L_2}) f_{3,N_3,L_3} d\xi d\eta d\tau \\
 &\leqslant \| f_{1,N_1,L_1} \|_2 \| f_{2,N_2,L_2} * \tilde{f}_{3,N_3,L_3} \|_{L^2_{\tau,\xi,\eta}} \\
 &\lesssim (N_1^\alpha N_2)^{-\frac{1}{2}+\delta} L_{\max}^{\frac{1}{2}-\delta} \log(L_{\max}) N_2^{\frac{1}{2}} L_{\min}^{\frac{1}{2}} \langle L_{\text{med}} / N_1^{\frac{\alpha}{2}} \rangle^{\frac{1}{2}} \prod_{i=1}^3 \| f_{i,N_i,L_i} \|_2.
 \end{split}
 \end{equation*}
The claim follows by the lower bound on $L_{\text{med}}$.

\end{proof}
We remark that by interpolation with Lemma \ref{lem:AlternativeBilinearStrichartzEstimate} we can always obtain a small power of $N_2$ and smoothing in $L_{\max}$:
\begin{equation*}
\int (f_{1,N_1,L_1} * f_{2,N_2,L_2} ) f_{3,N_3,L_3} d\xi d\eta d\tau \lesssim N_2^{-\frac{1}{2}+} N_1^{-\frac{1}{3}-\frac{\alpha}{3}-\frac{\varepsilon}{4}} L_{\max}^{0-} \prod_{i=1}^3 L_i^{\frac{1}{2}} \| f_{i,N_i,L_i} \|_2.
\end{equation*}

\subsection{Short-time bilinear estimates}
\label{section:ShorttimeBilinearEstimatesCylinder}
In this section we shall obtain short-time nonlinear estimates 
\begin{equation}
\label{eq:ShorttimeNonlinearEstimate}
\|\partial_x (u_1 u_2) \|_{\mathcal{N}^{s}} \lesssim T^{\kappa} ( \| u_1 \|_{F^s} \| u_2 \|_{F^0} + \| u_1 \|_{F^0} \| u_2 \|_{F^s} )
\end{equation}
and modified estimates in weighted norms for differences of solutions. The small power of $T$ is a consequence of Lemma \ref{lem:TradingModulationRegularity} from a smoothing in the maximal modulation $L_{\max}^{0-}$, see the remark above. We turn to the proof of Proposition \ref{prop:ShorttimeBilinearEstimateGeneral}.

\subsubsection{High $\times$ Low $\to$ High interaction}
 
In this case we suppose that $N \sim N_1 \gtrsim N_2$.
 We focus in the following on $N \gtrsim 1$. 
The case $N \lesssim 1$ can be estimated via Lemma \ref{lem:AlternativeBilinearStrichartzEstimate}.
The following reductions to convolution estimates in Fourier variables are carried out like in previous sections.
We shall establish an estimate
\begin{equation}
\label{eq:DyadicallyLocalizedShorttimeEstimate}
\| P_N \partial_x (P_{N_1} u_1 P_{N_2} u_2) \|_{\mathcal{N}_N(T)} \lesssim C(N,N_1,N_2) \| P_{N_1} u_1 \|_{F_{N_1}(T)} \| P_{N_2} u_2 \|_{F_{N_2}(T)}.
\end{equation}
The claim then follows from dyadic summations.

\medskip

In the following we consider extensions $P_{N_i} \tilde{u}_i$ of $P_{N_i} u_i$ such that
\begin{equation*}
\| P_{N_i} \tilde{u}_i \|_{F_{N_i}} \leqslant 2 \| P_{N_i} u_i \|_{F_{N_i}(T)}.
\end{equation*}
To lighten notations, we shall again redenote $\tilde{u}_i \to u_i$.

We use the definition of the $\mathcal{N}_N$-norm to bound
the left hand side of \eqref{eq:DyadicallyLocalizedShorttimeEstimate} by
\begin{equation*}
\begin{split}
&\sup_{t_N \in \R} \big\| (\tau - \omega_\alpha(\xi,\eta) + i N^{\frac{5-\alpha}{3}+\varepsilon} )^{-1} \xi 1_{A_N}(\xi) \mathcal{F}[u_{1,N_1} \cdot \eta_0(N^{\frac{5-\alpha}{3}+\varepsilon}(t-t_N))] \\
&\quad * \mathcal{F}[u_{2,N_2} \cdot \eta_0(N^{\frac{5-\alpha}{3}+\varepsilon}(t-t_N))] \big\|_{X_N}.
\end{split}
\end{equation*}
We define
\begin{equation*}
f_{i,N_i} = \mathcal{F}[u_{i,N_i} \cdot \eta_0(N^{\frac{5-\alpha}{3}+\varepsilon}(t-t_N))].
\end{equation*}
Consider $L_1,L_2 \geqslant N^{\frac{5-\alpha}{3}+\varepsilon}$ and $f_{i,N_i,L_i}: \R \times \R \times \Z \to \R_+$ are functions supported in $D_{N_i,L_i}$ and for $L_i = N^{\frac{5-\alpha}{3}+\varepsilon}$ in $D_{N, \leqslant L_i}$. By the function space properties it suffices to
%
%
 obtain the following estimate for $L \geqslant N_1^{\frac{5-\alpha}{3}+\varepsilon}$:
\begin{equation*}
 N_1 \| 1_{D_{\alpha,N,L}} (f_{1,N_1,L_1} * f_{2,N_2,L_2}) \|_{L^2_{\tau,\xi,\eta}} \lesssim C(N_1,N_2) L_{\max}^{0-} L^{\frac{1}{2}} \prod_{i=1}^2 L_i^{\frac{1}{2}} \| f_i \|_{L^2_{\tau,\xi,\eta}}.
\end{equation*}
We shall see that for the High-Low interaction, any $\varepsilon > 0$ suffices to obtain the estimate in the above display with $C(N_1,N_2) = N_1^{-\varepsilon'}$. Then
it follows from dyadic summation that \eqref{eq:ShorttimeEstimateI} and \eqref{eq:AsymmetricShorttimeBilinearEstimate} are true. As seen in the previous section, suitable interpolation with Lemma \ref{lem:AlternativeBilinearStrichartzEstimate} allows us always to squeeze out a factor $N_{\min}^{0+} L_{\max}^{0-}$, for which reason we obtain summability in the low frequency and modulation parameters. This will often be omitted not to overburden the notation.
 We remark that it is the High $\times$ High $\to$ Low interaction ($N \ll N_1 \sim N_2$) which limits the regularity $s \geqslant s_1$ in \eqref{eq:AsymmetricShorttimeBilinearEstimate}.

\medskip

Note that by the time localization we have $L_i \gtrsim N_1^{\frac{5-\alpha}{3}+\varepsilon}$. Below we shall distinguish the cases of
\begin{enumerate}
\item very low frequencies: $N_2 \lesssim 1\lesssim N_1$, \quad $N_1^\alpha N_2 \lesssim N_1^{\frac{5-\alpha}{3}}$,
 \item resonant case: $L_{\max} \ll N_1^\alpha N_2$, \quad $N_1^\alpha N_2 \gtrsim N_1^{\frac{5-\alpha}{3}}$,
\item non-resonant case: $L_{\max} \gtrsim N_1^\alpha N_2$, \quad $N_1^\alpha N_2 \gtrsim N_1^{\frac{5-\alpha}{3}}$.
\end{enumerate}

\medskip

\noindent\textbf{(1) Very low frequencies} ($N_2 \lesssim 1 \lesssim N_1, N_1^\alpha N_2 \lesssim N_1^{\frac{5-\alpha}{3}}$): We estimate
\begin{equation*}
N_1 \| 1_{D_{N,L}} (f_{1,N_1,L_1} * f_{2,N_2,L_2}) \|_{L^2_{\tau,\xi,\eta}}
\end{equation*}
by bilinear Strichartz estimates recorded in Lemma \ref{lem:AlternativeBilinearStrichartzEstimate}. This gives the (crude) estimate
\begin{equation*}
\| 1_{D_{N,L}} (f_{1,N_1,L_1} * f_{2,N_2,L_2} ) \|_{L^2_{\tau,\xi,\eta}} \lesssim N_2^{\frac{1}{2}} L_{\max}^{0-} \prod_{i=1}^2 L_i^{\frac{1}{2}} \| f_{i,N_i,L_i} \|_2.
\end{equation*}
Taking into account the derivative loss $N_1$ and $L \geqslant N_1^{\frac{5-\alpha}{3}+\varepsilon}$ we obtain
\begin{equation*}
N_1  \| 1_{D_{N,L}} (f_{1,N_1,L_1} * f_{2,N_2,L_2} ) \|_{L^2_{\tau,\xi,\eta}} \lesssim N_1^{1-\frac{5-\alpha}{6}-\frac{\varepsilon}{2}} N_2^{\frac{1}{2}} L_{\max}^{0-}  L^{\frac{1}{2}} \prod_{i=1}^2 L_i^{\frac{1}{2}} \| f_{i,N_i,L_i} \|_2.
\end{equation*}
Summing over $N_2 \lesssim N_1^{\frac{5-\alpha}{3}-\alpha}$ and $L$ with weight $L^{-\frac{1}{2}}$ this estimate is sufficient.

In the following cases, we shall consider exclusively $N_2 N_1^\alpha \gtrsim N_1^{\frac{5-\alpha}{3}}$, which imposes a lower bound on $N_2$.

\medskip

\noindent\textbf{(2) Resonant case} ($L_{\max} = \max(L,L_1,L_2) \ll N_1^\alpha N_2, \quad 1 \lesssim N \sim N_1, \quad N_2 \gtrsim N_1^{-\alpha + \frac{5-\alpha}{3}}$): We consider two sub-cases:\\
(i) $N_1^{\frac{4-2\alpha}{3}} \lesssim N_2$: Applying the nonlinear Loomis-Whitney inequality \eqref{eq:LoomisWhitneyCylinder} gives
\begin{equation*}
\begin{split}
&\quad N_1 \sum_{N_1^{\frac{5-\alpha}{3}+\varepsilon} \leqslant L \leqslant N_1^\alpha N_2} L^{-\frac{1}{2}} \| 1_{D_{N,L}} (f_{1,N_1,L_1} * f_{2,N_2,L_2} ) \|_{L^2_{\tau,\xi,\eta}} \\
&\lesssim \frac{N_1^{\frac{3}{2}-\frac{\alpha}{2}}}{N_2^{\frac{1}{2}}} \sum_{N_1^{\frac{5-\alpha}{3}+\varepsilon} \leqslant L \leqslant N_1^\alpha N_2} L^{-\frac{1}{2}} L^{\frac{1}{2}} L_1^{\frac{1}{2}} L_2^{\frac{1}{2}} N_1^{-\frac{5-\alpha}{6}-\varepsilon} \| f_{1,N_1,L_1} \|_{L^2_{\tau,\xi,\eta}} \| f_{2,N_2,L_2} \|_{L^2_{\tau,\xi,\eta}} \\
&\lesssim \frac{N_1^{\frac{2-\alpha}{3}}}{N_2^{\frac{1}{2}}} \log(N_1) N_1^{-\varepsilon} \prod_{i=1}^2 L_i^{\frac{1}{2}} \| f_{i,N_i,L_i} \|_{L^2}.
\end{split}
\end{equation*}

This yields a favorable estimate provided that $N_2 \gtrsim N_1^{-\kappa}$ for $\kappa = \frac{2(\alpha-2)}{3}$. 

\medskip

\noindent(ii) $N_2 \ll N_1^{\frac{4-2\alpha}{3}}$: We find from the short-time bilinear Strichartz estimate recorded in Proposition \ref{prop:BilinearStrichartzGeneral}
\begin{equation*}
\begin{split}
&\quad N_1 \sum_{N_1^{\frac{5-\alpha}{3}+\varepsilon} \leqslant L \leqslant N_1^\alpha N_2} L^{-\frac{1}{2}} \| 1_{D_{N,L}} (f_{1,N_1,L_1} * f_{2,N_2,L_2}) \|_{L^2_{\tau,\xi,\eta}} \\
 &\lesssim N_2^{\frac{1}{2}} N_1^{1-\frac{5-\alpha}{3}-\varepsilon} \prod_{i=1}^2 L_i^{\frac{1}{2}} \| f_{i,N_i,L_i} \|_{L^2_{\tau,\xi,\eta}} \lesssim N_1^{-\varepsilon} \prod_{i=1}^2 L_i^{\frac{1}{2}} \| f_{i,N_i,L_i} \|_{L^2_{\tau,\xi,\eta}}.
\end{split}
\end{equation*}

\medskip

\noindent \textbf{(3) Non-resonant case} ($L_{\max} \geqslant N_1^\alpha N_2$, $N_2 \gtrsim N_1^{-\alpha + \frac{5-\alpha}{3}}$): We consider further sub-cases:

\smallskip

\noindent (i) $N_2\gtrsim 1$: Depending on $L_{\max}$, we have the following cases:\\
$\bullet~ L_{\max}=L$: The case $L_1 = L_{\max}$ (the high frequency is at high modulation) can be handled by dual arguments. Indeed, there is additional smoothing $N_1^{-\varepsilon}$ in the estimates below, which allows us to handle the summation over $L$.

We remark that we have by convolution constraint for $(\xi_i,\eta_i) \in \pi_{\xi,\eta}\text{supp}(f_{i,N_i,L_i})$:
\begin{equation*}
\big| \frac{\eta_1}{\xi_1} - \frac{\eta_2}{\xi_2} \big| \lesssim \big( \frac{L_{\max}}{N_2} \big)^{\frac{1}{2}}.
\end{equation*}
Indeed, we have
\begin{equation*}
\Omega_{\alpha} = (\tau - \omega_\alpha(\xi,\eta)) - (\tau_1 - \omega_{\alpha}(\xi_1,\eta_1)) - (\tau_2 - \omega_\alpha(\xi_2,\eta_2)).
\end{equation*}
This gives the claimed bound on the transversality:
\begin{equation*}
\frac{|\eta_1 \xi_2 - \eta_2 \xi_1|^2}{\xi_1 \xi_2 (\xi_1 + \xi_2)} \lesssim L_{\max} \Rightarrow \big| \frac{\eta_1}{\xi_1} - \frac{\eta_2}{\xi_2} \big| \lesssim \big( \frac{L_{\max}}{N_2} \big)^{\frac{1}{2}}.
\end{equation*}
For this reason we can apply a bilinear Strichartz estimate from Lemma \ref{lem:CFBilinearCylinder} and take into account the minimum modulation size to find
\begin{equation*}
\begin{split}
&\quad N_1 L^{-\frac{1}{2}} \| 1_{D_{N,L}} (f_{1,N_1,L_1} * f_{2,N_2,L_2} ) \|_{L^2_{\tau,\xi,\eta}} \\
&\lesssim N_1 (N_1^\alpha N_2)^{-\frac{1}{2}+\delta} N_2^{\frac{1}{2}} N_1^{-\frac{5-\alpha}{6}-\frac{\varepsilon}{2}} \prod_{i=1}^2 L_i^{\frac{1}{2}} \| f_{i,N_i,L_i} \|_2 \\
&\lesssim N_2^\delta N_1^{\frac{1}{6}-\frac{\alpha}{3}+\delta-\frac{\varepsilon}{2}} \prod_{i=1}^2 L_i^{\frac{1}{2}} \| f_{i,N_i,L_i} \|_2.
\end{split}
\end{equation*}
This is acceptable choosing $\delta=\delta(\varepsilon) > 0$ small enough.

\medskip

\noindent $\bullet~L_{\max}=L_2$: If $L \gtrsim N_1^\alpha N_2$ we can apply a bilinear Strichartz estimate from Lemma \ref{lem:AlternativeBilinearStrichartzEstimate} to find
\begin{equation*}
\begin{split}
&\quad \sum_{L \gtrsim N_1^\alpha N_2} L^{-\frac{1}{2}} \| 1_{D_{N,L}} (f_{1,N_1,L_1} * f_{2,N_2,L_2} ) \|_{L^2_{\tau,\xi,\eta}} \\
 &\lesssim (N_1^\alpha N_2)^{-\frac{1}{2}} N_2^{\frac{3}{4}} (N_1^\alpha N_2)^{-\frac{1}{4}+\delta} \prod_{i=1}^2 L_i^{\frac{1}{2}} \| f_{i,N_i,L_i} \|_2 \\
&\lesssim N_1^{-\frac{3\alpha}{4}+\delta \alpha} N_2^\delta \prod_{i=1}^2 L_i^{\frac{1}{2}} \| f_{i,N_i,L_i} \|_2.
\end{split}
\end{equation*}
Hence, we suppose $L \ll N_1^\alpha N_2$ in the following.

An application of Lemma \ref{lem:TrilinearConvolutionEstimateCylinder} after invoking duality gives
\begin{equation*}
\| 1_{D_{N,L}} (f_{1,N_1,L_1} * f_{2,N_2,L_2}) \|_{L^2_{\tau,\xi,\eta}} \lesssim L_{\max}^{-\delta} N_2^{-\frac{1}{2}+\delta} N_1^{-\frac{1}{3}-\frac{\alpha}{3}-\frac{\varepsilon}{2}} \prod_{i=1}^2 L_i^{\frac{1}{2}} \| f_{i,N_i,L_i} \|_2.
\end{equation*}
This is acceptable for $N_2 \gtrsim 1$.

\medskip

\noindent (ii) $N_2 \lesssim 1$: \\
$\bullet~L_{\max} = L~(\text{or } L_1)$: In this case, the previous estimates suffice.\\

\noindent $\bullet~L_{\max}=L_2 \gtrsim N_1^\alpha N_2 \gtrsim N_1^{\frac{\alpha}{2}}$. We impose an upper bound $L \lesssim N_1^3$. In this case we estimate by duality and Lemma \ref{lem:AlternativeBilinearStrichartzEstimate}:
\begin{equation}
    \label{eq:HighLowHighEstimateAuxI}
    \begin{split}
        &\quad \| 1_{D_{N,L}} ( f_{1,N_1,L_1} * f_{2,N_2,L_2}) \|_{L^2_{\tau,\xi,\eta}} \\
        &= \sup_{\| g_{N,L} \|_{L^2_{\tau,\xi,\eta}=1}} \int g_{N,L} (f_{1,N_1,L_1} * f_{2,N_2,L_2}) \\
        &= \sup_{\| g_{N,L} \|_{L^2_{\tau,\xi,\eta}=1}} \int (\tilde{g}_{N,L} * f_{1,N_1,L_1} ) f_{2,N_2,L_2} \\
        &\lesssim  \sup_{\| g_{N,L} \|_{L^2_{\tau,\xi,\eta}=1}} \| \tilde{g}_{N,L} * f_{1,N_1,L_1} \|_{L^2_{\tau,\xi,\eta}} \| f_{2,N_2,L_2} \|_{L^2_{\tau,\xi,\eta}} \\
        &\lesssim N_2^{\frac{1}{2}} N_1^{\frac{1}{4}} (L L_1)^{\frac{1}{2}} N_1^{-\frac{(5-\alpha)}{12} - \frac{\varepsilon}{2}} L_2^{\frac{1}{2}} (N_1^\alpha N_2)^{-\frac{1}{2}} \prod_{i=1}^2 \| f_{i,N_i,L_i} \|_{L^2_{\tau,\xi,\eta}}.
    \end{split}
\end{equation}

Summing \eqref{eq:HighLowHighEstimateAuxI} over $N_1^\alpha N_2 \lesssim L \lesssim N_1^3$ and $N_1^{-C} \lesssim N_2 \lesssim 1$ gives
\begin{equation*}
\begin{split}
    &\quad N_1 \sum_{N_1^\alpha N_2 \lesssim L \lesssim N_1^3} L^{-\frac{1}{2}} \| 1_{D_{N,L}} (f_{1,N_1,L_1} * f_{2,N_2,L_2}) \|_{L^2_{\tau,\xi,\eta}} \\
    &\lesssim N_1^{\frac{10}{12}-\frac{5\alpha}{12}-\frac{\varepsilon}{2}} \log(N_1) \prod_{i=1}^2 L_i^{\frac{1}{2}} \| f_{i,N_i,L_i} \|_{L^2_{\tau,\xi,\eta}}.
    \end{split}
\end{equation*}
This is acceptable.

In case $L \gtrsim N_1^3$ we can apply Lemma \ref{lem:AlternativeBilinearStrichartzEstimate} on $f_{1,N_1,L_1}* f_{2,N_2,L_2}$:
\begin{equation*}
\begin{split}
    &\quad N_1 \sum_{L \gtrsim N_1^3} L^{-\frac{1}{2}} \| 1_{D_{N,L}} ( f_{1,N_1,L_1} * f_{2,N_2,L_2}) \|_{L^2_{\tau,\xi,\eta}} \\
    &\lesssim N_1^{-\frac{1}{2}} N_2^{\frac{3}{4}} \prod_{i=1}^2 L_i^{\frac{1}{2}} \| f_{i,N_i,L_i} \|_{L^2_{\tau,\xi,\eta}}.
\end{split}    
\end{equation*}
This is sufficient.

\smallskip

Interpolation with the bilinear Strichartz estimate from Lemma \ref{lem:AlternativeBilinearStrichartzEstimate} yields factors $N_2^{0+} L_{\max}^{0-}$, which finishes the argument for the High $\times$ Low $\to$ High interaction.

\subsubsection{High $\times$ High $\to$ Low interaction}
Here, reductions to frequency and modulation localized estimates require additional time localization. We shall estimate for $N_1 \sim N_2 \gg N$:
\begin{equation*}
\begin{split}
&\| P_N \partial_x (P_{N_1} u_1 P_{N_2} u_2) \|_{\mathcal{N}_N} = \sup_{t_N \in \R} \| (\tau - \omega_\alpha(\xi,\eta) + i N^{\frac{5-\alpha}{3}+\varepsilon}_+ )^{-1} \\
&\quad \times \xi \mathcal{F}_{t,x,y}[ \eta_0(N_+^{\frac{5-\alpha}{3}+\varepsilon}(t-t_N)) P_{N_1} u_1 P_{N_2} u_2 ] \|_{X_N}.
\end{split}
\end{equation*}
Since $N_1 \sim N_2 \gg N$, the time localization $N_+^{\frac{5-\alpha}{3}+\varepsilon}$ does not suffice for $N_1 \gg 1$. We need to additionally localize to time intervals to size $N_{1,+}^{-(\frac{5-\alpha}{3}+\varepsilon)}$. To this end, let $\gamma \in C^\infty_c(-2,2)$ with
\begin{equation*}
\sum_{k \in \Z} \gamma^2(t-k) \equiv 1.
\end{equation*}
We write
\begin{equation*}
\begin{split}
&\quad \eta_0(N_+^{\frac{5-\alpha}{3}+\varepsilon}(t-t_N) P_{N_1} u_1 P_{N_2} u_2 ) \\
&= \sum_{k \in \Z} \eta_0(N_+^{\frac{5-\alpha}{3}+\varepsilon}(t-t_N)) (\gamma(N_{1,+}^{\frac{5-\alpha}{3}+\varepsilon}t - k) P_{N_1} u_1) (\gamma(N_{1,+}^{\frac{5-\alpha}{3}+\varepsilon}t - k) P_{N_2} u_2)).
\end{split}
\end{equation*}
Note that
\begin{equation*}
\# \{ k \in \Z : \eta_0(N_+^{\frac{5-\alpha}{3}+\varepsilon}(t-t_N)) \gamma(N_{1,+}^{\frac{5-\alpha}{3}+\varepsilon} t -k) \not\equiv 0 \} \sim \Big(\frac{N_{1,+}}  {N_+}\Big)^{\frac{5-\alpha}{3}+\varepsilon}.
\end{equation*}
By Minkowski's inequality it will suffice to estimate $\Big(\frac{N_{1,+}}  {N_+}\Big)^{\frac{5-\alpha}{3}+\varepsilon}$ expressions of the kind (with estimate uniformly in $k$):
\small
\begin{equation*}
\begin{split}
&\quad \| (\tau - \omega_\alpha(\xi,\eta)+ i N_+^{\frac{5-\alpha}{3}+\varepsilon} )^{-1} i \xi 1_{A_N}(\xi) \\
&\quad \times \mathcal{F}_{t,x,y}[ \eta_0(N_+^{\frac{5-\alpha}{3}+\varepsilon}(t-t_N)) (\gamma(N_{1,+}^{\frac{5-\alpha}{3}+\varepsilon}(t-k) P_{N_1} u_1) (\gamma(N_{1,+}^{\frac{5-\alpha}{3}+\varepsilon} t -k) P_{N_2} u_2)] \|_{X_N}.
\end{split}
\end{equation*}
\normalsize
The time localization suffices to conclude
\begin{equation*}
\| \mathcal{F}_{t,x}[ \gamma(N_{1,+}^{\frac{5-\alpha}{3}+\varepsilon} (t-k)) P_{N_i} u_i ] \|_{X_{N_i}} \lesssim \| P_{N_i} u_i \|_{F_{N_i}}.
\end{equation*}
Let
\begin{equation*}
\begin{split}
f_{1,N_1,L_1} &= \mathcal{F}_{t,x}[\eta_0(N_+^{\frac{5-\alpha}{3}+\varepsilon} \gamma(N_{1,+}^{\frac{5-\alpha}{3}+\varepsilon} (t-k)) P_{N_1} u_1 ], \\
f_{2,N_2,L_2} &= \mathcal{F}_{t,x}[\gamma(N_{1,+}^{\frac{5-\alpha}{3}+\varepsilon} (t-k)) P_{N_2} u_2 ].
\end{split}
\end{equation*}
By the function space properties, we have
\begin{equation*}
N \sum_{L_i \geqslant N_1^{\frac{5-\alpha}{3}+\varepsilon}} L_i^{\frac{1}{2}} \| f_{i,N_i,L_i} \|_{L^2_{\tau,\xi,\eta}} \lesssim \| \mathcal{F}_{t,x}[ \gamma(N_{1,+}^{\frac{5-\alpha}{3}+\varepsilon} t-k) P_{N_i} u_i] \|_{X_{N_i}}.
\end{equation*}
Consequently, it suffices to prove dyadic estimates
\begin{equation*}
\begin{split}
&\quad N \Big(\frac{N_{1,+}}  {N_+}\Big)^{\frac{5-\alpha}{3}+\varepsilon} \sum_{L \geqslant N^{\frac{5-\alpha}{3}+\varepsilon}} L^{-\frac{1}{2}} \| 1_{D_{N,L}} (f_{1,N_1,L_1} * f_{2,N_2,L_2} ) \|_{L^2_{\tau,\xi,\eta}} \\
&\lesssim C(N,N_1) \prod_{i=1}^2 L_i^{\frac{1}{2}} \| f_{i,N_i,L_i} \|_{L^2_{\tau,\xi,\eta}}
\end{split}
\end{equation*}
for $L_{i} \gtrsim N_1^{\frac{5-\alpha}{3}+\varepsilon}$.

\medskip

\textbf{(1) Very small frequencies} ($N \lesssim 1 \lesssim N_1\sim N_2$): We distinguish the resonant and the non-resonant cases:\\
\noindent (i) $L_{\max} \ll NN_1^{\alpha}$: We carry out a dyadic decomposition in the output frequency $N$ and add time localization $N_1^{\frac{5-\alpha}{3}+\varepsilon}$. Using the nonlinear Loomis-Whitney inequality \eqref{eq:LoomisWhitneyCylinder}, we obtain
\begin{equation*}
\begin{split}
&\quad N L^{-\frac{1}{2}} N_1^{\frac{5-\alpha}{3}+\varepsilon} \| 1_{D_{N,L}} (f_{1,N_1,L_1} * f_{2,N_2,L_2}) \|_{L^2_{\tau,\xi,\eta}} \\
&\lesssim N N_1^{\frac{5-\alpha}{6}} N_1^{-\frac{\alpha}{2}} \big( \frac{N_1}{N} \big)^{\frac{1}{2}} \prod_{i=1}^2 L_i^{\frac{1}{2}} \| f_{i,N_i,L_i} \|_{L^2} \\
&\lesssim N^{\frac{1}{2}} N_1^{\frac{4}{3}-\frac{2 \alpha}{3}+\varepsilon} \prod_{i=1}^2 L_i^{\frac{1}{2}} \| f_{i,N_i,L_i} \|_{L^2}.
\end{split}
\end{equation*}
This is acceptable provided that an upper bound on $\varepsilon$, 
\begin{equation}
\label{eq:UpperBoundEpsI}
0< \varepsilon < \frac{4 \alpha}{6} - \frac{4}{3},
\end{equation}
holds true.

\noindent (ii) $L_{\max} \gtrsim NN_1^{\alpha}$: Depending on $L_{\max}$, the following sub-cases arise:

$\bullet ~L_{\max}=L_1$ (or $L_2$): Employing duality and Lemma \ref{lem:AlternativeBilinearStrichartzEstimate} like in \eqref{eq:HighLowHighEstimateAuxI} gives
\begin{equation*}
\begin{split}
&\quad N L^{-\frac{1}{2}} N_1^{\frac{5-\alpha}{3}+\varepsilon} \| 1_{D_{N,L}} (f_{1,N_1,L_1} * f_{2,N_2,L_2}) \|_{L^2_{\tau,\xi,\eta}} \\
&\lesssim N N_1^{\frac{5-\alpha}{3}+\varepsilon} (N_1^\alpha N)^{-\frac{1}{2}+\delta} N^{\frac{3}{4}} N_1^{- \frac{5-\alpha}{12} - \frac{\varepsilon}{2}} L_{\max}^{-\delta} \prod_{i=1}^2 L_i^{\frac{1}{2}} \| f_{i,N_i,L_i} \|_{L^2} \\
&\lesssim N^{\frac{1}{4}+\delta} N_1^{\frac{5-\alpha}{4}-\frac{\varepsilon}{2}-\frac{\alpha}{2}+\alpha \delta} L_{\max}^{-\delta} \prod_{i=1}^2 L_i^{\frac{1}{2}} \| f_{i,N_i,L_i} \|_{L^2_{\tau,\xi,\eta}},
\end{split}
\end{equation*}
which is likewise acceptable.

\medskip
$\bullet ~L_{\max}=L$: In this case we apply Lemma \ref{lem:AlternativeBilinearStrichartzEstimate} to $f_{1,N_1,L_1} * f_{2,N_2,L_2}$ and find
\begin{equation*}
    \begin{split}
        &\quad N N_1^{\frac{5-\alpha}{3}+\varepsilon} L^{-\frac{1}{2}} \| 1_{D_{N,L}}(f_{1,N_1,L_1} * f_{2,N_2,L_2}) \|_{L^2_{\tau,\xi,\eta}} \\
        &\lesssim N N_1^{\frac{5-\alpha}{3}+\varepsilon} (N_1^\alpha N)^{-\frac{1}{2}} N^{\frac{1}{2}} N_1^{\frac{1}{4}} N_1^{- \frac{5-\alpha}{12} - \frac{\varepsilon}{4}} \prod_{i=1}^2 L_i^{\frac{1}{2}} \| f_{i,N_i,L_i} \|_{L^2_{\tau,\xi,\eta}} \\
        &\lesssim N N_1^{\frac{3}{2}-\frac{3 \alpha}{4}+ \frac{3 \varepsilon}{4}} \prod_{i=1}^2 L_i^{\frac{1}{2}} \| f_{i,N_i,L_i} \|_{L^2_{\tau,\xi,\eta}}.
    \end{split}
\end{equation*}

This is acceptable provided that
\begin{equation}
\label{eq:UpperBoundEpsII}
    0 < \varepsilon < \alpha - 2
\end{equation}
holds true. Note that this is less restrictive than \eqref{eq:UpperBoundEpsI}.

\medskip

In the following cases, we assume that $N \gtrsim 1$.\\
\textbf{(2) Resonant case} ($L_{\max} \ll N_1^\alpha N$ and $1 \lesssim N \ll N_1 \sim N_2$): We obtain from applying the nonlinear Loomis--Whitney inequality \eqref{eq:LoomisWhitneyCylinder} and taking into account the additional time localization:
\begin{equation}
\label{eq:ResonantHighHigh}
\begin{split}
&\quad N L^{-\frac{1}{2}} \Big(\frac{N_1}{ N}\Big)^{\frac{5-\alpha}{3}+\varepsilon} \| 1_{D_{N,L}} (f_{1,N_1,L_1} * f_{2,N_2,L_2}) \|_{L^2_{\tau,\xi,\eta}} \\
&\lesssim N \Big(\frac{N_1}{ N}\Big)^{\frac{5-\alpha}{3}+\varepsilon} N_1^{-\frac{5-\alpha}{6}-\frac{\varepsilon}{2}} N_1^{-\alpha/2} \big( \frac{N_1}{N} \big)^{\frac{1}{2}} \prod_{i=1}^2 L_i^{\frac{1}{2}} \| f_{i,N_i,L_i} \|_{L^2_{\tau,\xi,\eta}} \\
&\lesssim N_1^{\frac{4}{3}-\frac{4 \alpha}{6}+\varepsilon} N^{\frac{1}{2}-\frac{5-\alpha}{3}-\varepsilon} \prod_{i=1}^2 L_i^{\frac{1}{2}} \| f_{i,N_i,L_i} \|_{L^2_{\tau,\xi,\eta}}.
\end{split}
\end{equation}

Estimate \eqref{eq:ShorttimeEstimateI} follows from dyadic summation. It is \eqref{eq:ResonantHighHigh} which suggests to carry out \eqref{eq:AsymmetricShorttimeBilinearEstimate} in $\bar{\mathcal{N}}^{-\frac{1}{2}}$ because
\begin{equation*}
\frac{4}{3} - \frac{2 \alpha}{3}+ \frac{1}{2} = s_1(\alpha,0).
\end{equation*}


\medskip

\noindent \textbf{(3) Non-resonant case} ($L_{\max} \gtrsim N_1^\alpha N$): 
Since $N \gtrsim 1$, applying Lemma \ref{lem:TrilinearConvolutionEstimateCylinder} gives
\begin{equation*}
\begin{split}
&\quad N L^{-\frac{1}{2}} \Big(\frac{N_1}{N}\Big)^{\frac{5-\alpha}{3}+\varepsilon} \| 1_{D_{N,L}} (f_{1,N_1,L_1} * f_{2,N_2,L_2}) \|_{L^2_{\tau,\xi,\eta}} \\
&\lesssim N^{\frac{1}{2}-\frac{5-\alpha}{3}-\varepsilon+\delta} N_1^{\frac{5-\alpha}{3}-\frac{1}{3}-\frac{\alpha}{3}+\frac{\varepsilon}{2}} L_{\max}^{-\delta} \prod_{i=1}^2 L_i^{\frac{1}{2}} \| f_{i,N_i,L_i} \|_{L^2},
\end{split}
\end{equation*}
which is acceptable.


\medskip
This completes the proof of Proposition \ref{prop:ShorttimeBilinearEstimateGeneral}. $\hfill \Box$


\subsection{Energy estimates}
\label{subsection:EnergyEstimates}

We carry out energy estimates in short-time Fourier restriction norms.

\subsubsection{Energy estimates for solutions}
In this subsection we show energy estimates for the solution as stated in \eqref{eq:EnergyEstimateSolutionsCylinder}.

%
%
%
\begin{proof}[Proof~of~Proposition~\ref{prop:EnergyEstimateSolutionsCylinder}]
Let $u \in C_T E^\infty$. By the fundamental theorem of calculus we have for any $N \in 2^{\Z}$:
\begin{equation*}
\| P_N u(t) \|_{L^2}^2 = \| P_N u(0) \|_{L^2}^2 + \int_0^t \int_{\R \times \T} P_N u \, \partial_x P_N (u^2) dx dy ds.
\end{equation*}
Note that by the definition of the function spaces $B^s$, the low frequencies satisfy
\begin{equation*}
\| P_{\ll 1} u \|_{B^s(T)}^2 = \| P_{\ll 1} u(0) \|_{L^2}^2.
\end{equation*}

\medskip

We carry out the same reductions to take advantage of the derivative nonlinearity like in the previous sections.
It suffices to estimate
\begin{equation*}
N_2 \int_0^t \int_{\R \times \T} P_{N_1} u P_{N_2} u P_{N_3} u \, dx dy ds \quad (N_2 \lesssim N_1 \sim N_3, \; N_1 \gtrsim 1).
\end{equation*}

We need to add time localization to estimate the expression in the correct function spaces. This amounts to a factor of $N_1^{\frac{5-\alpha}{3}+\varepsilon(\alpha)}$. After applying Plancherel's theorem we reduce to dyadic estimates:
\begin{equation*}
N_2 N_1^{\frac{5-\alpha}{3}+\varepsilon} \big| \int \big( f_{1,N_1,L_1} * f_{2,N_2,L_2} \big) f_{3,N_3,L_3} d\xi d\eta d\tau \big|
\end{equation*}
with $L_i \gtrsim N_1^{\frac{5-\alpha}{3}+\varepsilon(\alpha)}$, which is attributed to the corresponding time localization.

\medskip

$\bullet$ High $\times$ Low $\to$ High interaction: In this case, we have $N_1 \sim N \gtrsim N_2$, $N_1 \gtrsim 1$. Applying Lemma \ref{lem:TrilinearConvolutionEstimateCylinder}
we find
\begin{equation*}
\begin{split}
&\quad N_2 N_1^{\frac{5-\alpha}{3}+\varepsilon} \big| \int \big( f_{1,N_1,L_1} * f_{2,N_2,L_2} \big) f_{3,N_3,L_3} d\xi d\eta d\tau \big| \\
&\lesssim N_2^{\frac{1}{2}+\delta} N_1^{\frac{4}{3}-\frac{2\alpha}{3}+\frac{3\varepsilon}{4}+\delta} L_{\max}^{-\delta} \prod_{i=1}^3 L_i^{\frac{1}{2}} \| f_{i,N_i,L_i} \|_2.
\end{split}
\end{equation*}
This is acceptable for $s \geqslant s_1(\alpha,\varepsilon,\R \times \T)$ choosing $\varepsilon = \varepsilon(\alpha)$ small enough.

\medskip

$\bullet$ High$\times$ High $\to$ Low interaction: In this case, we have $1 \lesssim N \sim N_2 \ll N_1 \sim N_3$. We check that the estimate in the above display suffices again. The proof is complete.

\end{proof}

\subsubsection{Energy estimates for differences of solutions}

We estimate the solutions to the difference equation. At a crucial step in the analysis, we cannot integrate by parts which motivates us to estimate the difference of solutions at negative Sobolev regularity.

\begin{proof}[Proof~of~Proposition~\ref{prop:EnergyEstimatesDifference}]
Like above, we use the fundamental theorem of calculus to find
\begin{equation*}
\| P_N v(t) \|_{L^2}^2 = \| P_N v(0) \|_{L^2}^2 + \int_0^t \int_{\R \times \T} P_N v \partial_x P_N (v(u_1+u_2)) dx dt.
\end{equation*}
Again, by definition of function spaces, it suffices to obtain estimates for $N \gtrsim 1$.

First, we shall show the estimate for $i=1,2$ with $s \geqslant s_1(\alpha,\varepsilon,\R \times \T)$:
\begin{equation}
\label{eq:AuxEnergyEstimateDifferenceCylinder}
\sum_{N \gtrsim 1} N^{-1} \big| \int_0^T \int_{\R \times \T} P_N v \partial_x P_N (v \, u_i) dx dt \big| \lesssim T^{\kappa} \| v \|_{F^{-\frac{1}{2}}(T)}^2 \| u_i \|_{F^s(T)}.
\end{equation}
To ease notation, let $u := u_i$.

\medskip

After dyadic frequency localization and using commutator arguments like in the proof of Proposition \ref{prop:EnergyEstimateSolutionsCylinder}, we see that we need to estimate three expressions. The first one is given by:
\begin{equation*}
N_2 N_1^{-1} \int_0^T \int_{\R \times \T} P_{N_1} v P_{N_2} u P_{N_3} v dx dt, \quad N_1 \sim N_3 \gtrsim N_2,
\end{equation*}
where $u$ denotes a solution to \eqref{eq:FKPI}. Here we have integrated by parts to shift the derivative to the low frequency.

 This expression can be handled like in the energy estimates for solutions and we obtain estimates
\begin{equation*}
\sum_{N_2 \lesssim N_1 \sim N_3} N_2 N_1^{-1} \big| \int_0^T \int_{\R \times \T} P_{N_1} v P_{N_2} u P_{N_3} v dx dt \big| \lesssim T^{\kappa} \| v \|_{\bar{F}^{-\frac{1}{2}}(T)}^2 \| u \|_{F^{s}(T)}
\end{equation*}
for $s \geqslant s_1(\alpha,\varepsilon,\R \times \T)$. 

\medskip

The second expression is given by the High $\times$ High $\to$ Low interaction: For $N_1 \sim N_3 \gg N_2 \gtrsim 1$, we require to estimate
\begin{equation*}
N_2^{-1} \int_0^T \int_{\R \times \T} P_{N_2} v \partial_x (P_{N_1} u P_{N_3} v) dx dy dt.
\end{equation*}
Again the derivative acts on the low frequency, but this time the high frequency terms include a solution ($P_{N_1}u$) and a difference of solutions ($P_{N_3}v$).

Taking the derivative into account and adding time localization we need to establish estimates for
\begin{equation*}
N_1^{\frac{5-\alpha}{3}+\varepsilon} \int (f_{1,N_1,L_1} * f_{2,N_2,L_2} ) f_{3,N_3,L_3} d \xi d\eta d\tau.
\end{equation*}

We apply the estimate from Lemma \ref{lem:TrilinearConvolutionEstimateCylinder} to find:
\begin{equation*}
\begin{split}
&\quad N_1^{\frac{5-\alpha}{3}+\varepsilon} \int (f_{1,N_1,L_1} * f_{2,N_2,L_2} ) f_{3,N_3,L_3} d \xi d\eta d\tau \\
&\lesssim N_2^{-\frac{1}{2}+\delta} N_1^{\frac{4-2\alpha}{3}+\frac{3 \varepsilon}{4}+\delta} L_{\max}^{-\delta} \prod_{i=1}^3 L_i^{\frac{1}{2}} \| f_{i,N_i,L_i} \|_2.
\end{split}
\end{equation*}
This estimate is summable for $s \geqslant s_1(\alpha,\varepsilon,\R \times \T)$.

\medskip

\medskip

The third expression is given by
\begin{equation*}
N_1 N_1^{-1} \int_0^T \int_{\R \times \T} P_{N_1} v P_{N_2} v P_{N_3} u dx dt, \quad N_1 \sim N_3 \gtrsim N_2.
\end{equation*}
In this case, we cannot use integration by parts to place the derivative on the low frequency. We observe that the estimate at negative regularity comes to the rescue, compensating for the derivative loss. Moreover, the case of low frequencies $N_1 \lesssim 1$ is easy to handle because it is not necessary to add time localization and the derivative exhibits a smoothing effect. 

After frequency-dependent time localization and applying Plancherel's theorem, with the above notations it suffices to obtain summable estimates for
\begin{equation*}
N_1^{\frac{5-\alpha}{3}+\varepsilon} \int (f_{1,N_1,L_1} * f_{2,N_2,L_2}) f_{3,N_3,L_3} d\xi d\eta d\tau.
\end{equation*}
An application of Lemma \ref{lem:TrilinearConvolutionEstimateCylinder} gives
\begin{equation*}
\begin{split}
&\quad N_1^{\frac{5-\alpha}{3}+\varepsilon} \int (f_{1,N_1,L_1} * f_{2,N_2,L_2}) f_{3,N_3,L_3} d\xi d\eta d\tau \\
&\lesssim N_2^{-\frac{1}{2}+\delta} N_1^{\frac{4}{3} - \frac{2 \alpha}{3} + \frac{3 \varepsilon}{4}+\delta} L_{\max}^{-\delta} \prod_{i=1}^3 L_i^{\frac{1}{2}} \| f_{i,N_i,L_i} \|_2.
\end{split}
\end{equation*}
Summation yields again \eqref{eq:AuxEnergyEstimateDifferenceCylinder}.

\medskip

We turn to the proof of \eqref{eq:DifferenceEnergyEstimateII}:
\begin{equation*}
\| v \|^2_{E^s(T)} \lesssim \| v(0) \|^2_{H^{s,0}} + T^{\kappa} \| v \|^3_{F^{s,0}(T)} + T^\kappa \| v \|_{\bar{F}^{-\frac{1}{2}}(T)} \| v \|_{F^s(T)} \| u_2 \|_{F^{2s+\frac{1}{2}}(T)}.
\end{equation*}
We invoke again the fundamental theorem of calculus to find for $N \gg 1$
\begin{equation*}
\begin{split}
&\quad \| P_N v(t) \|_{L^2}^2 \\
 &= \|P_N v(0) \|^2_{L^2} + \int_0^t \int_{\R^2} P_N v \partial_x (P_N(v(u_1+u_2))) dx ds \\
&= \|P_N v(0) \|^2_{L^2} + \int_0^t \int_{\R^2} P_N v \partial_x (P_N(v^2)) dx ds + \int_0^t \int_{\R^2} P_N v \partial_x (P_N(v u_2)) dx ds.
\end{split}
\end{equation*}
For symmetry reasons the first integral can be handled like in the estimate for solutions \eqref{eq:EnergyEstimateSolutionsCylinder}. 

For the second term we take once more a paraproduct decomposition:
\begin{equation}
\label{eq:ParaproductDifferencesCylinder}
P_N(v u_2) = P_N( v P_{\ll N} u_2)+ P_N(P_{\ll N} v \cdot u_2) + 
 P_N (P_{\gtrsim N} v P_{\gtrsim N} u_2).
\end{equation}
The first term can be handled by integration by parts and the arguments from the proof of \eqref{eq:EnergyEstimateSolutionsCylinder} gives
\begin{equation*}
\sum_{N \geqslant 1} N^{2s} \big| \int_0^t \int_{\R^2} P_N v \partial_x (P_N(v P_{\ll N} u_2)) dx dy ds \big| \lesssim T^\kappa \| v \|^2_{F^s(T)} \| u_2 \|_{F^s(T)}.
\end{equation*}
Moreover, the third term can be estimated like in \eqref{eq:EnergyEstimateSolutionsCylinder} because
the derivative acts on the lowest frequency. This yields again
\begin{equation*}
\sum_{N \gtrsim 1} N^{2s} \big| \int_0^t \int_{\R \times \T} \partial_x P_N v P_N (P_{\gtrsim N} v P_{\gtrsim N} u_2 )) dx dy dt \big| \lesssim T^{\kappa} \| v \|_{F^s(T)}^2 \| u_2 \|_{F^s(T)}. 
\end{equation*}

It remains to estimate the second term in \eqref{eq:ParaproductDifferencesCylinder}. Here we need to consider
\begin{equation*}
N \big| \int_0^t \int_{\R \times \T} P_{N_1} v P_{N_2} v P_{N_3} u dx dy ds \big|, \quad N_2 \ll N_1 \sim N_3 \sim N
\end{equation*}
in a summable manner. Time localization amounts to a factor $N^{\frac{5-\alpha}{3}+\varepsilon}$ and after applying Plancherel's theorem and modulation localization we need to estimate
\begin{equation*}
N_1^{\frac{5-\alpha}{3}+1+\varepsilon} \int (f_{1,N_1,L_1} * f_{2,N_2,L_2} ) f_{3,N_3,L_3} d\xi d\eta d\tau
\end{equation*}
for $L_i \gtrsim N_1^{\frac{5-\alpha}{3}+\varepsilon}$. Applying Lemma \ref{lem:TrilinearConvolutionEstimateCylinder} gives
\begin{equation*}
\begin{split}
&\quad N_1^{\frac{5-\alpha}{3}+1+\varepsilon} \int (f_{1,N_1,L_1} * f_{2,N_2,L_2} ) f_{3,N_3,L_3} d\xi d\eta d\tau \\
&\lesssim N_1^{\frac{5-\alpha}{3}+1+\varepsilon} N_2^{-\frac{1}{2}+\delta} N_1^{-\frac{1}{3}-\frac{\alpha}{3}-\frac{\varepsilon}{2}+\delta} L_{\max}^{-\delta} \prod_{i=1}^3 L_i^{\frac{1}{2}} \| f_{i,N_i,L_i} \|_2 \\
&\lesssim N_2^{-\frac{1}{2}} N_1^{\frac{1}{2}+s_1(\alpha,\frac{3\varepsilon}{4}+2\delta,\R \times \T)} L_{\max}^{-\delta} \prod_{i=1}^3 L_i^{\frac{1}{2}} \| f_{i,N_i,L_i} \|_2.
\end{split}
\end{equation*}

In conclusion, choosing $\varepsilon(\alpha)$ small enough, we can cover any regularity $s \geqslant s_1$ in \eqref{eq:DifferenceEnergyEstimate}. The proof is complete.

\end{proof}

%

\section{Semilinear well-posedness}
\label{section:SemilinearWellposedness}

In this section, we show the sharp semilinear local well-posedness for KP-I equations posed on the Euclidean plane and on the cylinder. The results are proved by invoking the contraction mapping principle in Fourier restriction spaces.

\subsection{$\R^2$ case} 
\label{subsection:R2SemilinearLWP}
We state the following theorem which is an improvement of \cite[Theorem 6.1]{SanwalSchippa2023} and proves well-posedness in the full subcritical range.

Let $s,b \in \R$, recall $\omega_\alpha(\xi,\eta) = \xi |\xi|^\alpha + \frac{\eta^2}{\xi}$, and we define the space $X^{s,b}$ as closure of Schwartz functions with respect to the norm:
\begin{equation*}
    \| u \|_{X^{s,b}} = \| \langle \xi \rangle^s \langle \tau - \omega_\alpha(\xi,\eta) \rangle^b \mathcal{F}_{t,x,y}(u) \|_{L^2_{\tau,\xi,\eta}}.
\end{equation*}
We define for measurable $u:[0,T] \times \R^2 \to \C$:
\begin{equation*}
    \| u \|_{X^{s,b}_T} = \inf_{\substack{\tilde{u} \in X^{s,b}, \\ u = \tilde{u} \vert_{[0,T]} } } \| \tilde{u} \|_{X^{s,b}}.
\end{equation*}

\begin{theorem}
\label{thm:SemilinearWPR2}
Let $\alpha \geqslant \frac{5}{2}$ and $s>1-\frac{3\alpha}{4}$. Then, there is some $b>\frac{1}{2}$ such that for $T=T(\|u_0\|_{H^{s,0}})$, \eqref{eq:FKPIR2} is analytically locally well-posed in $H^{s,0}$ with the solution lying in $X^{s,b}_T \hookrightarrow C([0,T];H^{s,0})$.
\end{theorem}

We shall be brief here and refer the reader to \cite[Section 6]{SanwalSchippa2023} for the properties of the auxiliary spaces $X^{s,b}_T$. The proof of the theorem is concluded in Subsection \ref{subsection:ConclusionSemilinearLWP}. The following estimate is crucial:
\begin{proposition}
\label{prop:BilinearEstimateRR}
	Let $\alpha\geqslant \frac{5}{2}$ and $s>1-\frac{3\alpha}{4}$. Then  there is some $b>\frac{1}{2}$ such that the following estimate is true:
	\begin{equation}
		\label{eq:XsbBilinearEstimate}
		\| \partial_x(uv)\|_{X^{s,b-1}} \lesssim \|u\|_{X^{s,b}} \|v\|_{X^{s,b}}.
	\end{equation}
\begin{proof}
	By duality and Plancherel's theorem, we can reduce the above to proving
	\begin{equation}
		\int_{\R^3} \xi~ \widehat{uv}\cdot \overline{\widehat{w}} d\tau d\xi d\eta \lesssim \|u\|_{X^{s,b}} \|v\|_{X^{s,b}} \|w\|_{X^{-s,1-b}}.
	\end{equation}
After a dyadic decomposition, for $N_i \in 2^{\Z}, L_i\in 2^{\N_0}$, we prove the following estimate
\begin{equation}
\label{eq:DyadicLocalizedFourierRestrictionEstimateR2}
\begin{split}
	\Big| \int_{\R^3}(f_{N_1,L_1}\ast g_{N_2,L_2})h_{N,L}d\tau d\xi d\eta\Big| &\lesssim C(N_1,N_2,N) L_1^{\frac{1}{2}} L_2^{\frac{1}{2}} L^{\frac{1}{2}-} \\
&\quad \|f_{N_1,L_1}\|_{L^2} \|g_{N_2,L_2}\|_{L^2} \|h_{N,L} \|_{L^2}
\end{split}
\end{equation}
for a suitable summability constant $C(N_1,N_2,N)$ which also incorporates the derivative loss from the nonlinearity. To avoid writing the integral on the left-hand side in the above display repetitively, we denote:
\begin{equation*}
	I:= 	\Big| \int_{\R^3}(f_{N_1,L_1}\ast g_{N_2,L_2})h_{N,L}d\tau d\xi d\eta\Big|. 
\end{equation*}
In the following we do a case-by-case analysis depending on the size of the $x$-frequencies.\\
\textbf{(i) High $\times$ Low $\rightarrow$ High} ($N_2 \lesssim N_1\sim N$): We consider two more cases:\begin{itemize}
	\item \underline{$L_{\max}\ll N_1^{\alpha}N_2$}: In case $N_1^{-\frac{\alpha}{2}+\frac{1}{2}} L_3^{\varepsilon} \lesssim N_2 \lesssim 1$, using the nonlinear Loomis--Whitney estimate \eqref{eq:LoomisWhitneyEuclidean}, we obtain
\begin{equation*}
	\begin{split}
		\sum_{N_1^{-\frac{\alpha}{2}+\frac{1}{2}} L_3^{\varepsilon} \lesssim N_2 \lesssim 1} I &\lesssim \sum_{N_1^{-\frac{\alpha}{2}+\frac{1}{2}} L_3^{\varepsilon} \lesssim N_2 \lesssim 1} N_1^{-\frac{3\alpha}{4} +\frac{1}{2}} N_2^{-\frac{1}{2}} (L_1 L_2 L )^{\frac{1}{2}} \\
		&\quad \quad \times \|f_{N_1,L_1}\|_{L^2} \|g_{N_2,L_2}\|_{L^2} \|h_{N,L}\|_{L^2}\\
		& \lesssim N_1^{-\frac{\alpha}{2}+\frac{1}{4}} (L_1L_2)^{\frac{1}{2}} L_3^{\frac{1}{2}-\frac{\varepsilon}{2}} \|f_{N_1,L_1}\|_{L^2} \|g_{\lesssim 1,L_2}\|_{L^2} \|h_{N,L}\|_{L^2}.
	\end{split}
\end{equation*}
In the other case where $N_2 \lesssim N_1^{-\frac{\alpha}{2} +\frac{1}{2}} L^{\varepsilon}$, using the bilinear Strichartz estimate \eqref{eq:BilinearStrichartzEstimate}, we have
\begin{equation*}
	\begin{split}
\sum_{ N_2 \lesssim N_1^{-\frac{\alpha}{2}+\frac{1}{2}} L^{\varepsilon}} I	&\lesssim 	\sum_{ N_2 \lesssim N_1^{-\frac{\alpha}{2}+\frac{1}{2}} L_3^{\varepsilon}} N_1^{-\frac{\alpha}{4}} N_2^{\frac{1}{2}} (L_1L_2)^{\frac{1}{2}} \|f_{N_1,L_1}\|_{L^2} \|g_{N_2,L_2}\|_{L^2} \|h_{N,L}\|_{L^2}\\
&\lesssim N_1^{-\frac{\alpha}{2}+\frac{1}{4}} (L_1L_2)^{\frac{1}{2}} L^{\frac{\varepsilon}{2}} \|f_{N_1,L_1}\|_{L^2} \|g_{N_2,L_2}\|_{L^2} \|h_{N,L}\|_{L^2}.
	\end{split}
\end{equation*}
In both cases, after considering the derivative loss and summing up, the estimate leads to \eqref{eq:DyadicLocalizedFourierRestrictionEstimateR2}. In the case $N_2 \gtrsim 1$, we use the nonlinear Loomis--Whitney estimate \eqref{eq:LoomisWhitneyEuclidean} to obtain
\begin{equation*}
	\begin{split}
		\sum_{1\lesssim N_2\lesssim N_1} I &\lesssim \sum_{1 \lesssim N_2\lesssim N_1} N_1^{-\frac{1}{2}-\frac{3\alpha}{4}}N_2^{-\frac{1}{2}} (L_1L_2L)^{\frac{1}{2}} \|f_{N_1,L_1}\|_{L^2} \|g_{N_2,L_2}\|_{L^2}\|h_{N,L}\|_{L^2}\\
		&\lesssim N_1^{1-\frac{3\alpha}{4}-s+(\alpha+1)\varepsilon} \sum_{1\lesssim N_2 \lesssim N_1} \Big(\frac{N_2}{N_1}\Big)^{-\frac{1}{2}-s+\varepsilon} (L_1L_2)^{\frac{1}{2}}L^{\frac{1}{2}-\varepsilon} N_2^s N^{-1} \\
		&\quad \quad \|f_{N_1,L_1}\|_{L^2} \|g_{N_2,L_2}\|_{L^2}\|h_{N,L}\|_{L^2}.
	\end{split}
\end{equation*}
\item \underline{$L_{\max}\gtrsim  N_1^{\alpha}N_2$}: Using \cite[Lemma 4.6]{SanwalSchippa2023}, we have 
\begin{equation*}
	\begin{split}
		 I &\lesssim \frac{(L_1L_2L)^{\frac{1}{2}}}{L_{\max}^{\frac{1}{4}}} N_1^{-\frac{\alpha}{2}} N_2^{\frac{1}{4}} \|f_{N_1,L_1}\|_{L^2} \|g_{N_2,L_2}\|_{L^2}\|h_{N,L}\|_{L^2}\\
		&\lesssim (L_1L_2)^{\frac{1}{2}} L_{\max}^{\frac{1}{2}-\varepsilon} N_1^{-\frac{3\alpha}{4}+\alpha \varepsilon} N_2^{\varepsilon} \|f_{N_1,L_1}\|_{L^2} \|g_{N_2,L_2}\|_{L^2}\|h_{N,L}\|_{L^2}\\
		&\lesssim N_1^{-1} (L_1L_2)^{\frac{1}{2}} L_{\max}^{\frac{1}{2}-\varepsilon} N_1^{1-\frac{3\alpha}{4}+\alpha\varepsilon} N_2^{-1+\frac{3\alpha}{4} -\alpha\varepsilon} N_2^{1-\frac{3\alpha}{4}+\alpha\varepsilon+\varepsilon} \\
		&\quad \quad \times \|f_{N_1,L_1}\|_{L^2} \|g_{N_2,L_2}\|_{L^2}\|h_{N,L}\|_{L^2}\\
		&\lesssim N_1^{-1} (L_1L_2)^{\frac{1}{2}} L_{\max}^{\frac{1}{2}-\varepsilon} \Big(\frac{N_1}{N_2}\Big)^{1-\frac{3\alpha}{4}+\alpha \varepsilon} N_2^{1-\frac{3\alpha}{4} +(\alpha+1)\varepsilon} \\
		&\quad \quad \times \|f_{N_1,L_1}\|_{L^2} \|g_{N_2,L_2}\|_{L^2}\|h_{N,L}\|_{L^2},
	\end{split}
\end{equation*}
which is summable since $\alpha \geqslant \frac{5}{2}$.
\end{itemize}

\noindent \textbf{(ii) High $\times$ High $\rightarrow$ Low} ($N \lesssim N_1 \sim N_2$): The two cases are
\begin{itemize}
	\item \underline{$L_{\max}\ll N_1^{\alpha}N_2$}: In case $N\gtrsim 1$, using the nonlinear Loomis--Whitney estimate \eqref{eq:LoomisWhitneyEuclidean}, we obtain
	\begin{equation*}
		\begin{split}
			\sum_{N\ll N_1\sim N_2} I &\lesssim  \sum_{N \ll N_1\sim N_2} N^{-\frac{1}{2}+\varepsilon} N_1^{\frac{1}{2}-\frac{3\alpha}{4}+\alpha\varepsilon} N N_1^{-2s} N^s  (L_1L_2)^{\frac{1}{2}} L^{\frac{1}{2}-\varepsilon}\\
   &\qquad N^{-1}N_1^s N_2^s N^{-s} \|f_{N_1,L_1}\|_{L^2} \|g_{N_2,L_2}\|_{L^2}\|h_{N,L}\|_{L^2}\\
			&=\sum_{N\ll N_1\sim N_2} N^{\frac{1}{2}+s+\epsilon} N_1^{\frac{1}{2}-\frac{3\alpha}{4}-2s+\alpha\varepsilon} N^{-1}N_1^s N_2^s N^{-s}\\
   &\qquad \|f_{N_1,L_1}\|_{L^2} \|g_{N_2,L_2}\|_{L^2}\|h_{N,L}\|_{L^2}\\
			&\lesssim \sum_{N_1\sim N_2 \gtrsim 1} N_1^{1-\frac{3\alpha}{4}-s+\alpha\varepsilon +\varepsilon} N^{-1}N_1^s N_2^s N^{-s} \|f_{N_1,L_1}\|_{L^2} \|g_{N_2,L_2}\|_{L^2}\|h_{N,L}\|_{L^2}.
		\end{split}
	\end{equation*}
The above can be summed up for $s>1-\frac{3\alpha}{4}$.
\item \underline{$L_{\max}\gtrsim N_1^{\alpha}N_2$}: This case can be handled in the same way as the analogous subcase in High $\times$ Low $\rightarrow$ High interaction.
\end{itemize}
\noindent \textbf{(iii) Very low frequencies} ($N , N_1 , N_2 \lesssim 1$): Using the bilinear Strichartz estimate from Lemma \ref{lem:AlternativeBilinearStrichartzEstimate}, we have
\begin{equation*}
	\begin{split}
	\sum_{N, N_1,  N_2 \lesssim 1} N \cdot I &\lesssim \sum_{N\sim N_1\sim N_2 \lesssim 1} \|f_{N_1,L_1} * g_{N_2,L_2} \|_{L^2} N \|h_{N,L}\|_{L^2}\\
	&\lesssim \sum_{N, N_1,  N_2 \lesssim 1} (L_1 L_2)^{\frac{1}{2}} (N_1 \wedge N_2)^{\frac{1}{2}} (N_1 \vee N_2)^{\frac{1}{4}} L^{\frac{1}{2}-}\\ &\qquad \times  \|f_{N_1,L_1}\|_{L^2} \|g_{N_2,L_2}\|_{L^2} N \|h_{N,L}\|_{L^2}\\
\end{split}
\end{equation*}
with straight-forward summation.
\end{proof}
\end{proposition}

\subsection{$\R \times \T$ case} We prove the following result:
\begin{theorem}
\label{thm:SemilinearWPRT}
Let $\alpha \geqslant5$ and $s>\frac{1}{4}-\frac{\alpha}{4}$. Then, there is some $b>\frac{1}{2}$ such that for $T=T(\|u_0\|_{H^{s,0}})$, \eqref{eq:FKPI} is analytically locally well-posed in $H^{s,0}$ with the solution lying in $X^{s,b}_T \hookrightarrow C([0,T];H^{s,0})$.
\end{theorem}

As in the $\R^2$ case, the claim is implied by the bilinear estimate:
\begin{proposition}
\label{prop:BilinearEstimateRT}
    Let $\alpha\geqslant 5$ and $s>\frac{1}{4}-\frac{\alpha}{4}$. Then  there is some $b>\frac{1}{2}$ such that the following estimate is true:
	\begin{equation}
		\label{eq:XsbBilinearEstimateRT}
		\| \partial_x(uv)\|_{X^{s,b-1}} \lesssim \|u\|_{X^{s,b}} \|v\|_{X^{s,b}}.
	\end{equation}
 \end{proposition}

 We prove Proposition \ref{prop:BilinearEstimateRT} in a series of lemmata. After dyadic decomposition in the $x$-frequencies, we aim to prove an estimate similar to \eqref{eq:DyadicLocalizedFourierRestrictionEstimateR2} with a suitable, but different summability constant, say $\tilde{C}(N,N_1,N_2)$.
We first observe that for $L_{med} \gtrsim N_{\max}^{\frac{\alpha}{2}}$, the nonlinear Loomis--Whitney estimate and the bilinear Strichartz estimates are the same as in the $\R^2$ case. The boundary cases $N_{\max} \lesssim 1$ and $N_{\min} \lesssim N_{\max}^{-\frac{\alpha}{2}}$ are treated in Lemma \ref{lem:SemilinearCylinderBoundaryCases}.

\smallskip

Hence, it remains to consider the following cases:
\begin{itemize}
	\item $L_{\max} \lesssim N_{\max}^{\frac{\alpha}{2}}$,
	\item $L_{med} \lesssim N_{\max}^{\frac{\alpha}{2}} \lesssim L_{\max} \lesssim N_{\max}^\alpha N_{\min}$,
 \item $L_{\max} \gtrsim N_{\max}^{\alpha}N_{\min}$.
\end{itemize}
The first two cases are sub-cases of the resonant case, while the last case is the non-resonant case. In the following, we assume that
\begin{equation}
\label{eq:MainFrequencyAssumptions}
N_{\max} \gtrsim 1, \text{ and } N_{\min} \gtrsim N_{\max}^{-\frac{\alpha}{2}}.
\end{equation}

\smallskip

We handle the first case in the following:
\begin{lemma}
    Let $\alpha \geqslant 5$, $s>\frac{1-\alpha}{4}$ and $N, N_i \in 2^{\Z}$, $L_i \in 2^{\mathbb{N}_0}$, and suppose that \eqref{eq:MainFrequencyAssumptions}. Let $f_{N_1,L_1}, g_{N_2,L_2}, h_{N,L}:\R \times \R \times \Z\to \R_{+}$ and $supp(f_{N_1,L_1}) \subseteq D_{N_1,L_1}$, $supp(g_{N_2,L_2}) \subseteq D_{N_2,L_2}$, and $supp(h_{N,L}) \subseteq D_{N,L}$ such that $L_{\max} \lesssim \max(N, N_1,N_2)^{\frac{\alpha}{2}}$. Then, the estimate \eqref{eq:XsbBilinearEstimateRT} holds.
    \begin{proof}
    To prove the result, we consider the following cases:\\
\noindent  \textbf{(i) Low $\times$ High $\rightarrow$ High} ($N_2 \lesssim N_1\sim N$): If the size of the low frequency, viz. $N_2\lesssim 1$, we consider two cases:
In case $N_1^{-\frac{\alpha}{2}+\frac{1}{2}} L_{\max}^{\frac{1}{2}} \lesssim N_2 \lesssim 1$, we use the Loomis--Whitney estimate \eqref{eq:LoomisWhitneyCylinder}, (note: $L_{\max} \ll N_1^\alpha N_2$) to obtain
\begin{equation*}
	\begin{split}
\sum_{N_1^{-\frac{\alpha}{2}+\frac{1}{2}}L_{\max}^{\frac{1}{2}} \lesssim N_2 \lesssim 1} I &\lesssim \sum_{N_1^{-\frac{\alpha}{2}+\frac{1}{2}} L_{\max}^{\frac{1}{2}} \lesssim N_2 \lesssim 1} N_1^{\frac{1}{2}-\frac{\alpha}{2}} N_2^{-\frac{1}{2}} (L_{\max}L_{\min})^{\frac{1}{2}} \\
&\quad\|f_{N_1,L_1}\|_{L^2} \|g_{N_2,L_2}\|_{L^2}\|h_{N,L}\|_{L^2}\\
&\lesssim N_1^{\frac{5}{4}-\frac{\alpha}{4}}N_1^{-1} L_{\max}^{\frac{1}{4}} L_{\min}^{\frac{1}{2}} \|f_{N_1,L_1}\|_{L^2} \|g_{N_2,L_2}\|_{L^2}\|h_{N,L}\|_{L^2}.
 \end{split}
\end{equation*}
Summing up the above in the spatial frequencies gives the required estimate. In case $N_2 \lesssim N_1^{-\frac{\alpha}{2}+\frac{1}{2}}L_{\max}^{\frac{1}{2}}$, we use the bilinear Strichartz estimate \eqref{eq:BilinearStrichartzEstimate} to obtain
\begin{equation*}
	\begin{split}
		\sum_{N_2 \lesssim N_1^{-\frac{\alpha}{2}+\frac{1}{2}}L_{\max}^{\frac{1}{2}}} I &\lesssim \sum_{N_2 \lesssim N_1^{-\frac{\alpha}{2}+\frac{1}{2}}L_{\max}^{\frac{1}{2}}} N_2^{\frac{1}{2}} L_{\min}^{\frac{1}{2}} \|f_{N_1,L_1}\|_{L^2} \|g_{N_2,L_2}\|_{L^2}\|h_{N,L}\|_{L^2}\\
		&\lesssim  N_1^{\frac{5}{4}-\frac{\alpha}{4}} N_1^{-1} L_{\min}^{\frac{1}{2}} L_{\max}^{\frac{1}{4}} \|f_{N_1,L_1}\|_{L^2} \|g_{N_2,L_2}\|_{L^2}\|h_{N,L}\|_{L^2}.
	\end{split}
\end{equation*}
In the case $N_2\gtrsim 1$, the nonlinear Loomis--Whitney estimate \eqref{eq:LoomisWhitneyCylinder} gives
\begin{equation*}
	\begin{split}
		\sum_{1 \lesssim N_2\lesssim N_1} I &\lesssim \sum_{1 \lesssim N_2 \lesssim N_1} N_2^{-\frac{1}{2}} N_1^{-\frac{\alpha}{2}+\frac{1}{2}} (L_{\min}L_{\max})^{\frac{1}{2}} \|f_{N_1,L_1}\|_{L^2} \|g_{N_2,L_2}\|_{L^2}\|h_{N,L}\|_{L^2}\\
		&\lesssim \sum_{1\lesssim N_2 \lesssim N_1} N_2^{-\frac{1}{2}-s} N_1^{\frac{3}{2}-\frac{\alpha}{2}+\varepsilon'} L_{\min}^{\frac{1}{2}} L_{\max}^{\frac{1}{2}-\varepsilon} \|f_{N_1,L_1}\|_{L^2} \|h_{N,L}\|_{L^2} N_2^s \|g_{N_2,L_2}\|_{L^2}.
	\end{split}
\end{equation*}
We observe that in case $s>-\frac{1}{2}$, it is straightforward to sum the above expression in $N_1$ and $N_2$. In case $-\frac{1}{2}-s \geq 0$, we obtain for the above expression that it is bounded by
\begin{equation*}
	\sum_{1 \lesssim N_2 \lesssim N_1} \Big(\frac{N_2}{N_1}\Big)^{-\frac{1}{2}-s} N_1^{1-\frac{\alpha}{2}-s+\varepsilon} N^{-1} L_{\min}^{\frac{1}{2}} L_{\max}^{\frac{1}{2}-\varepsilon'} \|f_{N_1,L_1}\|_{L^2} \|h_{N,L}\|_{L^2}\sup_{N_2\gtrsim 1} N_2^s\|g_{N_2,L_2}\|_{L^2}.
\end{equation*}
It is easy to observe that the above is summable for $s>1-\frac{\alpha}{2}$.\\

\noindent \textbf{(ii) High $\times$ High $\rightarrow$ Low} ($N_1 \sim N_2 \gtrsim N$): We first consider the case $N\lesssim 1$ where the derivative in the nonlinearity is smoothing. Using the nonlinear Loomis--Whitney estimate \eqref{eq:LoomisWhitneyCylinder}, we have
\begin{equation*}
	\begin{split}
\sum_{N\lesssim 1\lesssim N_1\sim N_2} I &\lesssim \sum_{N\lesssim 1\lesssim N_1\sim N_2}N^{-\frac{1}{2}} N_1^{\frac{1}{2}-\frac{\alpha}{2}} (L_{\min}L_{\max})^{\frac{1}{2}}\|f_{N_1,L_1}\|_{L^2} \|h_{N,L}\|_{L^2}\|g_{N_2,L_2}\|_{L^2}\\
&\lesssim \sum_{N\lesssim 1\lesssim N_1\sim N_2} N^{\frac{1}{2}} N_1^{\frac{1}{2}-\frac{\alpha}{2}-2s+\varepsilon} L_{\min}^{\frac{1}{2}} L_{\max}^{\frac{1}{2}-\varepsilon'} N^{-1}N_1^s \|f_{N_1,L_1}\|_{L^2} \\
&\quad \quad \times N_2^s\|g_{N_2,L_2}\|_{L^2} \|h_{N,L}\|_{L^2}.
\end{split}
\end{equation*}
We observe that the above is summable for $s>\frac{1}{4}-\frac{\alpha}{2}$.

For $N \gtrsim 1$, the same estimate yields
\begin{equation*}
\begin{split}
	\sum_{1 \lesssim N \lesssim N_1\sim N_2}I &\lesssim  \sum_{1 \lesssim N \lesssim N_1\sim N_2} N^{\frac{1}{2}+s} N_1^{\frac{1}{2}-\frac{\alpha}{2}-2s+\varepsilon'}L_{\min}^{\frac{1}{2}} L_{\max}^{\frac{1}{2}-\varepsilon'}N^{-1} N_1^s\|f_{N_1,L_1}\|_{L^2}  \\
 &\quad \qquad \qquad \times N_2^s\|g_{N_2,L_2}\|_{L^2} N^{-s}\|h_{N,L}\|_{L^2}.
 \end{split}
\end{equation*}
The above is summable for $s>\frac{1}{4}-\frac{\alpha}{4}$. We observe that this is the case which determines the regularity threshold for local well-posedness.\\
    \end{proof}
\end{lemma}

Next, we have the following result to deal with the intermediate case $L_{med} \lesssim N_{\max}^{\frac{\alpha}{2}} \lesssim L_{\max}$:
\begin{lemma}
    Let $\alpha \geqslant 5$, $s>\frac{1-\alpha}{4}$, and $N, N_i \in 2^{\Z}$, $L_i,L \in 2^{\N}$, and suppose that \eqref{eq:MainFrequencyAssumptions}. Let $f_{N_1,L_1}, g_{N_2,L_2}, h_{N,L}:\R \times \R \times \Z\to \R_{+}$ and $supp(f_{N_1,L_1}) \subseteq D_{N_1,L_1}$, $\text{supp}(g_{N_2,L_2}) \subseteq D_{N_2,L_2}$, $\text{supp}(h_{N,L}) \subseteq D_{N,L}$ such that $L_{med} \lesssim N_{\max}^{\frac{\alpha}{2}} \lesssim L_{\max}$. Then, the estimate \eqref{eq:XsbBilinearEstimateRT} holds. 
\begin{proof}
 Note that in this case the Loomis--Whitney estimate remains the same, but the bilinear Strichartz estimate gains $N_1^{-\frac{\alpha}{4}}$ at the cost of $L_{\max}^{\frac{1}{2}}$.

\smallskip
 
\textbf{(i) Low $\times$ High $\rightarrow$ High} ($N_2 \lesssim N_1\sim N$): In the case $N_2 \lesssim 1$, we consider two subcases. For $N_1^{\frac{1}{2}-\frac{\alpha}{4}} \lesssim N_2$, we use the Loomis-Whitney estimate \eqref{eq:LoomisWhitneyCylinder} to obtain
\small
\begin{equation*}
	\begin{split}
	\sum_{N_1^{\frac{1}{2}-\frac{\alpha}{4}} \lesssim N_2}I &\lesssim  \sum_{N_1^{\frac{1}{2}-\frac{\alpha}{4}} \lesssim N_2} N^{\frac{3}{2}-\frac{\alpha}{2}+\varepsilon} N_2^{-\frac{1}{2}} (L_1L_2)^{\frac{1}{2}}L^{\frac{1}{2}-\varepsilon'}N^{-1} \|f_{N_1,L_1}\|_{L^2} \|g_{N_2,L_2}\|_{L^2}\|h_{N,L}\|_{L^2}\\
	&\lesssim N_1^{\frac{5}{4}-\frac{3\alpha}{8}+\varepsilon} \sum_{N_1^{\frac{1}{2}-\frac{\alpha}{4}} \lesssim N_2}  (L_1L_2)^{\frac{1}{2}}L^{\frac{1}{2}-\varepsilon'}N^{-1} \|f_{N_1,L_1}\|_{L^2} \|g_{N_2,L_2}\|_{L^2}\|h_{N,L}\|_{L^2}.
	\end{split}
\end{equation*}
\normalsize
In the complementary case viz. $N_2\lesssim N_1^{\frac{1}{2}-\frac{\alpha}{4}}$, using the bilinear Strichartz estimate \eqref{eq:BilinearStrichartzEstimate}, we have
\begin{equation*}
	\begin{split}
		\sum_{N_2\lesssim N_1^{\frac{1}{2}-\frac{\alpha}{4}}} I &\lesssim \sum_{N_2\lesssim N_1^{\frac{1}{2}-\frac{\alpha}{4}}} N_2^{\frac{1}{2}} N_1^{-\frac{\alpha}{4}} (L_{\min}L_{\max})^{\frac{1}{2}} \|f_{N_1,L_1}\|_{L^2} \|g_{N_2,L_2}\|_{L^2}\|h_{N,L}\|_{L^2}\\
		&\lesssim N_1^{\frac{5}{4}-\frac{3\alpha}{8}+\varepsilon}\sum_{N_2\lesssim N_1^{\frac{1}{2}-\frac{\alpha}{4}}} (L_1L_2)^{\frac{1}{2}}L^{\frac{1}{2}-\varepsilon'}N^{-1} \|f_{N_1,L_1}\|_{L^2} \|g_{N_2,L_2}\|_{L^2}\|h_{N,L}\|_{L^2}.
	\end{split}
\end{equation*}
For $N_2\gtrsim 1$, the same estimate as the case $L_{\max} \lesssim N_1^{\frac{\alpha}{2}}$ suffices.\\

\noindent \textbf{(ii) High $\times$ High $\rightarrow$ Low} ($N_1\sim N_2 \gtrsim N$):
This case can be dealt with in the same way as the corresponding case in $L_{\max} \lesssim N_1^{\frac{\alpha}{2}}$ since we employ the Loomis--Whitney estimate which is the same for both the cases.\\   
\end{proof}
    \end{lemma}

     The case $L_{\max}\gtrsim N_{\max}^{\alpha}N_{\min}$ is dealt with in the following:

    \begin{lemma}
    Let $\alpha \geqslant 5$, $s>\frac{1-\alpha}{4}$ and $N,N_i \in 2^{\Z}$, $L_i,L \in 2^{\N_0}$, and suppose that \eqref{eq:MainFrequencyAssumptions}. Let $f_{N_1,L_1}, g_{N_2,L_2}, h_{N,L}:\R \times \R \times \Z\to \R_{+}$ and $supp(f_{N_1,L_1}) \subseteq D_{N_1,L_1}$, $\text{supp}(g_{N_2,L_2}) \subseteq D_{N_2,L_2}$, $\text{supp}(h_{N,L}) \subseteq D_{N,L}$ such that $L_{\max}\gtrsim N_{\max}^{\alpha}N_{\min}$. Then, the estimate \eqref{eq:XsbBilinearEstimateRT} holds.
    \begin{proof}
    \textbf{(i) Low $\times$ High $\rightarrow$ High} ($N_2 \lesssim N_1\sim N$): In case $L_{\max} = L_2$, using Lemma \ref{lem:AlternativeBilinearStrichartzEstimate}, we have
\begin{equation*}
\begin{split}
    I &\lesssim \|f_{N_1,L_1} \ast h_{N,L}\|_{L^2} \|g_{N_2,L_2} \|_{L^2}\\
    &\lesssim N_2^{\frac{1}{2}} (L \wedge L_1)^{\frac{1}{2}} \langle (L \vee L_1) N_1\rangle^{\frac{1}{4}} \|f_{N_1,L_1}\|_{L^2} \|g_{N_2,L_2}\|_{L^2} \|h_{N,L}\|_{L^2}\\
    &\lesssim N_1^{-\frac{\alpha}{2}+\frac{5}{4}} (L\wedge L_1)^{\frac{1}{2}} (L \vee L_1)^{\frac{1}{4}}L_2^{\frac{1}{2}} N^{-1}\|f_{N_1,L_1}\|_{L^2} \|g_{N_2,L_2}\|_{L^2} \|h_{N,L}\|_{L^2}.
\end{split}
\end{equation*}
If $N_2 \lesssim 1$, the above is summable for any $s$. For $N_2 \gtrsim 1$, the above expression is dominated by
\begin{equation*}
N_1^{-\frac{\alpha}{2} +\frac{5}{4}} N_2^{-s}(L\wedge L_1)^{\frac{1}{2}} (L \vee L_1)^{\frac{1}{4}}L_2^{\frac{1}{2}} N^{-1}\|f_{N_1,L_1}\|_{L^2} N_2^s\|g_{N_2,L_2}\|_{L^2} \|h_{N,L}\|_{L^2}.
\end{equation*}
If $s\geqslant 0$, the expression is summable in the spatial frequencies. For $s<0$, the above is dominated by
\begin{equation*}
\Big(\frac{N_1}{N_2}\Big)^{-\frac{\alpha}{2}+\frac{5}{4}}N_2^{-s-\frac{\alpha}{2}+\frac{5}{4}} (L\wedge L_1)^{\frac{1}{2}} (L \vee L_1)^{\frac{1}{4}}L_2^{\frac{1}{2}} N^{-1}\|f_{N_1,L_1}\|_{L^2} N_2^s\|g_{N_2,L_2}\|_{L^2} \|h_{N,L}\|_{L^2}
\end{equation*}
which is summable for $s>\frac{5-\alpha}{4}$.\\
If $L_{\max} = L$, we use Lemma \ref{lem:AlternativeBilinearStrichartzEstimate} as follows:
\begin{equation*}
    \begin{split}
    I &\lesssim \|f_{N_1,L_1}\ast g_{N_2,L_2}\|_{L^2} \|h_{N,L}\|_{L^2}\\
&\lesssim N_2^{\frac{1}{2}} (L_1 \wedge L_2)^{\frac{1}{2}} \langle (L_1 \vee L_2)N_2 \rangle^{\frac{1}{4}} \|f_{N_1,L_1}\|_{L^2} \|g_{N_2,L_2}\|_{L^2} \|h_{N,L}\|_{L^2}\\
&\lesssim N_2^{\frac{1}{4}+\varepsilon} N_1^{1-\frac{\alpha}{2}+\varepsilon} (L_1 \wedge L_2)^{\frac{1}{2}} (L_1 \vee L_2)^{\frac{1}{4}} L^{\frac{1}{2}-\varepsilon'} N^{-1}\|f_{N_1,L_1}\|_{L^2} \|g_{N_2,L_2}\|_{L^2} \|h_{N,L}\|_{L^2}
    \end{split}
\end{equation*}
For $N_2\lesssim 1$, summability follows since $\alpha \geqslant 5$. For $N_2 \gtrsim 1$, the above expression is bounded by
\begin{equation*}
    \begin{split}
 N_2^{\frac{1}{4}-s+\varepsilon} N_1^{1-\frac{\alpha}{2}+\varepsilon} (L_1 \wedge L_2)^{\frac{1}{2}} (L_1 \vee L_2)^{\frac{1}{4}} L^{\frac{1}{2}-\varepsilon'} N^{-1}\|f_{N_1,L_1}\|_{L^2}N_2^s \|g_{N_2,L_2}\|_{L^2} \|h_{N,L}\|_{L^2}       
    \end{split}
\end{equation*}
which can be summed up for $s>\frac{5}{4}-\frac{\alpha}{2}$.\\

\noindent \textbf{(ii) High $\times$ High $\rightarrow$ Low} ($N\lesssim N_1\sim N_2$): We first consider the case $L_{\max} = L$. If $N\lesssim 1$, the derivative in the nonlinearity is smoothing and using the bilinear Strichartz estimate from Lemma \ref{lem:AlternativeBilinearStrichartzEstimate} for the high frequencies, we obtain
\begin{equation*}
	\begin{split}
		I &\lesssim \| f_{1,N_1,L_1} * g_{2,N_2,L_2} \|_{L^2} \|h_{N,L}\|_{L^2} \\
		&\lesssim (N_1^\alpha N)^{-\frac{1}{2}+\varepsilon} N^{\frac{1}{2}} N_1^{\frac{1}{4}} (L_1 L_2)^{\frac{1}{2}} \|f_{N_1,L_1}\|_{L^2} \|g_{N_2,L_2}\|_{L^2} L^{\frac{1}{2}-\varepsilon} \|h_{N,L}\|_{L^2}\\
		&\lesssim N_1^{\frac{1}{4}- \frac{\alpha}{2}+\alpha \varepsilon - 2s} (L_1L_2)^{\frac{1}{2}} L^{\frac{1}{2}-\varepsilon} N_1^s \|f_{N_1,L_1}\|_{L^2} N_2^s \|g_{N_2,L_2}\|_{L^2} N^{\varepsilon} \|h_{N,L}\|_{L^2}.
	\end{split}
\end{equation*}
The above is summable for $s>\frac{1}{8}-\frac{\alpha}{4}$.\\
For $N\gtrsim 1$, the above estimate is still sufficient and is summable for $s>\frac{1}{8}-\frac{\alpha}{4}$.

\smallskip

If $L_{\max} = L_1$, we have using Lemma \ref{lem:AlternativeBilinearStrichartzEstimate},
\begin{equation*}
\begin{split}
I &\lesssim \|g_{N_2,L_2}\ast h_{N,L}\|_{L^2} \|f_{N_1,L_1}\\
&\lesssim N^{\frac{1}{2}} (L_2 \wedge L)^{\frac{1}{2}} \langle (L_2 \vee L) N\rangle^{\frac{1}{4}} \|f_{N_1,L_1}\|_{L^2}  \|g_{N_2,L_2}\|_{L^2} \|h_{N,L}\|_{L^2}\\
&\lesssim NN_1^{-\frac{\alpha}{2}} (L_2 \wedge L)^{\frac{1}{2}} \langle (L_2 \vee L) N\rangle^{\frac{1}{4}}L_1^{\frac{1}{2}} N^{-1}\|f_{N_1,L_1}\|_{L^2}  \|g_{N_2,L_2}\|_{L^2} \|h_{N,L}\|_{L^2}\\
&\lesssim  NN_1^{-\frac{\alpha}{2}-2s} (L_2 \wedge L)^{\frac{1}{2}} \langle (L_2 \vee L) N\rangle^{\frac{1}{4}}L_1^{\frac{1}{2}} N^{-1} N_1^s\|f_{N_1,L_1}\|_{L^2}  N_2^s \|g_{N_2,L_2}\|_{L^2} \|h_{N,L}\|_{L^2}.
\end{split}
\end{equation*}
If $N \lesssim 1$, we require $s>-\frac{\alpha}{4}$  for summability in spatial frequencies. If $N \gtrsim 1$, the following bound
\begin{equation*}
\begin{split}
   &\lesssim  N^{\frac{5}{4}+s} N_1^{-\frac{\alpha}{2}-2s} (L_2 \wedge L)^{\frac{1}{2}} \langle (L_2 \vee L) N\rangle^{\frac{1}{4}}L_1^{\frac{1}{2}} N^{-1} N_1^s\|f_{N_1,L_1}\|_{L^2}  N_2^s \|g_{N_2,L_2}\|_{L^2} \\
   &\quad \quad \times N^{-s}\|h_{N,L}\|_{L^2}
   \end{split}
\end{equation*}
implies that we require $s>\frac{-\alpha}{4}$ to be able to sum up the above estimate.
    \end{proof}
    \end{lemma}

    The very small frequencies can be handled as follows:
    \begin{lemma}
    \label{lem:SemilinearCylinderBoundaryCases}
     Let $\alpha \geqslant 5$, $s>\frac{1-\alpha}{4}$ and $N, N_i \in 2^{\Z}$ such that $N_{\max} \lesssim 1$ or $N_{\min} \lesssim N_{\max}^{-\frac{\alpha}{2}}$. Let $f_{N_1,L_1}, g_{N_2,L_2}, h_{N,L}:\R \times \R \times \Z\to \R_{+}$ and $supp(f_{N_1,L_1}) \subseteq D_{N_1,L_1}$, $\text{supp}(g_{N_2,L_2}) \subseteq D_{N_2,L_2}$, $\text{supp}(h_{N,L}) \subseteq D_{N,L}$. Then, the estimate \eqref{eq:XsbBilinearEstimateRT} holds.
        \begin{proof}
          In this case, we do not distinguish between the resonant and the non-resonant case. Using Lemma \ref{lem:AlternativeBilinearStrichartzEstimate}, we have
\begin{equation*}
\begin{split}
    I &\lesssim \|f_{N_1,L_1} \ast g_{N_2,L_2}\|_{L^2}\|h_{N,L}\|_{L^2}\\
    &\lesssim N_{\min}^{\frac{1}{2}} (L_1\wedge L_2)^{\frac{1}{2}} \langle (L_1 \vee L_2) N_{\max} \rangle^{\frac{1}{4}} \|f_{N_1,L_1}\|_{L^2} \|g_{N_2,L_2}\|_{L^2}\|h_{N,L}\|_{L^2},
    \end{split}
\end{equation*}
which is sufficient.\\  
        \end{proof}
    \end{lemma}

\subsection{Conclusion of the theorems on semilinear local well-posedness}
\label{subsection:ConclusionSemilinearLWP}

\begin{proof}[Proof of Theorem \ref{thm:SemilinearWPR2}]
The proof follows along the same lines as for \cite[Theorem~6.1] {SanwalSchippa2023}, and we shall be brief. We use $X^{s,b}$ spaces as the auxiliary spaces to run a fixed point argument for the operator given by:
\begin{equation*}
    \Gamma(u)(t):= \eta(t) S_{\alpha}(t)u_0 + \eta(t) \int_0^t S_{\alpha}(t-s)(u \partial_x u)(s)ds.
\end{equation*}
Here, $\eta$ is a smooth compactly supported time cut-off. Using the linear estimate,  the energy estimate for $X^{s,b}$ spaces \cite[Section 2.6]{Tao2006NonlinearDispersiveEquations} and Proposition \ref{prop:BilinearEstimateRR}, we obtain
\begin{equation*}
    \|\Gamma(u)\|_{X^{s,b}_1} \leqslant C (\|u(0)\|_{H^{s,0}} + \|u\|_{X^{s,b}_1}^2).
\end{equation*}
Similarly,
\begin{equation*}
    \| \Gamma(u_1) -\Gamma(u_2)\|_{X^{s,b}_1} \leqslant 2\tilde{C}(\|u_0\|_{H^{s,0}}) \|u_1-u_2\|_{X^{s,b}_1}.
\end{equation*}
For small initial data (attributed in the constants $C,\tilde{C}$),  we can prove local well-posedness for \eqref{eq:KPIR2} on $X^{s,b}_1$. Using scaling and subcriticality of the regularity, any large data can be scaled to be small and one obtains well-posedness on a time interval $[0,T]$, with $T$ depending on the size of the initial data.
\end{proof}

In preparation of the proof, we recall the following lemma to trade regularity in modulation to powers of time:
\begin{lemma}
\label{lem:TradingModulationRegularitySemilinear}
Let $\eta \in \R$, $-\frac{1}{2}<b' \leq b < \frac{1}{2}$, then for any $0<T<1$ and $s \in \mathbb{R}$, it holds
\begin{equation*}
    \| \eta(t/T) u \|_{X^{s,b'}} \lesssim_{\eta,b,b'} T^{b-b'} \| u \|_{X^{s,b}}.
\end{equation*}
\end{lemma}

\begin{proof}[Proof of Theorem \ref{thm:SemilinearWPRT}]
For small initial data and a fixed time interval $[0,1]$, the proof of local well-posedness is the same as in the $\R^2$ case. In the large data case, we use the leeway in the modulation regularity in the proof of Proposition \ref{prop:BilinearEstimateRT} to apply Lemma \ref{lem:TradingModulationRegularitySemilinear}. This yields a modified bilinear estimate
\begin{equation*}
\| \partial_x(u v) \|_{X^{s,b-1}_T} \lesssim T^{\delta} \| u \|_{X^{s,b}} \| v \|_{X^{s,b}}.    
\end{equation*}
Thus, we see that the time of the existence of the solution will depend on the size of the initial data and the well-posedness result follows similarly as in the $\R^2$ case.
\end{proof}

\section*{Acknowledgements}

R.S. conducted initial work on this project at the Korea Institute for Advanced Study, whose
financial support through the grant No. MG093901 is gratefully acknowledged. Moreover, R.S. would like to thank the Department of Mathematics at the Tokyo Institute of Technology for kind hospitality in December 2023.


\begin{thebibliography}{10}

\bibitem{AblowitzSegur1979}
Mark~J. Ablowitz and Harvey Segur.
\newblock On the evolution of packets of water waves.
\newblock {\em J. Fluid Mech.}, 92(4):691--715, 1979.

\bibitem{BejenaruHerrTataru2010}
Ioan Bejenaru, Sebastian Herr, and Daniel Tataru.
\newblock A convolution estimate for two-dimensional hypersurfaces.
\newblock {\em Rev. Mat. Iberoam.}, 26(2):707--728, 2010.

\bibitem{BejenaruTao2006}
Ioan Bejenaru and Terence Tao.
\newblock Sharp well-posedness and ill-posedness results for a quadratic
  non-linear {S}chr\"{o}dinger equation.
\newblock {\em J. Funct. Anal.}, 233(1):228--259, 2006.

\bibitem{BennettBez2010}
Jonathan Bennett and Neal Bez.
\newblock Some nonlinear {B}rascamp-{L}ieb inequalities and applications to
  harmonic analysis.
\newblock {\em J. Funct. Anal.}, 259(10):2520--2556, 2010.

\bibitem{BennettBez2021}
Jonathan Bennett and Neal Bez.
\newblock Higher order transversality in harmonic analysis.
\newblock In {\em Harmonic analysis and nonlinear partial differential
  equations}, RIMS K\^{o}ky\^{u}roku Bessatsu, B88, pages 75--103. Res. Inst.
  Math. Sci. (RIMS), Kyoto, 2021.

\bibitem{BennettCarberyWright2005}
Jonathan Bennett, Anthony Carbery, and James Wright.
\newblock A non-linear generalisation of the {L}oomis-{W}hitney inequality and
  applications.
\newblock {\em Math. Res. Lett.}, 12(4):443--457, 2005.

\bibitem{Bourgain1993B}
J.~Bourgain.
\newblock Fourier transform restriction phenomena for certain lattice subsets
  and applications to nonlinear evolution equations. {II}. {T}he
  {K}d{V}-equation.
\newblock {\em Geom. Funct. Anal.}, 3(3):209--262, 1993.

\bibitem{Bourgain1993KPII}
J.~Bourgain.
\newblock On the {C}auchy problem for the {K}adomtsev-{P}etviashvili equation.
\newblock {\em Geom. Funct. Anal.}, 3(4):315--341, 1993.

\bibitem{BourgainDemeter2015}
Jean Bourgain and Ciprian Demeter.
\newblock The proof of the {{\(l^2\)}} decoupling conjecture.
\newblock {\em Ann. Math. (2)}, 182(1):351--389, 2015.

\bibitem{Cordoba1979}
A.~C\'{o}rdoba.
\newblock A note on {B}ochner-{R}iesz operators.
\newblock {\em Duke Math. J.}, 46(3):505--511, 1979.

\bibitem{Dryuma1974}
V.~S. {Dryuma}.
\newblock {Analytic solution of the two-dimensional Korteweg-de Vries (KdV)
  equation}.
\newblock {\em Soviet Journal of Experimental and Theoretical Physics Letters},
  19:387, June 1974.

\bibitem{Fefferman1973}
Charles Fefferman.
\newblock A note on spherical summation multipliers.
\newblock {\em Israel J. Math.}, 15:44--52, 1973.

\bibitem{Guo2024}
Zihua Guo.
\newblock Remark on the low regularity well-posedness of the kp-i equation.
\newblock 2024.

\bibitem{GuoMolinet2024}
Zihua {Guo} and Luc {Molinet}.
\newblock {On the well-posedness of the KP-I equation}.
\newblock {\em arXiv e-prints}, page arXiv:2404.12364, April 2024.

\bibitem{GuoOh2018}
Zihua Guo and Tadahiro Oh.
\newblock Non-existence of solutions for the periodic cubic {NLS} below
  {$L^2$}.
\newblock {\em Int. Math. Res. Not. IMRN}, (6):1656--1729, 2018.

\bibitem{GuoPengWang2010}
Zihua Guo, Lizhong Peng, and Baoxiang Wang.
\newblock On the local regularity of the {KP}-{I} equation in anisotropic
  {S}obolev space.
\newblock {\em J. Math. Pures Appl. (9)}, 94(4):414--432, 2010.

\bibitem{Hadac2008}
Martin Hadac.
\newblock Well-posedness for the {Kadomtsev}-{Petviashvili} {II} equation and
  generalisations.
\newblock {\em Trans. Am. Math. Soc.}, 360(12):6555--6572, 2008.

\bibitem{HadacHerrKoch2009}
Martin Hadac, Sebastian Herr, and Herbert Koch.
\newblock Well-posedness and scattering for the {KP}-{II} equation in a
  critical space.
\newblock {\em Ann. Inst. Henri Poincar{\'e}, Anal. Non Lin{\'e}aire},
  26(3):917--941, 2009.

\bibitem{HerrSanwalSchippa2024}
Sebastian Herr, Akansha Sanwal, and Robert Schippa.
\newblock Low regularity well-posedness of {KP}-{I} equations: {T}he
  three-dimensional case.
\newblock {\em J. Funct. Anal.}, 286(4):Paper No. 110292, 2024.

\bibitem{HerrSchippaTzvetkov2024}
Sebastian {Herr}, Robert {Schippa}, and Nikolay {Tzvetkov}.
\newblock {The Cauchy problem for the periodic Kadomtsev--Petviashvili--II
  equation below $L^2$}.
\newblock {\em arXiv e-prints}, page arXiv:2407.12222, July 2024.

\bibitem{IonescuKenigTataru2008}
A.~D. Ionescu, C.~E. Kenig, and D.~Tataru.
\newblock Global well-posedness of the {KP}-{I} initial-value problem in the
  energy space.
\newblock {\em Invent. Math.}, 173(2):265--304, 2008.

\bibitem{Kadomtsev1970}
Boris~Borisovich Kadomtsev and Vladimir~I. Petviashvili.
\newblock On the stability of solitary waves in weakly dispersing media.
\newblock 1970.

\bibitem{KillipVisan2019}
Rowan Killip and Monica Vi{\c{s}}an.
\newblock {KdV} is well-posed in {{\(H^{-1}\)}}.
\newblock {\em Ann. Math. (2)}, 190(1):249--305, 2019.

\bibitem{KillipVisanZhang2018}
Rowan Killip, Monica Vi{\c{s}}an, and Xiaoyi Zhang.
\newblock Low regularity conservation laws for integrable {PDE}.
\newblock {\em Geom. Funct. Anal.}, 28(4):1062--1090, 2018.

\bibitem{KinoshitaSchippa2021}
Shinya Kinoshita and Robert Schippa.
\newblock Loomis-{W}hitney-type inequalities and low regularity well-posedness
  of the periodic {Z}akharov-{K}uznetsov equation.
\newblock {\em J. Funct. Anal.}, 280(6):Paper No. 108904, 53, 2021.

\bibitem{KochTzvetkov2003}
H.~Koch and N.~Tzvetkov.
\newblock On the local well-posedness of the {Benjamin}-{Ono} equation in
  {{\(H^S(\mathbb{R})\)}}.
\newblock {\em Int. Math. Res. Not.}, 2003(26):1449--1464, 2003.

\bibitem{KochSteinerberger2015}
Herbert Koch and Stefan Steinerberger.
\newblock Convolution estimates for singular measures and some global nonlinear
  {B}rascamp-{L}ieb inequalities.
\newblock {\em Proc. Roy. Soc. Edinburgh Sect. A}, 145(6):1223--1237, 2015.

\bibitem{KochTataru2018}
Herbert Koch and Daniel Tataru.
\newblock Conserved energies for the cubic nonlinear {Schr{\"o}dinger} equation
  in one dimension.
\newblock {\em Duke Math. J.}, 167(17):3207--3313, 2018.

\bibitem{MolinetSautTzvetkov2002}
L.~Molinet, J.~C. Saut, and N.~Tzvetkov.
\newblock Global well-posedness for the {KP}-{I} equation.
\newblock {\em Math. Ann.}, 324(2):255--275, 2002.

\bibitem{MolinetSautTzvetkov2002Illposedness}
L.~Molinet, J.-C. Saut, and N.~Tzvetkov.
\newblock Well-posedness and ill-posedness results for the
  {Kadomtsev}-{Petviashvili}-{I} equation.
\newblock {\em Duke Math. J.}, 115(2):353--384, 2002.

\bibitem{MolinetSautTzvetkov2007}
L.~Molinet, J.~C. Saut, and N.~Tzvetkov.
\newblock Global well-posedness for the {KP}-{I} equation on the background of
  a non-localized solution.
\newblock {\em Commun. Math. Phys.}, 272(3):775--810, 2007.

\bibitem{Robert2018}
Tristan Robert.
\newblock Global well-posedness of partially periodic {KP}-{I} equation in the
  energy space and application.
\newblock {\em Ann. Inst. H. Poincar\'{e} C Anal. Non Lin\'{e}aire},
  35(7):1773--1826, 2018.

\bibitem{SanwalSchippa2023}
Akansha Sanwal and Robert Schippa.
\newblock Low regularity well-posedness for {KP}-{I} equations: the
  dispersion-generalized case.
\newblock {\em Nonlinearity}, 36(8):4342--4383, 2023.

\bibitem{rsc2019}
Robert Schippa.
\newblock {\em Short-time {Fourier} transform restriction phenomena and
  applications to nonlinear dispersive equations}.
\newblock PhD thesis, Bielefeld University, 2019.

\bibitem{Tao2006NonlinearDispersiveEquations}
Terence Tao.
\newblock {\em Nonlinear dispersive equations: local and global analysis}.
\newblock Number 106. American Mathematical Soc., 2006.

\bibitem{Zhang2016}
Yu~Zhang.
\newblock Local well-posedness of {KP}-{I} initial value problem on torus in
  the {B}esov space.
\newblock {\em Comm. Partial Differential Equations}, 41(2):256--281, 2016.

\end{thebibliography}
\end{document}